\documentclass[11pt]{amsart}
\usepackage{amsmath,amssymb,amsthm}
\usepackage{tikz,xstring,ifthen}
\usepackage[numbers,comma,square,sort&compress]{natbib}
\usepackage{mathabx}   
\usepackage{stackengine}
\usepackage{caption}
\usepackage{subcaption} 
\RequirePackage{color}
\RequirePackage[colorlinks,urlcolor=my-blue,linkcolor=my-red,citecolor=my-green]{hyperref}

\numberwithin{figure}{section}

\usepackage{enumerate}

\usepackage{scalerel}    
\usepackage{stmaryrd}   
\usepackage{mathabx}   

\definecolor{my-blue}{rgb}{0.0,0.0,0.6}
\definecolor{my-red}{rgb}{0.5,0.0,0.0}
\definecolor{my-green}{rgb}{0.0,0.5,0.0}
\definecolor{nicos-red}{rgb}{0.75,0.0,0.0}

\newtheorem{theorem}{\sc Theorem}[section]
\newtheorem{lemma}[theorem]{\sc Lemma}

\numberwithin{equation}{section}
\theoremstyle{remark}

\newtheorem{definition}[theorem]{\sc Definition}

\newcommand{\be}{\begin{equation}}
\newcommand{\ee}{\end{equation}}
\newcommand{\beq}{\begin{equation}}
\newcommand{\eeq}{\end{equation}}

\newcommand{\nn}{\nonumber}
\providecommand{\abs}[1]{\vert#1\vert}

\newcommand{\fl}[1]{\lfloor{#1}\rfloor} 
\newcommand{\ce}[1]{\lceil{#1}\rceil}

  \def\cD{\mathcal{D}}
\def\cF{\mathcal{F}}

\def\cN{\mathcal{N}}

\def\cS{\mathcal{S}} \def\cT{\mathcal{T}}
\def\cI{\mathcal{I}}

\def\cA{\mathcal{A}}

\def\bE{\mathbb{E}} 

\def\bP{\mathbb{P}}

\def\bR{\mathbb{R}}
\def\bZ{\mathbb{Z}}

 \def\Z{\bZ}  \def\R{\bR}
  \def\uvec{\mathbf{u}}

\def\evec{\mathbf{e}}

\def\dvec{\mathbf{d}}

\def\Ivec{\mathbf{I}} \def\Jvec{\mathbf{J}}

\def\Yvec{\mathbf{Y}}

\def\w{\omega}

\def\e{\varepsilon}

\def\mP{\mathbf{P}}
\def\m1{\mathbf{1}}

 \def\cJ{\mathcal{J}}  
 \def\wt{\widetilde}  \def\wh{\widehat}  \def\wc{\widecheck}
 
\def\E{\bE}
\def\P{\bP} 

\def\trigamf{\psi_1} 
\def\funct lp{L} 
\def\funct lpbar{\bar L} 

\def\cA{\mathcal A}

\def\cE{{\mathcal E}}

\DeclareMathOperator{\ex}{ex}

\def\og{\preccurlyeq}

\DeclareMathOperator{\Vvar}{\mathbb{V}ar}

\DeclareMathOperator{\distance}{dist}

\newcommand{\rhodown}[1]{{#1}_{\scaleobj{1.2}\star}} 
\newcommand{\rhoup}[1]{{#1}^{\scaleobj{1.2}\star}} 

 \def\zevec{\mathbf{0}}  

\definecolor{darkgreen}{rgb}{0.0,0.5,0.0}
\definecolor{darkblue}{rgb}{0.0,0.0,0.3}
\definecolor{nicosred}{rgb}{0.65,0.1,0.1}
\definecolor{light-gray}{gray}{0.7}
\allowdisplaybreaks[1]

\usepackage{filemod}

\def\cif1{v}   

\def\deq{\overset{d}=}

\def\Yw{Y}  
   
\def\rim#1{\widehat{#1}} 
\def\ed{e}

\newcommand{\lzb}{\llbracket}   
\newcommand{\rzb}{\rrbracket}   

\newcommand\bbullet{{\raisebox{0.5pt}{\scaleobj{0.6}{\bullet}}}} 
\newcommand\brbullet{{\raisebox{-0.5pt}{\scaleobj{0.5}{\bullet}}}} 
\newcommand\cbullet{{\raisebox{1pt}{\scaleobj{0.6}{\bullet}}}}

\def\tsp{\hspace{0.55pt}}  
\def\tspa{\hspace{0.7pt}}  
\def\tspb{\hspace{0.9pt}}

\def\Dop{D}    
\def\Rop{R}  
\def\Sop{S}    

\usepackage[nomessages]{fp}

          \def\wgtd{Y}     
\def\fo{\phi}
 \def\wt{\widetilde}  \def\wh{\widehat} 
 \def\ex{{\tau}}
 
 \def\OAbP{(\Omega, \cA, \P)}  
\def\mwalk{M}      
\def\swalk{S}      

\def\pathsp{\mathbb{X}}   

\newcounter{usedm}
\newcounter{usedn}
\newcommand*{\lppwl}[6]{
	
	\FPeval{\m}{round(#5*(#3-#1)/(#3+#4-#1-#2),0)}
	
	\FPeval{\n}{round(#5*(#4-#2)/(#3+#4-#1-#2),0)}
	
	\FPeval{\stlength}{(#3+#4-#1-#2)/#5}
	
	\setcounter{usedm}{0}
	\setcounter{usedn}{0}
	
	\foreach\i in{1,...,{#5}}{
		\FPrandom{\x}
		
		\FPeval{\y}{(\m-\theusedm)/(\m+\n-\theusedm-\theusedn)}
		
		\FPeval{\startx}{#1+(\theusedm)*\stlength}
		\FPeval{\starty}{#2+(\theusedn)*\stlength}
		
		\ifthenelse{\lengthtest{\x pt > \y pt}}{
			\FPset{\endx}{\startx}
			\FPadd{\endy}{\starty}{\stlength}
			\stepcounter{usedn}
		}
		{
			\FPadd{\endx}{\startx}{\stlength}
			\FPset{\endy}{\starty}
			\stepcounter{usedm}
		}
		\draw[#6](\startx,\starty) -- (\endx,\endy);
	};
}

\begin{document}
	
	\title[Non-existence of bi-infinite polymers]
	{Non-existence of bi-infinite polymer Gibbs measures}
	
	\author[O.~Busani]{Ofer Busani}
	\address{Ofer Busani\\ University of Bristol\\  School of Mathematics\\ Fry Building\\ Woodland Rd.\\   Bristol BS8 1UG\\ UK.}
	\email{o.busani@bristol.ac.uk}
	\urladdr{https://people.maths.bris.ac.uk/~di18476/}
	\thanks{O. Busani was supported by EPSRC's EP/R021449/1 Standard Grant.} 
	
	\author[T.~Sepp\"al\"ainen]{Timo Sepp\"al\"ainen}
	\address{Timo Sepp\"al\"ainen\\ University of Wisconsin-Madison\\  Mathematics Department\\ Van Vleck Hall\\ 480 Lincoln Dr.\\   Madison WI 53706-1388\\ USA.}
	\email{seppalai@math.wisc.edu}
	\urladdr{http://www.math.wisc.edu/~seppalai}
	\thanks{T.\ Sepp\"al\"ainen was partially supported by  National Science Foundation grant DMS-1854619 and by the Wisconsin Alumni Research Foundation.}

	\date{\today}
	\begin{abstract} 
		We show that nontrivial bi-infinite polymer Gibbs measures do not exist in typical environments in the inverse-gamma (or log-gamma) directed polymer model on the planar square lattice. The precise technical result is that, except for measures supported on straight-line paths, such Gibbs measures do not exist in almost every environment when the weights are independent and identically distributed inverse-gamma random variables. The proof proceeds by showing that when two endpoints of a point-to-point polymer distribution are taken to infinity in opposite directions but not parallel to lattice directions, the midpoint of the polymer path escapes.  The proof is based on couplings, planar comparison arguments, and a recently discovered joint distribution of Busemann functions.  
	\end{abstract}
	\maketitle
	\tableofcontents
	
	\section{Introduction}  
	


\subsection{Directed polymers} 
The {\it directed polymer model} is a stochastic model of a random path that interacts with a random environment.  In its simplest formulation on an integer lattice $\Z^d$, positive random weights $\{Y_x\}_{x\in\Z^d}$ are assigned to the lattice vertices and the quenched probability of a finite lattice path $\pi$ is declared to be proportional to the product $\prod_{x\tspa\in\tspa\pi} Y_x$. 
In the usual Boltzmann-Gibbs formulation we take $Y_x=e^{-\beta \w_x}$ so that the energy of a path is proportional to the potential  $\sum_{x\tspa\in\tspa\pi} \w_x$ and the strength of the coupling between  the path $\pi$ and the environment $\w$ is modulated by the inverse temperature parameter $\beta$.  

 The directedness of the model means that some spatial  direction $\uvec\in\R^d$  represents time and the admissible paths $\pi$ are required to be $\uvec$-directed.   One typical  example  would be to  require that the steps of $\pi$ are of the form $(\pm\evec_i,1)\in\Z^d$ for $i\in\{1,\dotsc,d-1\}$.  In this example  the time direction is $\uvec=\evec_d$, space is the $(d-1)$-dimensional lattice $\Z^{d-1}$,  and $\pi$ is a  simple random walk path in space-time.  Another common choice is to restrict   the steps of $\pi$   to directed basis vectors $\{\evec_i\}_{1\le i\le d}$ so that time proceeds in the diagonal direction $\uvec=\evec_1+\dotsm+\evec_d$.  

This  model was introduced in the statistical 
physics literature by Huse and Henley in 1985 \cite{huse-henl} as a model of the domain wall in an Ising model with impurities. Since the polymer can be viewed as a perturbation of a simple random walk, a natural question to investigate is whether the walk is diffusive across large  scales.  The early  rigorous mathematical work  by Imbrie and Spencer
\cite{imbr-spen}  and Bolthausen  \cite{bolt-cmp-89} in the late 1980s established that in dimensions $d\ge 4$ (one time dimension plus at least three spatial dimensions) the path behaves diffusively for small enough $\beta$. This behavior is now known as {\it weak disorder}.   Later work \cite{come-varg-06, laco-10} established that in lower dimensions $d\in\{2,3\}$ or if $\beta$ is large enough,  the polymer model exhibits {\it strong disorder}, characterized by localization.  Excellent reviews of this development can be found in \cite{come-16, denholl-polymer}. 

Since the early interest in the phase transition between weak and strong disorder, the study of directed polymers has branched out in several directions.  The discovery of exactly solvable 1+1 dimensional models, the first of which were the O'Connell-Yor Brownian directed polymer \cite{oconn-yor-01} and the inverse-gamma, or log-gamma, polymer \cite{sepp-12-aop-corr}, led to rigorous proofs that directed polymers are members of the Kardar-Parisi-Zhang (KPZ) universality class \cite{boro-corw-ferr-14, boro-corw-reme,sepp-valk-10}.  This had been expected since directed polymers are positive temperature analogues of  directed last-passage percolation, for which predictions of KPZ universality were first  rigorously verified \cite{baik-deif-joha-99, joha}.  On KPZ we refer the reader to the recent reviews \cite{corw-16-rev, corw-18-ams, quas-14, quas-spoh-15}. 
 
 Through Feynman-Kac-type representations, directed polymers provide  solutions to stochastic partial differential equations.  Early work in this direction by Kifer \cite{kife-97} connected a polymer in the weak disorder regime with a stochastic Burgers equation.  The significant current  example of this, which also takes us back to the study of KPZ universality,  is the connection between the continuum directed random polymer and the stochastic heat equation with multiplicative noise, whose logarithm is  the Hopf-Cole solution of the KPZ equation. We refer to Corwin's review  \cite{corw-12-rev}.  
 
 \subsection{Infinite polymers} 
 Another natural direction of polymer research is the  limit as the path length is taken to infinity.  This limit can be readily  taken    in weak disorder. This can be found in the work of Comets and Yoshida \cite{come-yosh-aop-06}. 
  In strong disorder the existence of limiting infinite quenched polymer measures    was first proved in 1+1 dimensions for the inverse-gamma polymer in \cite{geor-rass-sepp-yilm-15}.

 The limiting quenched probability distributions on infinite-length polymer paths can be naturally described  as the Gibbs measures whose finite-dimensional conditional distributions are given by the quenched point-to-point polymer distributions $Q_{x,y}(\pi)=Z_{x,y}^{-1}  \prod_{x\tspa\in\tspa\pi} Y_x$.   Here $\pi$ is a path between points $x$ and $y$ and the partition function $Z_{x,y}=\sum_\pi \prod_{x\tspa\in\tspa\pi} Y_x$ normalizes $Q_{x,y}$ to be a probability distribution on the paths between $x$ and $y$. (This notion is developed precisely in Section \ref{sec:cif}.) 
 
 This Gibbsian point of view arose prominently  in the work of  Bakhtin and Li \cite{bakh-li-19} who studied   a 1+1 dimensional model with a Gaussian random walk.  They used polymer Gibbs measures  to construct global solutions to a stochastic Burgers equation on the line, subject to random kick forcing at discrete time intervals.  Their sequel  \cite{bakh-li-18}   showed that as the temperature is taken to zero, the Gibbs measures concentrate around the geodesic  of the corresponding directed  percolation model. 
 
   Janjigian and Rassoul-Agha \cite{janj-rass-20-aop}   developed aspects of a general theory of  polymer Gibbs measures  for i.i.d.\ vertex weights and directed  nearest-neighbor paths on the discrete planar square lattice $\Z^2$.  We work in their setting, with a specialized choice of weight distribution.

\subsection{Bi-infinite polymers} 

The work cited above  
 addressed  the existence and uniqueness of {\it semi-infinite}  Gibbs measures. 
These  are measures on semi-infinite, or one-sided infinite, paths, with fixed initial point.   The existence of {\it bi-infinite} Gibbs measures was left open.   These would be measures on bi-infinite paths  that satisfy the Gibbs property. 

Bi-infinite polymer Gibbs measures would be special cases of the general theory of Gibbs measures as developed in Georgii's monograph \cite{geor}. The polymer specification is a Markovian one because the distribution $Q_{x,y}$ on paths from $x$ to $y$ depends only on the boundary points $x$ and $y$.   However, this specification is not shift-invariant and hence the general theory of Chapters 10-11 of \cite{geor} is not helpful here.

In this paper we assume that the i.i.d.\ vertex weights $\{Y_x\}_{x\tsp\in\tsp\Z^2}$  on the planar lattice $\Z^2$  have inverse-gamma distribution. Then we prove that, for almost every choice of weights,  nontrivial bi-infinite Gibbs measures do not exist. Trivial  bi-infinite Gibbs measures do exist, by which we mean ones that are  supported on bi-infinite  straight lines.    

The key tools of the nonexistence proof are the following. 
\begin{enumerate}  [{(i)}] \itemsep=3pt
\item Planar comparison inequalities, reviewed and proved in Appendix \ref{app:genpol}. 
\item 
KPZ wandering exponent  $2/3$ of the polymer path, quoted  in Appendix \ref{sec:kpz5} from \cite{sepp-12-aop-corr}. 
\item A jointly stationary bivariate  inverse-gamma polymer from the  forthcoming work \cite{fan-sepp-20+} of the second author and W.~L.~Fan, developed in full detail in Appendix \ref{sec:stat-pol}. 
\end{enumerate} 
	From these  ingredients and coupling arguments we derive a bound on the speed of decay of the probability that a polymer path from far away in the  southwest to far away in the northeast goes through the origin. This bound is given in Theorem \ref{thm:ub} at the end of Section \ref{sec:estim}. The KPZ fluctuation bounds on polymer paths enable us to deduce this result   from local point-to-point estimates and a coarse-graining step. 
	
	Item (iii) above is the joint distribution of two {\it Busemann functions} of the polymer process.  We do not use the Busemann functions themselves in this paper and hence do not develop them.  We refer the reader to \cite{bakh-li-19,geor-rass-sepp-yilm-15, janj-rass-20-aop}. 
	
	 A methodological point to emphasize is that our proof does not rely on any integrable probability features of the inverse-gamma polymer, such as those developed in \cite{boro-corw-reme, corw-ocon-sepp-zygo}. The KPZ fluctuation estimates of Appendix \ref{sec:kpz5} were proved in  \cite{sepp-12-aop-corr} with techniques that are the same in spirit as the arguments in the present paper. 
	

It is reasonable to expect that non-existence of bi-infinite Gibbs measures  extends to general weight distributions, since the present proof  boils down to path fluctuations which are expected to be universal in 1+1 dimensions  under mild hypotheses. 	 However, currently available  techniques do not appear to  yield sufficiently sharp estimates to prove this result  in general polymer models.  Specifically,  items (ii) and (iii) from the list above    force us to work with an exactly solvable model.   

	 The zero-temperature counterpart of our result is the non-existence of bi-infinite geodesics in first-passage or last-passage percolation models.  This has been proved for the planar exponential directed  last-passage percolation model \cite{basu-hoff-sly-arxiv-18, bala-busa-sepp-arxiv}. The organization of our estimates mimics  our zero-temperature proof in \cite{bala-busa-sepp-arxiv}. 
	
	

\subsection{Organization of the paper}    Section \ref{sec:cif} develops enough of the general polymer  theory from \cite{janj-rass-20-aop} so that   in Section \ref{sec:inv-ga-thm} we can state   the main result Theorem \ref{thm:noex}  on the nonexistence of bi-infinite inverse-gamma polymer Gibbs measures.  Along the way we apply results from \cite{janj-rass-20-aop} to  prove for general weights that infinite polymers have to be directed into the open quadrant, unless they are rigid straight lines (Theorem \ref{thm:e_i-mu}). This result will also contribute to the proof of the  main Theorem \ref{thm:noex}. 

Section \ref{sec:invga} gives a quick description of the   ratio-stationary inverse-gamma polymer and derives one estimate. 

The heart of the proof is in Section \ref{sec:estim}.  A coarse-graining argument decomposes the southwest  boundary of a large $2N\times 2N$ square into blocks of size $N^{2/3}$.  Two separate  estimates are developed. 
\begin{enumerate}[(a)] 
\item  The first kind is for the probability that a polymer path from  an $N^{2/3}$-block denoted by $\cI$ goes  through the origin and reaches the diagonally opposite block  $\rim\cI$ of size $N^{19/24}$.  
 This probability is shown to decay by controlling it with random walks that come from the ratio-stationary polymer processes (Lemma \ref{lm:close}). 
\item   The second estimate  (Lemma \ref{lm:far}) controls the paths from $\cI$  through the origin that miss $\rim\cI$.  Such paths  are rare due to KPZ bounds according to which  the typical path remains  within a range of order $N^{2/3}$ around   the straight line between its endpoints.  
\end{enumerate} 
Section \ref{sec:estim}  culminates in Theorem \ref{thm:ub} that combines the estimates.   

Section \ref{sec:pf-main} combines  Theorem \ref{thm:ub} with the earlier Theorem \ref{thm:e_i-mu} to complete the proof of Theorem \ref{thm:noex}.   The estimates for paths that go through the origin are generalized to other crossing points on the $y$-axis by  suitably shifting the environment.

Since the background polymer material  will be at least partly  familiar to some readers, we have collected these facts in the appendix. Appendix \ref{app:genpol} covers   polymers on $\Z^2$  with general vertex  weights  and Appendix  \ref{app:inv-gam} specializes to inverse-gamma weights.  Appendix \ref{sec:rw} states a positive lower bound on the running maximum of a random walk with a small negative drift  that we use in a proof. This result is quoted from the technical note   \cite{busa-sepp-rw} that we have published separately. 
	
	
\subsection{Notation and conventions} Subsets of reals and integers are denoted by subscripts, as in 
	$\Z_{>0}=\{1,2,3,\dotsc\}$ and $\Z_{>0}^2=(\Z_{>0})^2$.  
	 $\lzb a,b\rzb$ denotes the integer interval $[a,b]\cap\Z$ if $a,b\in\R$, and the integer rectangle $([a_1,b_1]\times[a_2,b_2])\cap\Z^2$  if $a,b\in\R^2$.

	For points $x=(x_1,x_2)$ and $y=(y_1,y_2)$ in $\R^2$, the $\ell^1$ norm is $\abs{x}_1=\abs{x_1}+\abs{x_2}$,   the inner product is $x\cdot y=x_1y_1+x_2y_2$,    the origin is $\zevec=(0,0)$,  and the standard basis vectors are $\evec_1=(1,0)$ and $\evec_2=(0,1)$.  
We utilize two partial orders: \\[-15pt] 
\begin{enumerate} [(i)]  \itemsep=2pt 
\item the {\it coordinatewise order}:  $(x_1,x_2)\le(y_1,y_2)$  if $x_r\le y_r$ for $r\in\{1,2\}$, and 
\item the  {\it down-right order}:  $(x_1,x_2)\preccurlyeq(y_1,y_2)$  if $x_1\le y_1$ and $x_2\ge y_2$. 
\end{enumerate} 
 Their strict versions mean that the defining inequalities are strict:   $(x_1,x_2)<(y_1,y_2)$  if $x_r<y_r$ for $r\in\{1,2\}$, and   $(x_1,x_2)\prec(y_1,y_2)$  if $x_1< y_1$ and $x_2> y_2$.
	
			Sequences are denoted by $x_{m:n}=(x_i)_{i=m}^n$ and $x_{m:\infty}=(x_i)_{i=m}^\infty$ for integers $m\le n<\infty$ and also generically by $x_\bbullet$.   An admissible path $x_\bbullet$  in $\Z^2$ satisfies $x_k-x_{k-1}\in\{\evec_1,\evec_2\}$.  
	Limit velocities of these paths lie in the simplex 
	$[\evec_2,\evec_1]=\{(u,1-u): u\in[0,1]\}$, whose relative interior is  the open line segment  $\,]\evec_2,\evec_1[\,$.

	

	$\E$ and $\P$ refer to the random weights (the environment) $\w$, and otherwise  
	$E^\mu$ denotes expectation under probability measure $\mu$.
	The usual gamma function for $\rho>0$ is 
	$\Gamma(\rho)=\int_0^\infty x^{\rho-1}e^{-x}\,dx$, and  the digamma and   trigamma
	functions are $\psi_0=\Gamma'/\Gamma$ and $\psi_1=\psi_0'$.  
	$X\sim{\rm Ga}(\rho)$ if the random variable $X$ has the density function  $f(x)=\Gamma(\rho)^{-1} x^{\rho-1}e^{-x}\,dx$ on $\R_{>0}$,  and $X\sim{\rm Ga}^{-1}(\rho)$ if $X^{-1}\sim{\rm Ga}(\rho)$.  

%
%

	\section{Polymer Gibbs measures} 
	\label{sec:cif} 

%

	
\subsection{Directed polymers} 		

	Let $(\wgtd_x)_{x\tspa\in\tspa\Z^2}$ be an assignment  of strictly  positive real weights on the vertices of $\Z^2$. 
		For vertices $o\le p$ in $\Z^2$  let $\pathsp_{o,p}$ denote the set of admissible 
	lattice paths $x_\bbullet=(x_i)_{0\le i\le n}$ with $n=\abs{p-o}_1$ that satisfy $x_0=o$, 
	$x_i-x_{i-1}\in\{\evec_1,\evec_2\}$, $x_n=p$.  
	Define point-to-point
	polymer partition functions 
	between vertices  $o\le p$ in $\Z^2$  by
	\begin{align}\label{h:Z}
	Z_{o,p}=\sum_{x_\brbullet\tspa\in\tspa\pathsp_{o,p}} \prod_{i=0}^{\abs{p-o}_1} \wgtd_{x_i} .  
	\end{align}
	We use the convention $Z_{o,p}=0$ of $o\le p$ fails. 
		The   quenched polymer probability distribution  on the set  $\pathsp_{o,p}$ is defined  by 
	\be\label{h:Q}
	Q_{o,p}\{x_\bbullet\} =\frac1{Z_{o,p}}   \prod_{i=0}^{\abs{p-o}_1} \wgtd_{x_i} ,
	\quad x_\bbullet\in\pathsp_{o,p}. 
	\ee    
When the weights $\w=(Y_x)$ are random variables  on some probability space  $\OAbP$, the averaged or annealed polymer distribution   $P_{o,p}$ on $\pathsp_{o,p}$ is defined  by
	\be\label{h:P}
		P_{o,p}(A)=\int_\Omega \sum_{x_\brbullet\tspa\in\tspa A} Q^\w_{o,p}(x_\bbullet)\,\P(d\w) 
		\qquad\text{for } A\subset\pathsp_{o,p}. 
	\ee
The notation $Q^\w_{o,p}$ highlights the dependence of the quenched measure 	on the weights. 
	It is also convenient to use the unnormalized quenched polymer measure, which is simply the sum of path weights:  
\be\label{h:Z(A)}  Z_{o,p}(A)= \sum_{x_\brbullet\tspa\in\tspa A} \prod_{i=0}^{\abs{p-o}_1} \wgtd_{x_i}
= Z_{o,p} \tspb Q_{o,p}(A) \qquad\text{for } A\subset\pathsp_{o,p}.  \ee

A basic law of large numbers object of this model is the limiting {\it free energy density}.  Assume now the following:   
\be\label{mom-ass}  
\text{the weights $(Y_x)_{x\tsp\in\tsp\Z^2}$ are i.i.d.\ random variables and } \ 
\E[\tspa\abs{\log Y_0}^p\tspa]<\infty \ \text{ for some } \  p>2. 
\ee
 Then there exists   a concave, positively homogeneous, nonrandom  continuous function $\Lambda:\R_{\ge0}^2\to\R$  that satisfies this {\it shape theorem}: 
\be\label{lln}  
\lim_{n\to\infty}   \sup_{x\tsp\in\tsp\Z_{\ge0}^2: \tspb \abs{x}_1\ge n} 
\frac{\log Z_{\zevec,x}  -  \Lambda(x)}{\abs{x}_1} =0
\qquad \P\text{-almost surely.} 
\ee
(See Section 2.3 in \cite{janj-rass-20-aop}.)   In general, further regularity of $\Lambda$ is unknown.  In certain exactly solvable cases, including the inverse-gamma polymer we study in this paper,  the following properties are known:  
\be\label{reg-ass} 
\text{the function  $\Lambda$ 
 is differentiable and strictly concave on the open interval $\,]\evec_2,\evec_1[\,$.}
 \ee

	

Fix the base point $o=\zevec$ (the origin) and consider sending the endpoint $p$ to infinity in the quenched measure $Q_{\zevec,p}$.  Fix a finite path $x_{0:n}\in\pathsp_{\zevec,y}$ where  $\zevec\le y\le p$ and  $n=y\cdot(\evec_1+\evec_2)$.   To understand what happens  as $\abs p_1\to\infty$ it is convenient to write $Q_{\zevec,p}$ as a Markov chain: 
\be\label{h:Q5}
	Q_{\zevec,p}\{X_{0:n}=x_{0:n}\} =\frac1{Z_{\zevec,p}}  \biggl(\; \prod_{i=0}^{n-1} \wgtd_{x_i}\biggr) Z_{x_n,p} =   \prod_{i=0}^{n-1} \frac{Z_{x_{i+1},p}\wgtd_{x_i}}{Z_{x_i,p}}  
	\ee    
with initial state $X_0=\zevec$,   transition probability $\pi^{\zevec,p}(x,x+\evec_i)=Z_{x,p}^{-1}\tsp {Z_{x+\evec_i,p}\wgtd_{x}}$   for $p\ne x\in\lzb\zevec,p\rzb$,  and  absorbing state $p$.  The formulation above reveals that when the limit of the ratio $Z_{x+\evec_i,p}/Z_{x,p}$ exists for each fixed  $x$  as $p$ tends to infinity, then  $Q_{\zevec,p}$ converges weakly to a Markov chain.     When $p$ recedes in some particular direction,  this can be proved under local hypotheses on the regularity of $\Lambda$.  See Theorem 3.8 of  \cite{janj-rass-20-aop} for a general result and Theorem 7.1 in \cite{geor-rass-sepp-yilm-15}  for the inverse-gamma polymer.  

The limiting  Markov chains are examples of  rooted semi-infinite polymer Gibbs measures, which we discuss in the next section.

	\subsection{Infinite Gibbs measures}
	
	In this section we adopt mostly  the terminology and notation of  \cite{janj-rass-20-aop}.  
	To describe semi-infinite and bi-infinite polymer Gibbs measures, introduce the  spaces of semi-infinite and bi-infinite polymer paths in $\Z^2$: 
	\begin{align*}
	\pathsp_u&=\{x_{m:\infty}:x_m=u,\, x_i\in\Z^2,\, x_i-x_{i-1}\in\{\evec_1, \evec_2\}\}\\
\text{and}\qquad 	\pathsp&=\{x_{-\infty:\infty}: x_i\in\Z^2,\, x_i-x_{i-1}\in\{\evec_1, \evec_2\}\}.
	\end{align*}
$\pathsp_u$ is the space of paths rooted or based at the vertex $u\in\Z^2$.  The indexing of the paths is immaterial.  However, it adds clarity   to index unbounded paths  so that $x_k\cdot(\evec_1+\evec_2)=k$, as done in \cite{janj-rass-20-aop}.  We follow this convention in the present section.  So in the  definition of $\pathsp_u$ above take $m=u\cdot(\evec_1+\evec_2)$.   The projection random variables on all  the path  spaces are denoted by $X_i(x_{m:n}) =x_i$ for all choices $-\infty\le m\le n\le \infty$ and $i$ in the correct range.

	
	Fix $\omega\in\Omega$ and $m\in\Z$.   Define a family of stochastic  kernels $\{\kappa^\omega_{k,l}: l\ge k\ge m\}$ on semi-infinite paths $x_{m:\infty}$ through the integral of a bounded  Borel function $f$: 
	\be\label{kernel}\begin{aligned}
	\kappa^\omega_{k,l}f(x_{m:\infty})&=\int f(y_{m,\infty})\,\kappa^\omega_{k,l}(x_{m:\infty},dy_{m,\infty})\\
	&=\sum_{y_{k:l}\in\pathsp_{x_k, x_l}} f(x_{m:k}\,y_{k:l}\, x_{l:\infty})\,Q^\omega_{x_k,x_l}(y_{k:l}).
	\end{aligned}\ee
In other words, the action of $\kappa^\omega_{k,l}$ amounts to replacing the segment $x_{k:l}$ of the path with a new path $y_{k:l}$  sampled from  the  quenched polymer distribution $Q^\omega_{x_k,x_l}$. The argument $x_{m:k}\,y_{k:l}\, x_{l:\infty}$ inside $f$ is the concatenation of the three path segments. There is no inconsistency because $y_k=x_k$ and $y_l=x_l$ $Q^\omega_{x_k,x_l}$-almost surely.  The key point is that the measure $\kappa^\omega_{k,l}(x_{m:\infty})$ is a function of the subpaths $(x_{m:k},x_{l:\infty})$. 

Note that the same kernel $\kappa^\omega_{k,l}$ works on paths $x_{m:\infty}$ for any $m\le k$ and also  on the space $\pathsp$ of bi-infinite paths  by replacing $m$ with $-\infty$ in the expressions above.  With these kernels one defines semi-infinite and bi-infinite polymer Gibbs measures.  Let $\cF_I=\sigma\{X_i:i\in I\}$ denote the $\sigma$-algebra generated by the    projection variables indexed by the subset $I$ of indices. 
	
	\begin{definition}
		Fix $\omega\in\Omega$ and $u\in\Z^2$ and let $m=u\cdot(\evec_1+\evec_2)$.  Then  a Borel probability measure $\nu$ on $\pathsp_u$ is  a {\it semi-infinite polymer Gibbs measure rooted at $u$ in environment $\omega$}   if for all integers $l\geq k\geq m$ and any bounded Borel function $f$ on $\pathsp_u$ we have $E^\nu[ \tsp f \tspa\vert\tspa  \cF_{\lzb m,k\rzb\cup\lzb l,\infty\lzb}\tspa]=\kappa^\omega_{k,l}f$.  This set of probability measures is denoted by  $\text{DLR}^\omega_u$. 
	\end{definition}

	\begin{definition}
		Fix $\omega\in\Omega$. Then  a Borel probability measure $\mu$ on $\pathsp$ is   a {\it bi-infinite Gibbs measure  in environment $\omega$}   if for all integers $k\leq l$ and any bounded Borel function $f$ on $\pathsp$ we have $E^\mu[ \tspa f \tspa\vert\tspa  \cF_{\rzb -\infty,k\rzb\cup\lzb l,\infty\lzb}\tspa]=\kappa^\omega_{k,l}f$.   This set of probability measures is denoted by $\overleftrightarrow{\text{\rm DLR}}^\omega$.
	\end{definition}

An equivalent way to state $\mu\in\overleftrightarrow{\text{\rm DLR}}^\omega$ is to require 
\[   \int_{\pathsp} f(X_{-\infty:k}) \tspb g(X_{k:l}) \tspb h(X_{l:\infty}) \,d\mu =
 \int_{\pathsp} f(X_{-\infty:k}) \tspb (\kappa^\w_{k,l}g)(X_{-\infty:k}, X_{l:\infty}) \tspb h(X_{l:\infty}) \,d\mu \] 
 for all bounded Borel functions on the appropriate path spaces. For $\mu\in\text{DLR}^\omega_u$ the requirement is the same with $\pathsp$ replaced by $\pathsp_u$ and with  $-\infty$ replaced by $m$.

The   issue addressed in our paper is the nonexistence of nontrivial bi-infinite Gibbs measures.  For the sake of context, we state an existence theorem for semi-infinite Gibbs measures.  

	\begin{theorem}{\rm\cite[Theorem 3.2]{janj-rass-20-aop}}
	Assume \eqref{mom-ass} and \eqref{reg-ass}.   Then  there exists an event  $\Omega_0$ such that $\P(\Omega_0)=1$ and for every $\omega\in \Omega_0$ the following holds.   For each $u\in\Z^2$ and interior direction $\xi\in\,]\evec_2,\evec_1[\,$  there exists a Gibbs measure $\Pi^{\w, \xi}_u\in\text{\rm DLR}^\omega_u$ such that   $X_n/n\to\xi$ almost surely  under $\Pi^{\w, \xi}_u$.  Futhermore, these measures can be chosen to satisfy this consistency property:  if $u\cdot(\evec_1+\evec_2)\le  y\cdot(\evec_1+\evec_2)=n\le z\cdot(\evec_1+\evec_2)=r$, then for any path $x_{n:r}\in\pathsp_{y,z}$, 
	\[  \Pi^{\w, \xi}_u(X_{n:r}=x_{n:r}\,\vert\,X_n=y) = \Pi^{\w, \xi}_y(X_{n:r}=x_{n:r}).  \]  
	\end{theorem} 

Uniqueness of Gibbs measures is a more subtle topic, and we refer the reader to \cite{janj-rass-20-aop}.  
Since the Gibbs measure $\Pi^{\w, \xi}_u$ satisfies the strong law of large numbers $X_n/n\to\xi$, we can call it (strongly) {\it $\xi$-directed}.  In general, a path $x_{m:\infty}$ is $\xi$-directed if $x_n/n\to\xi$ as $n\to\infty$.

	We turn to 
  bi-infinite Gibbs measures.  First we observe that there are  trivial bi-infinite Gibbs measures supported on straight line paths. 

	 \begin{definition}
	 	A    path $x_\brbullet$
		 is a {\it straight line}   if for a fixed  $i\in\{1,2\}$,  $x_{n+1}-x_n=\evec_i$ for all path indices $n$. 
	 \end{definition}

 	If $x_{\brbullet}$ is a bi-infinite straight line    then  $\mu=\delta_{x_{\brbullet}}$ is a bi-infinite Gibbs measure because the polymer distribution $Q_{u, u+m\evec_i}$ is supported on the straight line from $u$ to $u+m\evec_i$.  More generally,  any probability measure supported on bi-infinite  straight lines  is a bi-infinite Gibbs measure.  
	
	The next natural question is whether there can be  bi-infinite polymer paths that are not merely straight lines but still  directed into $\evec_i$.  That this  option can be ruled out is essentially contained in the results of  \cite{janj-rass-20-aop}.  We make this explicit   in the next theorem.  It  says that under both semi-infinite and bi-infinite Gibbs measures,  up to a zero probability event, $\evec_i$-directedness even along a subsequence   is possible  only for straight line paths.    Note that \eqref{dlr:70} covers both $\evec_i$- and  $(-\evec_i)$-directedness.

\begin{theorem} \label{thm:e_i-mu}  Assume \eqref{mom-ass}.  There exists an event $\Omega_0\subseteq \Omega$ such that $\P(\Omega_0)=1$ and for every $\omega\in \Omega$ the following statements hold  for both $i\in\{1,2\}$:  \smallskip 

{\rm (a)}  For all $u\in\Z^2$ and  $\nu\in{\text{\rm DLR}}^\omega_u$, with  $m=u\cdot(\evec_1+\evec_2)$,    
\be\label{dlr:65}  \nu\bigl\{ \tsp\varliminf_{n\to\infty} n^{-1}\abs{X_n\cdot\evec_{3-i}}  =  0 \bigr\}   = \nu\{  X_n=u+(n-m)\evec_i \text{ for }  n\ge m \}.  
\ee

{\rm (b)} 
For all  $\mu\in\overleftrightarrow{\text{\rm DLR}}^\omega$, 
\be\label{dlr:70}  \mu\bigl\{ \tsp\varliminf_{\abs n\to\infty} \abs{n^{-1} X_n\cdot\evec_{3-i}}  =  0 \bigr\}   = \mu\{ \text{$X_{-\infty:\infty}$ is an $\evec_i$-directed bi-infinite straight line} \}.  
\ee
\end{theorem} 

\begin{proof} 
Let the event $\Omega_0$ of full $\P$-probability be  the intersection of the events specified in Lemma 3.4 and Theorem 3.5 of   \cite{janj-rass-20-aop}.  

\medskip 

 Part (a).   We can assume that 
  the left-hand side of \eqref{dlr:65}  is positive because the event on the right is a subset of the one on the  left.   Since $A=\{ \varliminf_{n\to\infty} n^{-1}\abs{X_n\cdot\evec_{3-i}}  =  0\}$ is a tail event, it follows   that $\wt\nu=\nu(\cdot\tspb\vert\,A)\in \text{DLR}^{\omega}_u$.   Since $\pathsp_u$ is compact,   $\wt\nu$ is a mixture of extreme members of $\text{DLR}^{\omega}_u$. (This is an application of Choquet's theorem, discussed more thoroughly in Section 2.4  of   \cite{janj-rass-20-aop}.)   This mixture can be restricted to  extreme Gibbs measures that  give the event  $A$ full probability.
  
    By Lemma 3.4 and  Theorem 3.5 of   \cite{janj-rass-20-aop},  an extreme member of  $\text{DLR}^{\omega}_u$ that is not directed into the open interval $]\evec_2,\evec_1[$ must be a  degenerate point measure 
    $\Pi^{\evec_i}_u$, which  is the probability measure supported on the single straight line path $(u+(n-m)\evec_i)_{n:n\ge m}$.  We conclude that $\wt\nu=\Pi^{\evec_i}_u$.

 From this we deduce \eqref{dlr:65}.  Let $B^{\evec_i}_u=\{ X_n=u+(n-m)\evec_i \text{ for }  n\ge m\} $ be the event that from $u$ onwards the path is an $\evec_i$-directed line. Then by conditioning,  
 \begin{align*}
 \nu(B^{\evec_i}_u) = \nu(B^{\evec_i}_u\cap A) = \wt\nu(B^{\evec_i}_u) \tspa\nu(A) =\nu(A).    
 \end{align*}

\medskip 

Part (b).  Consider first the case $n\to\infty$.   Let $m\in\Z$ and $x\cdot(\evec_1+\evec_2)=m$.  
Suppose  $\mu(X_m=x)>0$.  Then, by Lemma 2.4 in \cite{janj-rass-20-aop}, 
  $\mu_x=\mu(\cdot\tspb\vert\,X_m=x) \in \text{DLR}^{\omega}_x$.
Part (a) applied to $\mu_x$ shows that 
\be\label{dlr:72}  \mu\{ X_m=x,\,  \,\varliminf_{n\to\infty} n^{-1}\abs{X_n\cdot\evec_{3-i}}  =  0 \}   = \mu\{  X_n=x+(n-m)\evec_i \text{ for }  n\ge m\}.   
\ee
By summing over the pairwise disjoint events $\{X_m=x\}$ gives, for each fixed $m\in\Z$,  
\[ \mu\{  \tsp\varliminf_{n\to\infty} n^{-1}\abs{X_n\cdot\evec_{3-i}}  =  0 \tsp\}   = \mu\{  X_n=X_m+(n-m)\evec_i \text{ for }  n\ge m\}.  \] 
The events on the right decrease as $m\to-\infty$, and in the limit we get 
\[ \mu\{ \tsp\varliminf_{n\to\infty} n^{-1}\abs{X_n\cdot\evec_{3-i}}  =  0\tsp \}   = \mu\{  X_n=X_m+(n-m)\evec_i \text{ for all }  n, m\in\Z\}  \] 
which is exactly the claim  \eqref{dlr:70} the case $n\to\infty$. 

The case $n\to-\infty$ of  \eqref{dlr:70} follows by reflection across the origin.  Let $\w=(Y_x)_{x\in\Z^2}$ and define reflected weights $\wt\w=(\wt Y_x)_{x\in\Z^2}$ by $\wt Y_x=Y_{-x}$.   Given  $\mu\in\overleftrightarrow{\text{\rm DLR}}^\omega$,  define  the reflected measure $\wt\mu$ by setting,  for $m\le n$ and $x_{m:n}\in\pathsp_{x_m,x_n}$,    
$\wt\mu(X_{m:n}=x_{m:n}) =\mu(X_i=-x_{-i}\ \text{ for } i=-n,\dotsc,-m)$. 
Then $\wt\mu\in\overleftrightarrow{\text{\rm DLR}}^{\wt\omega}$.   
Directedness towards $-\evec_i$ under $\mu$ is now directedness towards $\evec_i$ under $\wt\mu$, and we get the conclusion by applying the already proved part to $\wt\mu$. 
%
%
%
%
%
%
\end{proof}

Moving  away from the $\evec_i$-directed cases,    the non-existence problem was resolved by Janjigian and Rassoul-Agha in the case of Gibbs measures directed towards a fixed   interior  direction: 

	\begin{theorem} {\rm\cite[Thm.~3.13]{janj-rass-20-aop}}   \label{thm:jra6} 
	Assume \eqref{mom-ass} and \eqref{reg-ass}.   Fix $\xi\in\,]\evec_2,\evec_1[\,$.  Then  there exists an event  $\Omega_{\text{\rm bi},\xi}\subseteq\Omega$ such that $\P(\Omega_{\text{\rm bi},\xi})=1$ and for every $\omega\in \Omega_{\text{\rm bi},\xi}$ there exists no measure $\mu\in\overleftrightarrow{\text{\rm DLR}}^\omega$ such that as $n\to\infty$, $X_n/n\to\xi$ in probability under $\mu$.
	\end{theorem}
	
  We assumed \eqref{reg-ass} above to avoid introducing technicalities not needed in the rest of the paper.  	
	The global regularity assumption \eqref{reg-ass} can be weakened to local hypotheses, as done in 
 Theorem 3.13 in \cite{janj-rass-20-aop}.     
 
The results above illustrate how far one can presently go without stronger assumptions on the model.    The hard question  left open   is whether    bi-infinite  Gibbs measures can exist in random directions in the open interval $\,]\evec_2, \evec_1[\,$.  To rule these out we restrict our treatment  to   the exactly solvable case of  inverse-gamma distributed weights.  

  That only directed Gibbs measures would need to be considered in the sequel  is a consequence of Corollary 3.6 of  \cite{janj-rass-20-aop}.  However, we do not need to assume this directedness a priori and we do not use Theorem \ref{thm:jra6}.  At the end we will appeal to Theorem \ref{thm:e_i-mu} to rule out the extreme slopes.    As stated above,  Theorem \ref{thm:e_i-mu} does not seem to involve   the regularity of $\Lambda$.  But  in fact through appeal to Theorem 3.5 of  \cite{janj-rass-20-aop},  it does rely on the  nontrivial (but provable)  feature   that $\Lambda$  is not affine on any interval of the type $\,]\zeta,\evec_1]$ (and symmetrically on $[\evec_2, \eta[\,$).    This is the positive temperature counterpart of Martin's shape asymptotic on the boundary \cite{mart-04} and can be deduced from that (Lemma B.1 in \cite{janj-rass-20-aop}).


 \subsection{Bi-infinite Gibbs measures in the inverse-gamma polymer}
 \label{sec:inv-ga-thm}
 
A random variable $X$ has the {\it inverse gamma distribution} with parameter $\theta>0$, abbreviated $X\sim\text{\rm Ga}^{-1}(\theta)$, if its reciprocal $X^{-1}$ has the standard  gamma distribution with parameter $\theta$, abbreviated $X^{-1}\sim\text{\rm Ga}(\theta)$.    Their density functions for $x>0$ are 
\be\label{invga} \begin{aligned} 
 f_{X^{-1}}(x)&= \frac1{\Gamma(\theta)} \tsp x^{\theta-1} e^{-x} 
 \qquad\text{for the gamma distribution ${\rm Ga}(\theta)$} \\
\text{ and }\quad 
 f_X(x)&=\frac1{\Gamma(\theta)} \tsp x^{-1-\theta} e^{-x^{-1}} 
  \qquad\text{for the inverse gamma distribution ${\rm Ga}^{-1}(\theta)$.}
\end{aligned} \ee
Here $\Gamma(\theta)=\int_0^\infty s^{\theta-1} e^{-s}\,ds$ is the gamma function. 

 Our basic assumption is: 
 \be\label{ass-invga}  \begin{aligned} 
&\text{The weights $(Y_x)_{x\tspa\in\tspa\Z^2}$ are i.i.d.\ inverse-gamma distributed random variables} \\
&\text{on some probability space $\OAbP$.} 
\end{aligned} \ee
 
 The main result is stated as follows. 
 
	\begin{theorem}\label{thm:noex}
		Assume \eqref{ass-invga}. Then for $\P$-almost every $\w$, every bi-infinite Gibbs measure  
		is supported on straight lines: that is, $\mu\in\overleftrightarrow{\text{\rm DLR}}^\omega$ implies that $\mu(\text{$X_{-\infty:\infty}$ is a bi-infinite  straight line})=1$. 
	\end{theorem}
	
 Due to Theorem \ref{thm:e_i-mu}(b),  to prove Theorem \ref{thm:noex} we only need to rule out the possibility of bi-infinite polymer measures that are directed towards the open segments $\,]-\evec_2, -\evec_1[\,$ and $\,]\evec_2, \evec_1[\,$.   The detailed proof  is given in Section \ref{sec:pf-main}, after the development of preliminary estimates.  For the proof we take $Y_x$ to be a Ga$^{-1}(1)$ variable. 

For the interested reader, we mention that 	
the semi-infinite Gibbs measures of the inverse-gamma polymer are described in the forthcoming work \cite{fan-sepp-20+}. Earlier  results appeared in \cite{geor-rass-sepp-yilm-15} where such measures were obtained as almost sure weak limits of quenched point-to-point and point-to-line  polymer distributions.

\section{Stationary inverse-gamma polymer}  
\label{sec:invga}
The proof of Theorem  \ref{thm:noex}  relies on the fact that the inverse-gamma polymer possesses a stationary version with accessible distributional properties, first constructed in \cite{sepp-12-aop-corr}. 
This section gives a brief description of the stationary polymer and proves an estimate.  Further properties of the stationary polymer are developed in the appendixes.  

	Let $(Y_x)_{x\in\Z^2}$ be i.i.d.\ Ga$^{-1}(1)$ weights.  
A stationary version of the inverse-gamma polymer is defined in a quadrant by choosing suitable boundary weights   
 on  the south and west boundaries of the quadrant. For a parameter $0<\alpha<1$ and  a base vertex $o$, introduce independent  boundary weights  on the $x$- and $y$-axes emanating from $o$: 
 \be\label{IJ1}  I^\alpha_{o+i\evec_1}\sim{\rm Ga}^{-1}(1-\alpha)
 \qquad\text{and}\qquad
 J^\alpha_{o+j\evec_2}\sim{\rm Ga}^{-1}(\alpha)
 \qquad\text{for }  \ i,j\ge 1.  
 \ee
 The above convention, that the horizontal edge weight $I^\alpha$ has parameter $1-\alpha$ while the vertical $J^\alpha$ has $\alpha$,  is followed consistently and it determines various formulas  in the sequel. 

For  vertices  $p\ge o$   define the partition functions 
	\begin{align}\label{g:Z1}
	Z^\alpha_{o,p}=\sum_{x_\brbullet\tspa\in\tspa\pathsp_{o,p}} \prod_{i=0}^{\abs{p-o}_1} \wt\wgtd_{x_i} 
	\quad\text{with weights}\quad
	\wt\wgtd_x
	=\begin{cases}  
	1, &x=o\\[2pt] 
	Y_x,  &x\in o+\Z^2_{>0}\\[2pt]  
	I^\alpha_{x},  &x\in o+(\Z_{>0})\evec_1\\[2pt] 
	J^\alpha_{x},  & x\in o+(\Z_{>0})\evec_2.
	\end{cases}
	\end{align}
Note that now a weight at $o$ does not count. 	 The superscript $\alpha$ distinguishes $Z^\alpha_{o,p}$ from the generic partition function $Z_{o,p}$ of \eqref{h:Z}. 
 The stationarity property is that the joint distribution of the  ratios $Z^\alpha_{o,x}	/Z^\alpha_{o,x-\evec_i}$ is invariant under translations of $x$ in the quadrant $o+\Z_{\ge0}^2$.   See Appendix \ref{sec:stat-pol} for more details. 
 
 The quenched polymer distribution corresponding to \eqref{g:Z1} is given by 
 $Q^\alpha_{o,p}(x_\bbullet)=  (Z^\alpha_{o,p})^{-1}  \prod_{i=0}^{\abs{p-o}_1} \wt\wgtd_{x_i}$   for $x_\bbullet\in\pathsp_{o,p}$,  and the annealed measure is  $P^\alpha_{o,p}(x_\bbullet)=\E[Q^\alpha_{o,p}(x_\bbullet)]$. 
 
 It will be convenient to consider also backward polymer processes whose paths proceed in the southwest direction and the stationary version starts with boundary weights on the north and east.   For vertices $o\ge p$ let $\rim{\pathsp}_{o,p}$ be the set of down-left paths starting at $o$ and terminating at $p$. As sets of vertices and edges, paths in   $\rim{\pathsp}_{o,p}$ are exactly the same as those in  $\pathsp_{p,o}$. The difference is that in  $\rim{\pathsp}_{o,p}$  paths are indexed in the down-left direction.  

 For $o\ge p$, backward partition functions are then defined with i.i.d.\ bulk weights as 
 	\begin{align}
	\rim{Z}_{o,p}=\sum_{x_\brbullet\tspa\in\tspa\rim{\pathsp}_{o,p}} \prod_{i=0}^{\abs{o-p}_1} \wgtd_{x_i} 
	\end{align}
and in the stationary case as 
 	\begin{align}\label{g:Z1rev}
	\rim Z^\alpha_{o,p}=\sum_{x_\brbullet\tspa\in\tspa\rim\pathsp_{o,p}} \prod_{i=0}^{\abs{o-p}_1} \wt\wgtd_{x_i} 
	\quad\text{with weights}\quad
	\wt\wgtd_x
	=\begin{cases}  
	1, &x=o\\[2pt] 
	Y_x,  &x\in o-\Z^2_{>0}\\[2pt]  
	I^\alpha_{x},  &x\in o-(\Z_{>0})\evec_1\\[2pt] 
	J^\alpha_{x},  & x\in o-(\Z_{>0})\evec_2.
	\end{cases}
	\end{align}
The independent boundary weights $I^\alpha_{o-i\evec_1}$ and $J^\alpha_{o-j\evec_2}$ ($i,j\ge 1$)  have the distributions \eqref{IJ1}. 
   	
We define functions that capture the wandering of a path 	  $x_\bbullet\in\pathsp_{o,p}$.  
The (signed)  exit point or exit time   $\ex_{o,p}=\ex_{o,p}(x_\bbullet)$ marks the position where the path $x_\bbullet$  leaves the southwest boundary and moves into the bulk, with the convention that a negative value indicates a jump off the $y$-axis.    More generally, for 3 vertices $o\le v<p$,  $\ex_{o,v,p}=\ex_{o,v,p}(x_\bbullet)$ marks the position where  $x_\bbullet\in\pathsp_{o,p}$ enters the rectangle $\lzb v+\evec_1+\evec_2, p\rzb$, again with a negative sign if this entry happens on the east edge $\{v+\evec_1+j\evec_2: 1\le j\le (p-v)\cdot\evec_2\}$. Here is the precise definition: 
\be\label{exit77} 
\ex_{o,v,p}(x_\bbullet)
=\begin{cases}  -\max\{j\ge 1: v+j\evec_2\in x_\bbullet\}, &\text{if } x_\bbullet\cap(v+(\Z_{>0})\evec_2)\ne\emptyset\\[2pt]
\max\{i\ge 1: v+i\evec_1\in x_\bbullet\}, &\text{if } x_\bbullet\cap(v+(\Z_{>0})\evec_1)\ne\emptyset. 
\end{cases} 
\ee
Exactly one of the  two cases above happens for each path 	  $x_\bbullet\in\pathsp_{o,p}$.  The exit point from the boundary is then defined by $\ex_{o,p}=\ex_{o,o,p}$. 

An analogous definition is made for the backward polymer.   For $o\ge v> p$ and  $x_\bbullet\in\rim{\pathsp}_{o,p}$, 
\[ 
\rim\ex_{o,v,p}(x_\bbullet)
=\begin{cases}  -\max\{j\ge 1: v-j\evec_2\in x_\bbullet\}, &\text{if } x_\bbullet\cap(v-(\Z_{>0})\evec_2)\ne\emptyset\\[2pt]
\max\{i\ge 1: v-i\evec_1\in x_\bbullet\}, &\text{if } x_\bbullet\cap(v-(\Z_{>0})\evec_1)\ne\emptyset. 
\end{cases} 
\] 
  The signed exit point from the northeast  boundary is  $\rim\ex_{o,p}=\rim\ex_{o,o,p}$. 
 	
\medskip 

The remainder of this section is devoted to an estimate needed in the body of the proof.  
First recall that the {\it digamma function} $\psi_0=\Gamma'/\Gamma$ is strictly concave and strictly increasing on $(0,\infty)$,
	with $\psi_0(0+)=-\infty$ and $\psi_0(\infty)=\infty$.   Its derivative, the {\it trigamma function} $\psi_1=\psi_0'$, is positive,  strictly convex, and strictly decreasing, with
	$\psi_1(0+)=\infty$ and $\psi_1(\infty)=0$.   These functions appear as means and variances:  
\be\label{invga7} 	
	\text{for  }\ \eta\sim \text{\rm Ga}^{-1}(\alpha),  \ \  \E[\log \eta]=-\psi_0(\alpha) \quad\text{and}\quad \Vvar(\log \eta)=\psi_1(\alpha).  
	 \ee
	
 In the stationary polymer  $Z^\alpha_{o,p}$ in \eqref{g:Z1}, the boundary weights are stochastically larger than the bulk weights.  Consequently the polymer path prefers to run along one of the  boundaries,  its choice determined by the direction $(p-o)/\abs{p-o}_1\in[\evec_2, \evec_1]$.   For each parameter $\alpha\in(0,1)$ there  a particular {\it characteristic direction} $\xi(\alpha)\in\,]\evec_2, \evec_1[$ at which the attraction of the two boundaries balances out.  For $\rho\in[0,1]$ this function is given by 
 	\begin{align}\label{XtR}
	\xi(\rho)=\Bigl(\frac{\psi_1(\rho)}{\psi_1(\rho)+\psi_1(1-\rho)} \, ,\,
	\frac{\psi_1(1-\rho)}{\psi_1(\rho)+\psi_1(1-\rho)}\Bigr)\in[\evec_2,\evec_1].
	\end{align}
The extreme cases are interpreted as $\xi(0)=\evec_1$ and $\xi(1)=\evec_2$.  
 The inverse function $\rho=\rho(\xi)$ of a  direction  
	  $\xi=(\xi_1, \xi_2)\in[\evec_2,\evec_1]$   is defined by 
	$\rho(\evec_2)=1$,  $\rho(\evec_1)=0$, and 
	\begin{align*}
	-\xi_1\tspb\trigamf(1-\rho(\xi))+\xi_2\tspb\trigamf(\rho(\xi))=0
	\quad\text{for $\xi \in \,]\evec_2,\evec_1[\,$ }. 
	\end{align*}

The function $\rho(\xi)$ is  a strictly decreasing bijective 
	mapping of  $\xi_1\in[0,1]$ onto  $\rho\in[0,1]$, or, equivalently, a strictly decreasing mapping of $\xi$ in the down-right order. 
The  significance of 	the characteristic direction for fluctuations is that $\ex_{o,p}$ is of order $\abs{p-o}_1^{2/3}$ if and only if $p-o$ is directed towards $\xi(\alpha)$, and of order $\abs{p-o}_1$ in all other directions. These fluctuation questions were first investigated in \cite{sepp-12-aop-corr}.   

We insert here a lemma on the regularity of the characteristic direction.

	\begin{lemma}\label{lem:psi}
		There exist  functions  $\fo>0$ and $B>0$ on $(0,1)$ such that, whenever  $\rho_0\in(0,1)$
		and 
		$|\delta-\rho_0|<\rho_1=\frac{1}{2}(\rho_0\wedge(1-\rho_0))$, 
		\begin{align}\label{psi1}
		\frac{\xi_2(\rho_0+\delta)}{\xi_1(\rho_0+\delta)}-\frac{\xi_2(\rho_0)}{\xi_1(\rho_0)}&= \fo(\rho_0)\delta+f(\rho_0, \delta) 
		\end{align} 
		where the function $f$ satisfies 
		\begin{align}  
		|f(\rho_0, \delta)|&\leq B(\rho_0)\delta^2\label{psi2}.
		\end{align}
		The functions $\fo$, $\fo^{-1}$ and $B$ are   bounded  on any compact subset of  $(0,1)$.
	\end{lemma}
	\begin{proof}
		As the function $\psi_1$ is smooth on $(0,\infty)$
		\begin{align*}
		\frac{\xi_2(\rho_0+\delta)}{\xi_1(\rho_0+\delta)}-\frac{\xi_2(\rho_0)}{\xi_1(\rho_0)}
		&=\frac{\psi_1(1-(\rho_0+\delta))}{\psi_1(\rho_0+\delta)}-\frac{\psi_1(1-\rho_0)}{\psi_1(\rho_0)}\\
		&=-\, \delta\tspa\frac{\psi'_1(1-\rho_0)\psi_1(\rho_0)+\psi_1'(\rho_0)\psi_1(1-\rho_0)}{\psi_1(\rho_0)^2}+f(\rho_0,\delta) 
		\end{align*}
		where $\psi_1'<0$ and  $f$ satisfies \eqref{psi2}.
		\end{proof}
		
Recall that to prove  Theorem \ref{thm:noex},  our intention is to rule out bi-infinite polymer measures whose forward  direction is  into the open first quadrant, and whose backward direction is into the open third quadrant. The main step towards this is that, as $N$ becomes large, a polymer path from southwest to northeast across the square $\lzb -N,N\rzb^2$, with slope bounded away from $0$ and $\infty$, cannot cross the $y$-axis anywhere close to the origin.

 To achieve this we control partition functions from the southwest  boundary of the square $\lzb-N,N\rzb^2$ to the  interval  $\cJ=\lzb -N^{{2}/{3}}\evec_2,N^{{2}/{3}}\evec_2 \rzb$ on the $y$-axis,  and backward partition functions  from the northeast  boundary of the square $\lzb-N,N\rzb^2$  to the interval $\rim\cJ=\evec_1+\cJ$ shifted one unit  off the $y$-axis.

	Let $\e>0$. We establish notation for the southwest portion of the boundary of the square $\lzb -N,N\rzb^2$ that is bounded by the  lines of slopes $\e$ and $\e^{-1}$.    With $\mathcal W$ for west and $\cS$ for south, let  $\partial^N_\mathcal{W} =\{-N\}\times \lzb -N,-\e N \rzb$, $\partial^N_\mathcal{S}=\lzb -N,-\e N\rzb \times \{-N\}$, and then $\partial^N=\partial^{N, \e}=\partial^N_\mathcal{W}\cup \partial^N_{S}$. 
	The parameter $\e>0$ stays fixed for most  of the proof, and hence will be suppressed from much of the notation.   
	We also let $o_i=(-N,-\e N)$ and $o_f=(-\e N,-N)$.  A lattice point $o=(o_1, o_2)\in \partial^N$ is associated    with its (reversed) direction  vector $ \xi(o)=(\xi_1(o), 1-\xi_1(o))\in \,]\evec_2,\evec_1[\,$ and parameter  $\rho(o)\in (0,1)$   through the relations
	\begin{align}
	\xi(o)&=  \left(\frac{o_1}{o_1+o_2},\,\frac{o_2}{o_1+o_2}\right)\label{xi-o}
	\\ 
	\intertext{and indirectly via \eqref{XtR}:} 
	\rho(o)&=\rho(\xi(o))  
	 \ \iff\ \xi(\rho(o)) =\xi(o) 
	\label{rho-o}
	\end{align}
  For all $o\in \partial^N$ we have the bounds 
	\[   \xi(o)\in \Bigl[ \Bigl(\frac{1}{1+\e},\frac{\e}{1+\e}\Bigr), \Bigl(\frac{\e}{1+\e},\frac{1}{1+\e}\Bigr)\Bigr]=[\xi_i,\xi_f].
	\]
	If we define the extremal parameters (for a given $\e>0$) by 
	\begin{align*}
	\rho_i=\rho(o_i)=\rho\biggl( \frac{1}{1+\e},\frac{\e}{1+\e}\biggr)
	\qquad\text{and}\qquad
	\rho_f=\rho(o_f)=\rho\biggl(\frac{\e}{1+\e},\frac{1}{1+\e}\biggr)
	\end{align*}
then we have the uniform bounds 
\be\label{2089}
0< \rho_i \le \rho(o)\le \rho_f < 1 
\qquad\text{ for all }  o\in\partial^N=\partial^{N\!,\tspa\e}.   
\ee	
	 


 For $o\in \partial^N$ define perturbed parameters (with dependence on $r, N$ suppressed from the notation): 
\be\label{rho*} 
	\rhodown{\rho}(o)=\rho(o)-rN^{-\frac{1}{3}}
	\qquad\text{and}\qquad 
	\rhoup{\rho}(o)=\rho(o)+rN^{-\frac{1}{3}}. 
\ee
	The variable $r$ can be a function of $N$ and become  large but always $r(N)N^{-1/3}\to0$ as $N\to\infty$. Then for $N\ge N_0(\e)$  the perturbed parameters are bounded  uniformly away from $0$ and $1$: 
\be\label{2095}
0< \rho_0(\e)  \le  \rhodown{\rho}(o) 
	<  \rhoup{\rho}(o)  \le \rho_1(\e) <1 
\quad\text{ for all }  o\in\partial^N=\partial^{N\!,\tspa\e}\text{  and  } N\ge  N_0(\e).   
\ee

	We consider the stationary processes  $Z^{\rhoup{\rho}(o)}_{o,\bbullet}$ and $Z^{\rhodown{\rho}(o)}_{o,\bbullet}$.  Our next lemma shows that the perturbation $r$ can be taken such that,  for all  $o\in \partial^N$ and $x\in \cJ=\lzb -N^{{2}/{3}}\evec_2,N^{{2}/{3}}\evec_2 \rzb$,  on the scale $N^{2/3}$  the exit point under $Q^{\rhoup{\rho}(o)}_{o,x}$ is far enough in the $\evec_1$ direction, and under  $Q^{\rhodown{\rho}(o)}_{o,x}$ far enough in the $\evec_2$ direction,   with high probability.

	\begin{lemma}\label{lem-lb1} For each $\e>0$   there exist finite positive constants 
$c(\e), C_0(\e), C_1(\e)$ and $N_0(\e)$  such that, whenever  $1\le d\le c(\e)N^{1/3}$, $C_0(\e)d\le r\le c(\e)N^{1/3}$,  $N\ge N_0(\e)$, 	 $o\in\partial^N$,  and $y>0$,  we have the bounds 
		\be\label{lb-1}   \P\Big\{\sup_{x\in\cJ}Q^{\rhodown{\rho}(o)}_{o,x}\big(\ex_{o,x}\ge -dN^{\frac{2}{3}}\big)>y  \Big\} \le C_1(\e)y^{-1}r^{-3}
		\ee
		and 
		\be\label{lb-2}   \P\Big\{\sup_{x\in\cJ}Q^{\rhoup{\rho}(o)}_{o,x}\big(\ex_{o,x}\le dN^{\frac{2}{3}}\big)>y  \Big\} \le C_1(\e)y^{-1}r^{-3}.
		\ee
	\end{lemma}
	
	\begin{proof}
		We prove \eqref{lb-2} as \eqref{lb-1} is similar.  		
		We turn the quenched probability into a form to which we can apply fluctuation bounds.  The justifications of the steps below go as follows. 
	\begin{enumerate} [(i)] 
	\item  The first inequality below is from  \eqref{pmon}. 
	\item Observe that the path leaves the boundary to the left  of the point $o+dN^{2/3}\evec_1$ if and only if   it intersects the vertical line $o+dN^{2/3}\evec_1 +j\evec_2$ at some $j\ge 1$. 
	\item  Move the base point from $o$ to $o+dN^{2/3}\evec_1$ and apply  \eqref{m:830}. By the stationarity, the new boundary weights on the axes emanating from $o+dN^{2/3}\evec_1$ have the same distribution as the original ones. This gives the equality in distribution. 
	\item Choose an integer $\ell$ so that the vector from $o+dN^{2/3}\evec_1-\ell\evec_2$ to  $N^{{2}/{3}}\evec_2$ points in the characteristic direction $\xi(\rhoup{\rho}(o))$.   	Apply  \eqref{m:830} and stationarity.	
\end{enumerate} 	
		\begin{align*} 
			 &\sup_{x\in\cJ}Q^{\rhoup{\rho}(o)}_{o,x}\big(\ex_{o,x}< dN^{\frac{2}{3}}\big)  \leq  Q^{\rhoup{\rho}(o)}_{o,\tsp N^{2/3}\evec_2}\big(\ex_{o,\tsp N^{2/3}\evec_2}< dN^{\frac{2}{3}}\big) \\
	&\qquad =	 
		 Q^{\rhoup{\rho}(o)}_{o,\tsp N^{2/3}\evec_2}\big(\ex_{o, \tspa o+dN^{2/3}\evec_1, \tspa N^{2/3}\evec_2}< 0 \big)  
		 \deq Q^{\rhoup{\rho}(o)}_{o+dN^{2/3}\evec_1,\tspa N^{2/3}\evec_2}\big(\ex_{o+dN^{2/3}\evec_1, \tspa N^{2/3}\evec_2}< 0 \big) \\[3pt] 
		 &\qquad 
		 = Q^{\rhoup{\rho}(o)}_{o+dN^{2/3}\evec_1-\ell\evec_2, \tspa N^{2/3}\evec_2}\big(\ex_{o+dN^{2/3}\evec_1-\ell\evec_2, \tspa N^{2/3}\evec_2}< -\ell\tsp\big). 
		\end{align*}

We  show   that 
$\ell\ge c_0(\e)rN^{2/3}$ for a constant $c_0(\e)$.  Let $o=-(Na, Nb)$, with $\e\le a,b\le 1$.  
		Lemma \ref{lem:psi} 	gives the next identity.   The $O$-term hides an $\e$-dependent constant that  is uniform for all $\rho(o)$ because, as observed in \eqref{2089},  the assumption $o\in\partial^N$ bounds $\rho(o)$ away from $0$ and $1$. 
\begin{align*}
\frac{N^{2/3}+Nb+\ell}{Na-dN^{2/3}} = \frac{\xi_2(\rhoup{\rho}(o))}{\xi_1(\rhoup{\rho}(o))} = \frac{b}a + \fo(\rho(o)) rN^{-1/3}  +O(r^2N^{-2/3}) . 
\end{align*}		
From this we deduce 
\begin{align*}
\ell =  \fo(\rho(o)) arN^{2/3} -  \frac{b}a d N^{2/3} -N^{2/3} - \fo(\rho(o)) rdN^{1/3}  +O(r^2N^{1/3}) +O(r^2d) . 
\end{align*}
Recall from Lemma \ref{lem:psi} that $\fo(\rho(o))>0$ is uniformly bounded away from zero for $o\in\partial^N$.  
For a small enough constant $c(\e)$ and large enough constants $C_0(\e)$ and $N_0(\e)$, if we have   $1\le d\le c(\e)N^{1/3}$, $C_0(\e)d\le r\le c(\e)N^{1/3}$ and $N\ge N_0(\e)$,    the above simplifies to  $\ell\ge c_0(\e)rN^{2/3}$. 

We can derive the final bound.  
		\begin{align*}
			&\P\Big\{\sup_{x\in\cJ}Q^{\rhoup{\rho}(o)}_{o,x}\big(\ex_{o,x}< dN^{\frac{2}{3}}\big)>y  \Big\}\\
			&\qquad  \leq   \P\Big\{  Q^{\rhoup{\rho}(o)}_{o+dN^{2/3}\evec_1-\ell\evec_2, \tspa N^{2/3}\evec_2}\big(\ex_{o+dN^{2/3}\evec_1-\ell\evec_2, \tspa N^{2/3}\evec_2}< -\ell\tsp\big)   >y  \Big\}  \\  
		 &\qquad \le y^{-1} \tspa \E\Big[Q^{\rhoup{\rho}(o)}_{o+dN^{2/3}\evec_1-\ell\evec_2, \tspa N^{2/3}\evec_2}\big(\ex_{o+dN^{2/3}\evec_1-\ell\evec_2, \tspa N^{2/3}\evec_2}< -c(\e)rN^{2/3}\tsp\big)  \Big]\\
			&\qquad 
			=y^{-1} \tsp P^{\rhoup{\rho}(o)}_{o+dN^{2/3}\evec_1-\ell\evec_2, \tspa N^{2/3}\evec_2}\big(\ex_{o+dN^{2/3}\evec_1-\ell\evec_2, \tspa N^{2/3}\evec_2}< -c_0(\e)rN^{2/3}\tsp\big)\leq C_1(\e)y^{-1} \tspa r^{-3}.
		\end{align*} 
The final inequality comes from Theorem \ref{thm:kpz3}. 	
	\end{proof}

	\section{Estimates for paths across a large square} \label{sec:estim} 
	
After the preliminary work above we turn to develop the estimates that prove the main theorem. 	
	Throughout,  $\dvec=(d_1,d_2)\in \Z_{\ge 1}^2$ denotes  a pair of parameters that control the coarse graining on the southwest and northeast boundaries  of the square $\lzb-N,N\rzb^2$.   For  $o\in \partial^N$  let 
	\begin{align*}
	\cI_{o,\dvec}=\{u \in \partial^N: \abs{u-o}_1 
	\leq \tfrac12{d_1}N^\frac{2}{3}\}. 
	\end{align*}
	Let $o_c\in\cI_{o,\dvec}$ denote the minimal point of $\cI_{o,\dvec}$ in the coordinatewise partial order, that is, defined by the requirement that 
\[  o_c\in\cI_{o,\dvec} \quad\text{and}\quad   o_c\le u \ \ \forall u\in\cI_{o,\dvec}. \] 	
This setting is illustrated in Figure \ref{fig:points}. 

On the rectangle $\lzb o_c, N\evec_2\rzb$ we define coupled polymer processes.   For each  $u\in\cI_{o,\dvec}$ we have   the bulk process $Z_{u,\brbullet}$  that uses  $\text{\rm Ga}^{-1}(1)$   weights $Y$.   Two stationary comparison processes based at $o_c$  have parameters $\rhodown{\rho}(o_c)$ and $\rhoup{\rho}(o_c)$ defined as in \eqref{rho*}. Their  basepoint is taken as  $o_c$ so that we get simultaneous control over  all the processes based at vertices  $u\in \cI_{o,\dvec}$.  

Couple the boundary weights on the south and west boundaries of the rectangle $\lzb o_c, N\evec_2\rzb$ as described in Theorem \ref{thm:st-lpp} in Appendix \ref{sec:stat-pol}. In particular, for $k,\ell\ge 1$ we have the inequalities 
\be\label{2160} 
Y_{o_c+k\evec_1} \le I^{\rhodown{\rho}(o_c)}_{o_c+k\evec_1}\le I^{\rhoup{\rho}(o_c)}_{o_c+k\evec_1}
\quad\text{ and }\quad 
Y_{o_c+\ell\evec_2} \le J^{\rhoup{\rho}(o_c)}_{o_c+\ell\evec_2}\le J^{\rhodown{\rho}(o_c)}_{o_c+\ell\evec_2}. 
\ee

For all these coupled processes we define ratios of the partition functions from the base point  to the $y$-axis, for all $u\in \cI_{o,\dvec}$ and $i\in\lzb-N^{2/3}, N^{2/3}\rzb $: 	
\be\label{2165} 
	J^u_i=\frac{Z_{u,i\evec_2}}{Z_{u,(i-1)\evec_2}} \,,
\quad	J^{\rhodown{\rho}(o_c)}_i=\frac{Z^{\rhodown{\rho}(o_c)}_{o_c,i\evec_2}}{Z^{\rhodown{\rho}(o_c)}_{o_c,(i-1)\evec_2}} 
	\qquad\text{and}\qquad 
	J^{\rhoup{\rho}(o_c)}_i=\frac{Z^{\rhoup{\rho}(o_c)}_{o_c,i\evec_2}}{Z^{\rhoup{\rho}(o_c)}_{o_c,(i-1)\evec_2}}. 
\ee  

	
	Recall that $\cJ=\lzb -N^\frac{2}{3}\evec_2,N^\frac{2}{3}\evec_2 \rzb$. 
	
	\begin{lemma}\label{lem ge1} 
		For $0<y<1$, define  the event
		\be\label{Aod} 
		A_{o_c,\dvec,y} 
			= \left\{\, \inf_{x\in\cJ}Q^{\rhodown{\rho}(o_c)}_{o_c,\tsp x}\big(\ex_{o_c,\tsp x}<-d_1N^{\frac{2}{3}}\big)\ge 1-y\, , \; \inf_{x\in\cJ}Q^{\rhoup{\rho}(o_c)}_{o_c,\tsp x}\big(\ex^{\rhoup{\rho}(o_c)}_{o_c,\tsp x}>  d_1N^{\frac{2}{3}}\big)\ge 1-y \right\}.  
		\ee 
		Under the assumptions of Lemma \ref{lem-lb1} for $d=d_1$ we have the bound 
		\be\label{Aod2}   \P\bigl(A_{o_c,\dvec,y}\bigr) \ge 1-C_1(\e)y^{-1} r^{-3}. \ee 
		On the event $A_{o_c,\dvec,y}$, for any $m,n\in\lzb-N^{2/3}, N^{2/3}\rzb$ such that $m< n$ we have the inequalities 
		\be\label{Aod1} 
		(1-y)\prod_{i=m+1}^{n}J^{\rhoup{\rho}(o_c)}_i \; \leq \; \prod_{i=m+1}^{n}J^u_i \; \leq \; \frac{1}{1-y}\prod_{i=m+1}^{n}J^{\rhodown{\rho}(o_c)}_i\quad \forall u\in \cI_{o,d}.
		\ee 
	\end{lemma}
	\begin{proof}    Bound \eqref{Aod2} comes by switching to complements in  Lemma \ref{lem-lb1}. 
		We show the second inequality of \eqref{Aod1}.  The  first inequality follows similarly.  
Let $u\in \cI_{o,\dvec}$.   The first inequality in the calculation \eqref{Aod13} below is justified as follows in two cases.  Recall the notation \eqref{h:Z(A)} for restricted partition functions $Z_{o,p}(A)$. 

\medskip

(i)   Suppose $u=o_c+j\evec_2$ for some $0\le j\le d_1N^{2/3}$.  
 Apply \eqref{m:770} in the following  setting.  Take $Z^{(2)}_{u, \cbullet}$   to be $Z_{u, \cbullet}$.  Let  $Z^{(1)}_{u, \cbullet}$ use the same bulk weights $Y$. On the  boundary $Z^{(1)}_{u, \cbullet}$ takes 
	$Y^{(1)}_{u+\ell\evec_2}=J^{\rhodown{\rho}(o_c)}_{u+\ell\evec_2}$ on the $y$-axis, and  on the $x$-axis takes any  $Y^{(1)}_{u+m\evec_1}< Y_{u+m\evec_1}$ for $1\le m\le -u\cdot\evec_1$.    Then the second inequality of \eqref{m:770} followed by the second inequality of \eqref{m:772}  gives 
\[  	\frac{Z_{u,i\evec_2}}{Z_{u,(i-1)\evec_2}}    \leq \frac{Z^{(1)}_{u,i\evec_2}}{Z^{(1)}_{u,(i-1)\evec_2}}    \leq  \frac{Z^{(1)}_{u,i\evec_2}\bigl(\ex_{u,i\evec_2}< j-d_1N^{\frac{2}{3}}\bigr)}{Z^{(1)}_{u,(i-1)\evec_2}\bigl(\ex_{u, (i-1)\evec_2}< j-d_1N^{\frac{2}{3}}\bigr)}. 
\] 
Next observe  that the condition  
$\ex_{u,\brbullet}< j-d_1N^{\frac{2}{3}}<0$ renders the boundary weights on the $x$-axis  $u+(\Z_{>0})\evec_1$ irrelevant. Therefore  we can replace $Y^{(1)}_{u+m\evec_1}$  with the stationary boundary weights	 $I^{\rhodown{\rho}(o_c)}_{u+m\evec_1}$ without changing the restricted partition functions on the right-hand side.   This gives the first equality below: 
\begin{align*}
\frac{Z^{(1)}_{u,i\evec_2}\bigl(\ex_{u,i\evec_2}< j-d_1N^{\frac{2}{3}}\bigr)}{Z^{(1)}_{u,(i-1)\evec_2}\bigl(\ex_{u, (i-1)\evec_2}< j-d_1N^{\frac{2}{3}}\bigr)}
&=  \frac{Z^{\rhodown{\rho}(o_c)}_{u,i\evec_2}\bigl(\ex_{u,i\evec_2}< j-d_1N^{\frac{2}{3}}\bigr)}{Z^{\rhodown{\rho}(o_c)}_{u,(i-1)\evec_2}\bigl(\ex_{u,(i-1)\evec_2}< j-d_1N^{\frac{2}{3}}\bigr)}\\
&=
 \frac{Z^{\rhodown{\rho}(o_c)}_{o_c,i\evec_2}\bigl(\ex_{o_c,i\evec_2}<-d_1N^{\frac{2}{3}}\bigr)}{Z^{\rhodown{\rho}(o_c)}_{o_c,(i-1)\evec_2}\bigl(\ex_{o_c,(i-1)\evec_2}<-d_1N^{\frac{2}{3}}\bigr)}.  
\end{align*} 
The second equality comes by multiplying upstairs and downstairs with the boundary weights $J^{\rhodown{\rho}(o_c)}_{o_c+\ell\evec_2}$	for $1\le\ell\le j= (u-o_c)\cdot\evec_2$. 
	
%

\medskip 

(ii) On the other hand,  if $u=o_c+k\evec_1$ for some $0\le k\le d_1N^{2/3}$, then first by \eqref{in2} and then by applying the argument of the previous paragraph to $u=o_c$: 
\begin{align*}
\frac{Z_{u,i\evec_2}}{Z_{u,(i-1)\evec_2}}\leq 
 \frac{Z_{o_c,i\evec_2}}{Z_{o_c,(i-1)\evec_2}}
\le 
 \frac{Z^{\rhodown{\rho}(o_c)}_{o_c,i\evec_2}\bigl(\ex_{o_c,i\evec_2}<-d_1N^{\frac{2}{3}}\bigr)}{Z^{\rhodown{\rho}(o_c)}_{o_c,(i-1)\evec_2}\bigl(\ex_{o_c,(i-1)\evec_2}<-d_1N^{\frac{2}{3}}\bigr)}.  
\end{align*} 

\medskip 

Now for the derivation. 
	\be\label{Aod13}\begin{aligned}
			\prod_{i=m+1}^{n}J^u_i&=\prod_{i=m+1}^{n}\frac{Z_{u,i\evec_2}}{Z_{u,(i-1)\evec_2}}\leq \prod_{i=m+1}^{n}\frac{Z^{\rhodown{\rho}(o_c)}_{o_c,i\evec_2}\bigl(\ex_{o_c,i\evec_2}<-d_1N^{\frac{2}{3}}\bigr)}{Z^{\rhodown{\rho}(o_c)}_{o_c,(i-1)\evec_2}\bigl(\ex_{o_c,(i-1)\evec_2}<-d_1N^{\frac{2}{3}}\bigr)} \\
&=
 \prod_{i=m+1}^{n}
\frac{Q^{\rhodown{\rho}(o_c)}_{o_c,i\evec_2}\big(\ex_{o_c,i\evec_2}< -d_1N^{\frac{2}{3}}\big)}{Q^{\rhodown{\rho}(o_c)}_{o_c,(i-1)\evec_2}\big(\ex_{o_c,(i-1)\evec_2}<-d_1N^{\frac{2}{3}}\big)}
\cdot  \prod_{i=m+1}^{n}\frac{Z^{\rhodown{\rho}(o_c)}_{o_c,i\evec_2}}{Z^{\rhodown{\rho}(o_c)}_{o_c,(i-1)\evec_2}}\\
&=   \frac{Q^{\rhodown{\rho}(o_c)}_{o_c,n\evec_2}\big(\ex_{o_c,n\evec_2}< -d_1N^{\frac{2}{3}}\big)}{Q^{\rhodown{\rho}(o_c)}_{o_c,m\evec_2}\big(\ex_{o_c,m\evec_2}<-d_1N^{\frac{2}{3}}\big)}  \prod_{i=m+1}^{n}J^{\rhodown{\rho}(o_c)}_i  
\le \frac{1}{1-y}\prod_{i=m+1}^{n}J^{\rhodown{\rho}(o_c)}_i . 
		\end{aligned}\ee 
	\end{proof}

	Next we define the analogous construction reflected across the origin. Define east ($\cE$) and north ($\cN$) portions of the boundary by  $\partial^N_\mathcal{E} =\{N\}\times \lzb \e N,N \rzb$ and  $\partial^N_\mathcal{N}=\lzb \e N,N\rzb \times \{N\}$, and combine them into  $\rim{\partial}^N=\rim{\partial}^{N\!,\tspa\e}=\partial^N_\mathcal{E}\cup \partial^N_\mathcal{N}$. 
	  Each point $\rim{o}=(\rim{o}_1,\rim{o}_2)\in \rim{\partial}^N$ is associated with a parameter $\rho(\rim{o})\in (0,1)$ and a direction $ \xi(\rim{o})\in \,]\evec_2,\evec_1[\,$   through the relations in \eqref{rho-o} and \eqref{xi-o}. 
	For each point  $\rim{o}\in \rim{\partial}^N$ define the set 
	\begin{align*}
	\rim{\cI}_{\rim{o},\dvec}=\bigl\{v \in \rim{\partial}^N:\distance(v,\rim{o})\leq \tfrac12{d_2}N^\frac{2}{3}\bigr\} 
	\end{align*}
	and the maximal point  $\rim{o}_c\in 	\rim{\cI}_{\rim{o},\dvec}$
  in the coordinatewise partial order,  defined by the requirement that 
\[  \rim o_c\in\rim\cI_{\rim o,\dvec} \quad\text{and}\quad   v\le\rim o_c \ \ \forall v\in\rim \cI_{\rim o,\dvec}. \] 	
	
As previously for sets $\cI_{o,\dvec}$ on the southwest boundary, given now a northeast boundary point $\rim o\in\rim\partial^N$ 	we construct a family of coupled backward partition functions from $\rim{\cI}_{\rim{o},\dvec}$ to points  on the shifted $y$-axis $\evec_1+\Z\evec_2$.    From each 
 	  $v\in \rim{\cI}_{\rim{o},\dvec}$ we have  the backward bulk partition functions $\rim{Z}_{v,\bbullet}$  that use the i.i.d.\ Ga$^{-1}(1)$ weights $Y$.  From the base point $\rim o_c$ we define two stationary backward polymer processes $\rim{Z}^{\rhodown{\rho}(\rim{o}_c)}_{\rim{o}_c,\bbullet}$ and $\rim{Z}^{\rhoup{\rho}(\rim{o}_c)}_{\rim{o}_c, \bbullet}$ with parameters  $\rhodown{\rho}(\rim{o}_c)=\rho(\rim{o}_c)-rN^{-\frac{1}{3}}$ and $\rhoup{\rho}(\rim{o}_c)=\rho(\rim{o}_c)+rN^{-\frac{1}{3}}$.   Weights are coupled on the northeast boundary according to 
	  Theorem \ref{thm:st-lpp}:  for $k,\ell\ge 1$, 
\be\label{2168} 
Y_{\rim o_c-k\evec_1} \le I^{\rhodown{\rho}(\rim o_c)}_{\rim o_c-k\evec_1}\le I^{\rhoup{\rho}(\rim o_c)}_{\rim o_c-k\evec_1}
\quad\text{ and }\quad 
Y_{\rim o_c-\ell\evec_2} \le J^{\rhoup{\rho}(\rim o_c)}_{\rim o_c-\ell\evec_2}\le J^{\rhodown{\rho}(\rim o_c)}_{\rim o_c-\ell\evec_2}. 
\ee	
The boundary weights in \eqref{2160} and  in \eqref{2168} above are taken independent of each other.  

Ratio weights on the shifted $y$-axis are defined by
\be\label{2172} 
\rim{J}^{\tspb v}_i=\frac{\rim{Z}_{v,\evec_1+(i-1)\evec_2}}{\rim{Z}_{v,\evec_1+i\evec_2}}\,, 
\quad
 	\rim{J}^{\tspb\rhodown{\rho}(\rim{o}_c)}_i=\frac{\rim{Z}^{\rhodown{\rho}(\rim{o}_c)}_{\rim{o}_c,\,\evec_1+(i-1)\evec_2}}{\rim{Z}^{\rhodown{\rho}(\rim{o}_c)}_{\rim{o}_c,\,\evec_1+i\evec_2}}
\quad\text{and}\quad 	
	\rim{J}^{\tspb\rhoup{\rho}(\rim{o}_c)}_i=\frac{\rim{Z}^{\rhoup{\rho}(\rim{o}_c)}_{\rim{o}_c,\,\evec_1+(i-1)\evec_2}}{\rim{Z}^{\rhoup{\rho}(\rim{o}_c)}_{\rim{o}_c,\,\evec_1+i\evec_2}}. 
\ee
The collection of ration weights  in \eqref{2165} is independent of the collection   in \eqref{2172} above   because they are constructed from independent inputs.  
 
	
	 We have this analogue of  Lemma \ref{lem ge1}.    $\rim{\cJ}=\evec_1+\cJ= \lzb \evec_1-N^\frac{2}{3}\evec_2\,, \,\evec_1+N^\frac{2}{3}\evec_2 \rzb$ is the shift of the interval $\cJ$ in \eqref{Aod}.  
	\begin{lemma}\label{lem ge2} 
		For $0<y<1$, define  the event
		\be\label{Bod} 
		B_{\rim{o}_c,\dvec,y}
		 = \left\{\,\inf_{x\in \wh{\cJ}}\rim{Q}^{\rhodown{\rho}(\rim{o}_c)}_{\rim{o}_c,x}\big(\rim{\ex}_{\rim{o}_c,x}<-d_2N^{\frac{2}{3}}\big)\ge 1- y , \; \inf_{x\in\wh{\cJ}}\rim{Q}^{\rhoup{\rho}(\rim{o}_c)}_{\rim{o}_c,x}\big(\hat{\ex}_{\rim{o}_c,x}> d_2N^{\frac{2}{3}}\big)\ge 1- y \right\}. 
		\ee 
		Under the assumptions of Lemma \ref{lem-lb1} for $d=d_2$  we have the bound 
		\be\label{Bod2}   \P\bigl(B_{\rim{o}_c,\dvec,y}\bigr) \ge 1-C_1(\e)y^{-1} r^{-3}. \ee 
		On the event $B_{\rim{o}_c,\dvec,y}$, for any $m<n$ in $\lzb-N^{2/3}, N^{2/3}\rzb$   we have the inequalities 
		\be\label{Bod1} 
		(1-y)\prod_{i=m+1}^{n}\rim{J}^{\tspb\rhoup{\rho}(\rim{o}_c)}_i\leq \prod_{i=m+1}^{n}\rim{J}^{\tspb v}_i \leq \frac{1}{1-y}\prod_{i=m+1}^{n}\rim{J}^{\tspb\rhodown{\rho}(\rim{o}_c)}_i\quad \forall v\in \rim{\cI}_{\rim{o}_c,\dvec}.
		\ee 
	\end{lemma}
	
\smallskip 	
	
Now we use partition functions from the southwest and northeast together. 		Let $o\in \partial^N,\rim{o}\in\rim{\partial}^N$ and consider the polymers  from points $u\in \cI_{o,\dvec}$ to the interval $\cJ$ on the $y$-axis and reverse polymers from   points $v\in \rim{\cI}_{\rim{o},\dvec}$ to the shifted interval $\wh{\cJ}=\evec_1+\cJ$. Abbreviate the parameters for the base points as 
\be\label{2186}   \rhoup{\rho}=\rhoup{\rho}(o_c), \quad \rhodown{\rho}=\rhodown{\rho}(o_c), \quad 
 \rhoup{\lambda}=\rhoup{\rho}(\rim{o}_c), \quad\text{and}\quad  \rhodown{\lambda}=\rhodown{\rho}(\rim{o}_c).  \ee
	For $i\in\lzb-N^{2/3}, N^{2/3}\rzb $, take the $Z$-ratios from \eqref{2165} and  \eqref{2172} and define 
\be\label{2188}
	X^{u,v}_i=\frac{J^u_i}{\rim{J}^{\tspb v}_i}, \quad 
	Y'_i=\frac{J^{\rhodown{\rho}}_i}{\rim{J}^{\tspb\rhoup{\lambda}}_i}
	\quad\text{and}\quad 
	Y_i=\frac{J^{\rhoup{\rho}}_i}{\rim{J}^{\tspb\rhodown{\lambda}}_i}.
\ee
		A two-sided multiplicative walk $\mwalk(X)$ with steps $\{X_j\}$ is defined  by 
	\begin{align}\label{mw}
	\mwalk_n(X)=
	\begin{cases} 
	\prod_{j=1}^{n}X_j & n \geq 1 \\
	1 & n=0\\
	\prod_{j=n+1}^{0}X^{-1}_j & n \le-1. 
	\end{cases}
	\end{align}
	The ratios from \eqref{2188} above  define the walks 
\be\label{2192}
	\mwalk^{u,v}=\mwalk(X^{u,v})\,,\quad 
	\mwalk'=\mwalk(Y{'}) \quad\text{and}\quad 
	\mwalk=\mwalk(Y). 
	\ee 
	Specialize the parameter $y$ in the events   in \eqref{Aod} and \eqref{Bod} to set 
	\begin{align*}
		A_{o,\dvec}=A_{o,\dvec,\frac{\sqrt{2}-1}{\sqrt{2}}} 
		\qquad\text{and}\qquad 
		B_{\rim{o},\dvec}=B_{\rim{o},\dvec,\frac{\sqrt{2}-1}{\sqrt{2}}}.
	\end{align*}
	
	\begin{lemma}\label{lm:sw}
		The  processes 
		\be \label{cor b-rw1}
		\{\mwalk'_m :  m\in \lzb-N^{2/3},0\rzb\,\}  \quad\text{and}\quad 
		\{\mwalk_n: n\in \lzb0,N^{2/3} \rzb\,\} 
		\quad\text{are independent.} 
		\ee 
		On the event $A_{o,\dvec}\cap B_{\rim{o},\dvec}$, for all $u\in \cI_{o,\dvec}$ and $v\in \rim{\cI}_{\rim{o},\dvec}$,  
		\be\label{sw5} \begin{aligned}
		\tfrac12 \mwalk'_n &\leq \mwalk_n^{u,v}\leq 2\mwalk_n \quad\text{for } \   n\in \lzb-N^\frac{2}{3},-1\rzb\\ 
				\text{and}\quad 	\tfrac12 \mwalk_n &\leq \mwalk_n^{u,v}\leq 2\mwalk'_n \quad\text{for } \  n\in \lzb1,N^\frac{2}{3}\rzb. 
		\end{aligned}\ee
	\end{lemma}
	
	\begin{proof} 	To prove the  independence claim  \eqref{cor b-rw1}, observe first from the construction itself that 
the collection $\{J^{\rhodown{\rho}}_i,  J^{\rhoup{\rho}}_i\}_{i\tspa\in\tspa\lzb -N^{2/3}, N^{2/3}\rzb}$ is independent of 	 the collection $\{\rim J^{\tspb\rhodown{\lambda}}_i,  \rim J^{\tspb\rhoup{\lambda}}_i\}_{i\tspa\in\tspa\lzb -N^{2/3}, N^{2/3}\rzb}$, as pointed out below \eqref{2172}.   Then within these collections, Theorem \ref{thm:st-lpp}(i)  implies the independence of 
$\{J^{\rhodown{\rho}}_i\}_{i\le0}$ and  $\{J^{\rhoup{\rho}}_i\}_{i\ge1}$, and the independence of 
$\{\rim J^{\tspb\rhodown{\lambda}}_i\}_{i\ge 1}$ and  $\{\rim J^{\tspb\rhoup{\lambda}}_i\}_{i\le0}$.   With boundary weights on the southwest, the independence of  $\{J^{\rhodown{\rho}}_i\}_{i\le0}$ and  $\{J^{\rhoup{\rho}}_i\}_{i\ge1}$ is a direct application of Theorem \ref{thm:st-lpp}(i) with the choice $(\lambda, \rho, \sigma)=(\rhodown{\rho}, \rhoup{\rho}, 1)$.    After reflection of the entire setting of Theorem \ref{thm:st-lpp} across its base point $u$, the boundary weights reside on the northwest, as required for $\{\rim J^{\tspb\rhodown{\lambda}}_i\}_{i\ge 1}$ and  $\{\rim J^{\tspb\rhoup{\lambda}}_i\}_{i\le0}$, and  the  direction $\evec_2$ has been reversed   to $-\evec_2$.  Hence the inequalities $i\le0$ and $i\ge1$ in the independence statement must be switched around. 

To summarize, the collections $\{J^{\rhodown{\rho}}_i, \rim J^{\tspb\rhoup{\lambda}}_i\}_{i\le0}$  and $\{J^{\rhoup{\rho}}_i, \rim J^{\tspb\rhodown{\lambda}}_i\}_{i\ge1}$ are independent of each other, which implies the independence of  $\{Y'_i\}_{i\le0}$  from $\{Y_i\}_{i\ge1}$.  
	
\medskip

We  show the case 	$n\in \lzb 1, N^{2/3}\rzb$ of  \eqref{sw5}. 
		\begin{align*}
			\mwalk_n^{u,v}=\prod_{i=1}^{n}X^{u,v}_i=\prod_{i=1}^{n}J^u_i\cdot \prod_{i=1}^{n}(\rim{J}^{\tspb v}_i)^{-1}
	\begin{cases} 		
			\leq \sqrt{2}\prod_{i=1}^{n}J^{\rhodown{\rho}}_i \cdot \sqrt{2}\prod_{i=1}^{n}(\rim{J}^{\tspb\rhoup{\lambda}}_i)^{-1}=2\prod_{i=1}^{n} Y'_i = 2\mwalk'_n\,;   \\[5pt] 
			\geq \frac 1{\sqrt{2}}\prod_{i=1}^{n}J^{\rhoup{\rho}}_i \cdot \frac 1{\sqrt{2}}\prod_{i=1}^{n}(\rim{J}^{\tspb\rhodown{\lambda}}_i)^{-1}=\frac12\prod_{i=1}^{n} Y_i = \frac12 \mwalk_n. 
			\end{cases} 
		\end{align*}
An analogous argument gives the case $n\in \lzb -N^{2/3}, -1\rzb$. 		
	\end{proof}

	Each path that crosses the $y$-axis leaves the axis along a unique  edge $\ed_i=(i\evec_2,i\evec_2+\evec_1)$.  
Decompose the set of paths between $u\in\partial^N$ and $v\in\rim\partial^N$ according to the edge taken:  
\[  \pathsp_{u,v}=\bigcup_{i\in \Z} \pathsp^i_{u,v}   
  \]
where the sets 
	\begin{align}\label{tz}
		\pathsp^i_{u,v}=\{\pi\in\pathsp_{u,v}:\ed_i\in\pi\}    
	\end{align}
  satisfy $\pathsp^i_{u,v}\cap \pathsp^j_{u,v}=\emptyset $ for $i\neq j$. 
%
Let
	\begin{align}\label{pi}
		p^{u,v}_i=Q_{u,v}(\pathsp^i_{u,v})=\frac{Z_{u,i\evec_2}Z_{i\evec_2+\evec_1,v}}{Z_{u,v}}
	\end{align}
	be the quenched probability of paths going through the edge $\ed_i$. 
For all $n\in\lzb-N^{2/3}, N^{2/3}\rzb$   we claim that 
	\be\label{pi2}
		p_0^{u,v}\leq (\mwalk^{u,v}_n)^{-1}.
	\ee	
Here is the verification for 	$n\ge1$: 
	\be\begin{aligned}\nn
		p^{u,v}_0&\leq \frac{p^{u,v}_0}{p^{u,v}_n}  
		=\frac{Z_{u,\zevec}Z_{\evec_1,v}}{Z_{u,n\evec_2}Z_{n\evec_2+\evec_1,v}}
		= \prod_{i=1}^n\frac{Z_{u,(i-1)\evec_2}Z_{(i-1)\evec_2+\evec_1,v}}{Z_{u,i\evec_2}Z_{i\evec_2+\evec_1,v}}  
		= \prod_{i=1}^n\frac{\rim{J}^{\tspb v}_i}{J^u_i}= \prod_{i=1}^n (X^{u,v}_i)^{-1}=(\mwalk^{u,v}_n)^{-1}.  
	\end{aligned}\ee
The case 	$n\le-1$ goes similarly. 
	
	

We come to the key estimates.  The first one controls the quenched probability of paths between   $\cI_{o,\dvec}$ and  $\rim{\cI}_{\rim{o},\dvec}$ that go through the edge $\ed_0$ from $\zevec$ to $\evec_1$.  

	\begin{lemma}\label{lm:close}
		Let $r=N^{\frac2{15}}$ and $\dvec=(d_1,d_2)=(1,N^\frac18)$.   There exist finite positive constants  $C(\e)$ and $N_0(\e)$ such that,  for all $N\ge N_0(\e)$ and  $o\in \partial^N$ with $\rim o=-o$,   
		\begin{align*}
			\P\Big(\sup_{u\,\in\,\cI_{o,\dvec},\,v\,\in\,\rim{\cI}_{\rim{o},\dvec}}p^{u,v}_0>N^{-1}\Big)\leq C(\e) (\log N)^6N^{-2/5}. 
		\end{align*}
	\end{lemma}
\begin{proof}
	 For any $u\in\cI_{o,\dvec}$ and $v\in\rim{\cI}_{\rim{o},\dvec}$, by  \eqref{pi2} and  \eqref{sw5}, 
	\be\begin{aligned}\label{ine}
		&\{p^{u,v}_0>N^{-1}\}\cap \{A_{o,\dvec}\cap B_{\rim{o},\dvec}\}
		\subseteq \{\max_{n\tspa\in\tspa\lzb-N^{2/3}, N^{2/3}\rzb}\mwalk^{u,v}_n<N\}\cap (A_{o,\dvec}\cap B_{\rim{o},\dvec}) \\
		&\qquad\qquad\qquad
		\subseteq\bigl\{ \max_{-N^{2/3}\leq n\leq-1 }\mwalk'_n< 2N,\, \max_{1\leq n\leq N^{2/3}}\mwalk_n<2N\bigr\}\cap (A_{o,\dvec}\cap B_{\rim{o},\dvec}).
	\end{aligned}\ee
By the independence in \eqref{cor b-rw1}, 
	\be\label{2205} \begin{aligned}
		\P\bigl(\max_{u\,\in\,\cI_{o,\dvec},\,v\,\in\,\rim{\cI}_{\rim{o},\dvec}}p^{u,v}_0>N^{-1}\bigr)
		&\leq \P\bigl(\,\max_{-N^{2/3}\leq n\leq-1 }\mwalk'_n< 2N\bigr)\, \P\bigl(\,\max_{1\leq n\leq N^{2/3}}\mwalk_n<2N\bigr)\\
		&\qquad\qquad
		+2\tspa\P(A_{o,\dvec}^c\cup B_{\rim{o},\dvec}^c).
	\end{aligned}\ee
	
To apply the random walk bound from Appendix \ref{sec:rw}, we convert the multiplicative walks into additive walks. For given steps   $\xi=\{\xi_i\}$  define the two-sided walk $S(\xi)$ by 
	 \begin{align*}
	 \swalk_n(\xi)=
	 \begin{cases} 
	 \sum_{i=1}^{n}\xi_i & n \geq 1 \\
	 0 & n=0\\
	 -\sum_{i=n+1}^{0}\xi_i & n < 0. 
	 \end{cases}
	 \end{align*}
	 Recall the parameters defined in \eqref{2186}. 
With reference to \eqref{2188} and \eqref{2192}, define the additive walks 	 	\begin{align*}
		\swalk_n&=\log\mwalk_n  
		\qquad\text{with steps }
		 \   \xi_i=\log{J^{\rhoup{\rho}}_i}-\log {\rim{J}^{\tspb\rhodown{\lambda}}_i} \,, \\
		\swalk'_n&=\log\mwalk'_n  \qquad\text{with steps }
		 \   \xi'_i=\log{J^{\rhodown{\rho}}_i}-\log {\rim{J}^{\tspb\rhoup{\lambda}}_i}. 
	\end{align*}
With the bounds \eqref{Aod2} and  \eqref{Bod2}, \eqref{2205} becomes 
	\be\label{2208} \begin{aligned}
		\P\bigl(\max_{u\,\in\,\cI_{o,\dvec},\,v\,\in\,\rim{\cI}_{\rim{o},\dvec}}p^{u,v}_0>N^{-1}\bigr)
		&\leq \P\bigl(\,\max_{-N^{2/3}\leq n\leq-1 }\swalk'_n< \log(2N)\bigr)\, \P\bigl(\,\max_{1\leq n\leq N^{2/3}}\swalk_n<\log(2N)\bigr)\\
		&+ Cr^{-3}.
	\end{aligned}\ee

We  use Theorem \ref{thm:lm2} to bound $\P(\,\max_{1\leq n\leq N^{2/3}}\swalk_n<\log(2N))$.  
Since 
\[
\rhoup{\rho}= \rho(o_c)+rN^{-1/3}= \rho(o_c)+N^{-1/5}
\quad\text{and}\quad 
\rhodown{\lambda}= \rho(\rim{o}_c) -rN^{-1/3}=\rho(\rim{o}_c) -N^{-1/5}, 
\]
we can establish constants $0<\rho_{\rm min}<\rho_{\rm max}<1$ and  $N_0(\e)\in\Z_{>0}$ such that  
$\rhoup{\rho},  \rhodown{\lambda} \in [\rho_{\rm min},\rho_{\rm max}]$ for all $o\in\partial^N$ and $N\ge N_0(\e)$.  
As $|o-o_c|\leq \tfrac12d_1N^{2/3}$ and  $|\rim{o}-\rim{o}_c|\leq \tfrac12d_2N^{2/3}$, the restriction of the slope to  $[\e, \e^{-1}]$    implies  that there is a constant $C=C(\e)$ such that
	\begin{align*}
	\abs{\tspa \rho(o_c)-\rho(o)\tspa} \leq Cd_1N^{-1/3}
	\qquad\text{and}\qquad 
		\abs{\tspa  \rho(\rim{o}_c)-\rho(\rim{o})} \leq Cd_2N^{-1/3}.
	\end{align*}
Then,  since $\rho(o)=\rho(-o)=\rho(\rim{o})$, 
\begin{align*}
 \abs{\tspb \rho(\rim{o}_c)-\rho(o_c)\tsp}
 \le  \abs{\tspa  \rho(\rim{o}_c)-\rho(\rim{o})}  
 +  \abs{\tspa \rho(o_c)-\rho(o)\tspa} \leq  Cd_2N^{-1/3}  + Cd_1N^{-1/3}
 \le CN^{-5/24}. 
 \end{align*} 
 Hence 
 \begin{align*} 
\rhodown{\lambda} - \rhoup{\rho}= \rho(\rim{o}_c)-\rho(o_c)-2rN^{-1/3} 
 \begin{cases} 
 \le  -2N^{-1/5} (  1 - CN^{-1/120}\tspb)  \\[3pt] 
  \ge  -2N^{-1/5} (  1 + CN^{-1/120}\tspb) .
  \end{cases} 
\end{align*} 
We conclude that for $N\ge N_0(\e)$,    the mean step of $S_n$  satisfies 
\begin{align*}
\E(S_1)= \E\bigl[ \log{J^{\rhoup{\rho}}_i}-\log {\rim{J}^{\tspb\rhodown{\lambda}}_i}\bigr] 
=\psi_0(\rhodown{\lambda}) -\psi_0(\rhoup{\rho}) 
\in [-CN^{-/1/5}, 0]  
\end{align*} 
where the (new) constant $C=C(\e)$ works for all $o\in\partial^N$.

	In Theorem \ref{thm:lm2} set $x=(\log N)^2$ 
	to conclude that for $N\ge N_0(\e)$ 
	\begin{align}\label{la}
		\P\bigl\{\sup_{1\leq n\leq N^{2/3}}2\swalk_n<(\log N)^2\bigr\}\leq C(\log N)^3 N^{-1/5} . 
	\end{align} 
	This bound with the same constant $C=C(\e) $ works for all points $o\in\partial^N$ and all  $N\ge N_0(\e)$. 
	Similarly one can show that
	\begin{align}\label{la2}
		\P\bigl\{\sup_{-N^{2/3}\leq n\leq-1 }2\swalk'_n<(\log N)^2\bigr\}\leq 
		C(\log N)^3 N^{-1/5}.  
	\end{align}
The lemma follows by inserting these bounds and $r=N^{2/15}$ into  \eqref{2208}. 
\end{proof}

		
	 
The next lemma controls the quenched probability of paths from points $u\in\cI_{o,\dvec}$ that go through the edge $\ed_0$ from $\zevec$ to $\evec_1$  but miss the interval $\rim{\cI}_{\rim{o},\dvec}$ on the northeast side of the square $\lzb-N,N\rzb^2$.  
	The complement of $\rim{\cI}_{\rim{o},\dvec}$ on $\rim{\partial}^N$ is denoted by 
	\[
	\rim{\cF}_{\rim{o},\dvec}=\{v\in \rim{\partial}^N: \abs{v-\rim{o}}_1>\tfrac12{d_2}N^{\frac{2}{3}}\}. 
	\]

	\begin{figure}
		\begin{subfigure}{.5\textwidth}
			\begin{tikzpicture}[scale=0.5, every node/.style={transform shape}]
			\draw  (0,0) rectangle (10,10);
			\draw [dashed][line width=0.01cm] (0,5) -- (10,5);
			\draw [dashed][line width=0.01cm] (5,0) -- (5,10);
			\foreach \x in {0,...,6}
			{   \draw [fill] (0,\x*1/5) circle [radius=0.07];
			};
			\foreach \x in {0,...,3}
			{   \draw [fill] (\x*1/5,0) circle [radius=0.07];
			};
			\node [scale=2][above] at (1.2,0.3) {$\cI_{o,\dvec}$};
			\node [scale=2][above][red] at (-0.5,-0.7) {$o_c$};
			\node [scale=2][above][cyan] at (-0.5,0.3) {$o$};
			\draw [fill] (0,3/5) circle [radius=0.07][cyan];
			\draw [fill] (0,0) circle [radius=0.07][red];
			\draw [dashed][line width=0.01cm] (0,1/5) -- (0,3/5);
			\foreach \x in {0,...,8}
			{   \draw [fill] (10,10-\x*1/5) circle [radius=0.07];
			};
			\foreach \x in {0,...,8}
			{   \draw [fill] (10-\x*1/5,10) circle [radius=0.07];
			};
			\draw [fill] (10,10) circle [radius=0.07][red];
			\draw [fill] (10,9.4) circle [radius=0.07][cyan];
			\node [scale=2][above][red] at (10,10) {$\rim{o}_c$};
			\node [scale=2][above][cyan] at (10.5,9) {$\rim{o}$};
			\node [scale=2][above] at (8.9,8) {$\rim{\cI}_{\rim{o},\dvec}$};
			\node [scale=2][above][blue] at (6.8,8.2) {$\rim{\cF}^1_{\rim{o},\dvec}$};
			\node [scale=2][above][blue] at (8.8,5.6){$\rim{\cF}^2_{\rim{o},\dvec}$}; 
			
			\foreach \x in {9,...,23}
			{   \draw [fill] (10-\x*1/5,10) circle [radius=0.07][blue];
			};
			\foreach \x in {9,...,23}
			{   \draw [fill] (10,10-\x*1/5) circle [radius=0.07][blue];
			};
			\end{tikzpicture}
		\end{subfigure}%
		\begin{subfigure}{.5\textwidth}
			\begin{tikzpicture}[scale=0.5, every node/.style={transform shape}]
			\draw  (0,0) rectangle (10,10);
			\draw [dashed][line width=0.01cm] (0,5) -- (10,5);
			\draw [dashed][line width=0.01cm] (5,0) -- (5,10);
			\foreach \x in {1,...,9}
			{   \draw [fill] (0,2+\x*1/5) circle [radius=0.07];
			};
			\node [scale=2][above] at (1.2,2) {$\cI_{o,\dvec}$};
			\node [scale=2][above][red] at (-0.6,1.7) {$o_c$};
			\node [scale=2][above][cyan] at (-0.6,2.6) {$o$};
			\draw [fill] (0,3) circle [radius=0.07][cyan];
			\draw [fill] (0,2.2) circle [radius=0.07][red];
			\draw [dashed][line width=0.01cm] (0,1/5) -- (0,3/5);
			\foreach \x in {1,...,9}
			{   \draw [fill] (10,6+\x*1/5) circle [radius=0.07];
			};
			\draw [fill] (10,7.8) circle [radius=0.07][red];
			\draw [fill] (10,7) circle [radius=0.07][cyan];
			\node [scale=2][above][red] at (10.6,7.3) {$\rim{o}_c$};
			\node [scale=2][above][cyan] at (10.6,6.5) {$\rim{o}$};
			\node [scale=2][above] at (9,6.5) {$\rim{\cI}_{\rim{o},\dvec}$};
			\node [scale=2][above][blue] at (7.2,8.2) {$\rim{\cF}^1_{\rim{o},\dvec}$};
			\node [scale=2][above][blue] at (8.8,4.8) {$\rim{\cF}^2_{\rim{o},\dvec}$};
			\foreach \x in {0,...,23}
			{   \draw [fill] (10-\x*1/5,10) circle [radius=0.07][blue];
			};
			\foreach \x in {1,...,4}
			{   \draw [fill] (10,5.2+\x*1/5) circle [radius=0.07][blue];
			};
			\foreach \x in {1,...,10}
			{   \draw [fill] (10,10-\x*1/5) circle [radius=0.07][blue];
			};		
			\end{tikzpicture}
		\end{subfigure}
		\caption{\small The square  $\lzb-N,N\rzb^2$ with  two possible arrangements of the segments ${\cI}_{o,\dvec}$, $\rim{\cI}_{\rim{o},\dvec}$ and $\rim{\cF}_{\rim{o},\dvec}=\rim{\cF}^1_{\rim{o},\dvec}\cup\rim{\cF}^2_{\rim{o},\dvec}$ on the boundary of the square. In both cases $\rim{o}=-o$.}  
		\label{fig:points}
	\end{figure}

	\begin{lemma} \label{lm:far}  
		Let $\dvec=(d_1,d_2)=(1,N^\frac{1}{8})$.     There are finite constants $C(\e)$ and $N_0(\e)$ such that, for all  $\delta>0$,  $N\ge N_0(\e)$ and  $o\in \partial^N$ with $\rim{o}=-o\in\rim\partial^N$,  
		\be\label{far7} 
		\P\Big(\sup_{u\,\in\,\cI_{o,\dvec},\,v\,\in\,\rim{\cF}_{\rim{o},\dvec}} p^{u,v}_0>\delta\Big)\leq C(\e)\delta^{-1} N^{-\frac{3}{8}}. 
		\ee
	\end{lemma}
	\begin{proof}
		Define the   sets of boundary points
		\begin{align*}
		&\partial \rim{\cF}_{\rim{o},\dvec}=\{v\in \rim{\cF}_{\rim{o},\dvec}:\exists u\in \rim{\cI}_{\rim{o},\dvec} \text{ such that }  \abs{v-u}_1=1\}\\
		&\partial \cI_{o,\dvec}=\{v\in \cI_{o,\dvec} :\exists u\in \partial^N\setminus \cI_{o,\dvec}  \text{ such that }  \abs{v-u}_1=1\},
		\end{align*}
		Their cardinalities satisfy $1\le |\partial \rim{\cF}_{\rim{o},\dvec}|\leq   |\partial \cI_{o,\dvec}| \leq 2$.  (For example,  $\partial \rim{\cF}_{\rim{o},\dvec}$ is a singleton if $\rim{\cI}_{\rim{o},\dvec}$ contains one of the endpoints $(N, \fl{\e N})$ or $(\fl{\e N}, N)$  of $\rim{\partial}^N$.) 
		We denote  the points of $\partial \rim{\cF}_{\rim{o},\dvec}$ by $q^1,q^2$  and those of $\partial \cI_{o,\dvec}$  by $h^1,h^2$, 
		labeled so that  
	\[   	 q^1\preccurlyeq  \rim{o}  \preccurlyeq  q^2
	\qquad\text{and}\qquad 
	h^2 \preccurlyeq  o_1 \preccurlyeq  h^1. \]
		Geometrically,  starting from the north pole $N\evec_2$ and  traversing the boundary of the square $\lzb-N,N\rzb^2$ clockwise, we meet the points (those that exist) in this order:  $q^1\to \rim{o}\to q^2\to h^1\to o\to h^2$ (Figure \ref{fig:bigdev}). The set $\rim{\cF}_{\rim{o},\dvec}$ can be decomposed into two disjoint sets
		\begin{align*}
			\rim{\cF}_{\rim{o},\dvec}=\rim{\cF}^1_{\rim{o},\dvec}\cup \rim{\cF}^2_{\rim{o},\dvec}
		\end{align*}
		where
		\begin{align*}
			\rim{\cF}^1_{\rim{o},\dvec}=\{v\in \rim{\cF}_{\rim{o},\dvec}:v\og q^1\} \qquad\text{and}\qquad
			 \rim{\cF}^2_{\rim{o},\dvec}=\{v\in \rim{\cF}_{\rim{o},\dvec}: v\succcurlyeq q^2\}. 
		\end{align*}
		We show that
		\be\label{2206} 
			\P\Big(\sup_{u\,\in\,\cI_{o,\dvec},\,v\,\in\,\rim{\cF}^1_{\rim{o},\dvec}} p^{u,v}_0>\delta\Big)\leq C(\e)\delta^{-1}  N^{-\frac{3}{8}}.  
		\ee
		The same bound can be shown for $\rim{\cF}^2_{\rim{o},\dvec}$ and the lemma follows from a union bound.
		
		Recall the definition of $\pathsp^i_{u,v}$ in  \eqref{tz} and define the set
		\begin{align}\label{pm}
			\pathsp^-_{u,v}=\bigcup_{i\leq 0}\pathsp^i_{u,v}.
		\end{align}
		 For all $u\in\cI_{o,\dvec}$ and $v\in\rim{\cF}^1_{\rim{o},\dvec}$, the pairs $(u,v)$ and  $(h^1,q^1)$ satisfy the relation $(u,v)\og (h^1,q^1)$ defined in \eqref{def:ogg}.  
By Lemma \ref{lem:opm} we can couple random paths $\pi^{u,v}\sim Q_{u,v}$ and $\pi^{h^1,q^1}\sim Q_{h^1,q^1}$ so that $\pi^{u,v}\og\pi^{h^1,q^1}$  
in the path ordering defined in Appendix \ref{sec:p-ord}, 
simultaneously for all $u\in\cI_{o,\dvec}$ and $v\in\rim{\cF}^1_{\rim{o},\dvec}$.    Then $\pi^{u,v}\in \pathsp^0_{u,v}$ forces $\pi^{h^1,q^1}\in\pathsp^-_{h^1,q^1}$, and we conclude that 
			\begin{align*} 
			p^{u,v}_0=Q_{u,v}(\pathsp^0_{u,v})
			\leq Q_{h^1,q^1}\big(\pathsp^-_{h^1,q^1}\big)
			\qquad\text{for all }  \ 
			u\in\cI_{o,\dvec},\,v\in\rim{\cF}^1_{\rim{o},\dvec}. 
		\end{align*}
		Hence 
		\[ \text{the probability on the left of \eqref{2206}} \;\le\; \P\{ Q_{h^1,q^1}(\pathsp^-_{h^1,q^1})>\delta\}. \] 
The last probability will be shown to be small by appeal to a KPZ wandering exponent bound from \cite{sepp-12-aop-corr} stated in Appendix \ref{sec:kpz5}.  To this end we check  that the line segment $[h^1,q^1]$ from $h^1$ to $q^1$ crosses the vertical axis far above the origin on the scale $N^{2/3}$.

		
		For $o\in \partial^N$ and $\rim{o}=-o\in\rim\partial^N$,   decompose   $h^j=o+l^j$ and $q^j=\rim{o}+r^j$.  These vectors $l^j=(l^j_1, l^j_2)$  and $r^j=(r^j_1, r^j_2)$  satisfy  
		\be\label{e^uv}   \abs{l^j}_1= \tfrac12d_1N^\frac{2}{3}, \quad  \abs{r^j}_1= \tfrac12d_2N^\frac{2}{3}, 
		 \quad\text{and}\quad r^j_1r^j_2\leq 0. \ee

		Use first the definition of $h^j$ and then   $q^j_i-h^j_i= \rim{o}_i+r^j_i-(o_i+l_i^j) = -2o_i +r^j_i- l^j_i$ to obtain 
		\be \label{234} \begin{aligned}
		h^j_2 - \frac{q^j_2-h^j_2}{q^j_1-h^j_1}h^j_1 
		&=o_2- \frac{q^j_2-h^j_2}{q^j_1-h^j_1}o_1 +l^j_2  - \frac{q^j_2-h^j_2}{q^j_1-h^j_1} l^j_1 \\[4pt] 
			&= \frac{o_2r^j_1-o_1r^j_2}{q^j_1-h^j_1}  - \frac{o_2l^j_1-o_1l^j_2}{q^j_1-h^j_1} +l^j_2  - \frac{q^j_2-h^j_2}{q^j_1-h^j_1} l^j_1 . 
		\end{aligned}\ee 	
		The first term on the last line is of order $\Theta(d_2N^{2/3})$ because  there is no cancellation in the numerator.    It   is positive if $j=1$ and negative if $j=2$.   This term   dominates because $d_2=N^\frac{1}{8}>>1=d_1$.
		
	Let $y^1\evec_2\in[h^1,q^1]$, that is, $y^1$ is the distance from the origin to the point where  the line segment $[h^1,q^1]$ crosses the $y$-axis.   We bound this quantity from below.  In addition to \eqref{e^uv},  utilize $\;-N\le o_i\le -\e N$,  $2N\e\le q^j_i-h^j_i\le 2N$ and   the slope bound $\e\le \frac{q^j_2-h^j_2}{q^j_1-h^j_1}\le \e^{-1}$.  The last line of \eqref{234} gives    
		 		\be\label{237-1} \begin{aligned}
				y^1=
			h^1_2 +\frac{q^1_2-h^1_2}{q^1_1-h^1_1}(-h^1_1)  
			&\ge  \frac{\e N  \abs{r^1}_1}{ 2N}    - \Bigl(   \frac{N}{2N\e} +1+\e^{-1}\Bigr) \abs{l^1}_1  \\[4pt] 
			&\ge    \tfrac14\e d_2N^\frac{2}{3}   - 2\e^{-1}  d_1N^\frac{2}{3}  
			\ge  \tfrac1{8}\e d_2N^\frac{2}{3}. 
		\end{aligned} \ee
		The last inequality used $(d_1, d_2)=(1,N^{1/8})$ and took $N\ge (16\e^{-2})^8$.
	The wandering exponent bound stated in Theorem \ref{thm:kpz5} gives 			
		\begin{align*}
			P_{h^1,q^1}(\pathsp^-_{h^1,q^1})\leq C(\e)d_2^{\,-3}
		\end{align*}	
		for a constant $C(\e)$ that works for all $o\in\partial^N$ and $N\ge N_0(\e)$. 
		By Markov's inequality
		\begin{align}\label{ub}
			\P\{ Q_{h^1,q^1}(\pathsp^-_{h^1,q^1})>\delta\} \leq C(\e)\delta^{-1}d_2^{\,-3}=C(\e)\delta^{-1}N^{-3/8}.
		\end{align}
		The proof of \eqref{2206}  is complete.
	\end{proof}
	
We combine the estimates from above   to cover all vertices on $\partial^N$ and $\rim\partial^N$.  	
	
	\begin{figure}
		\includegraphics[width=8.3cm]{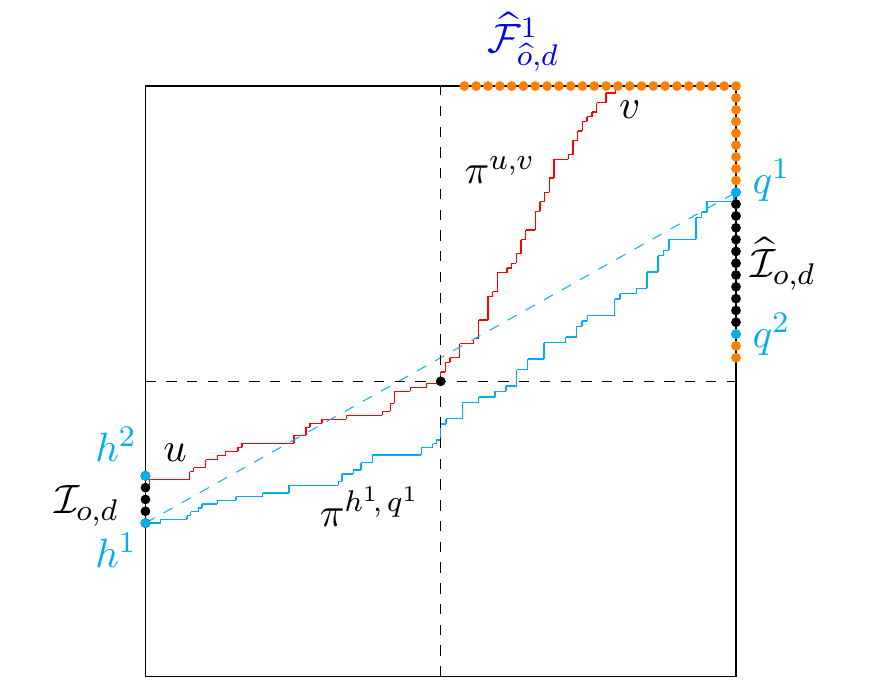} 
		\caption{\small Illustration of the proof of Lemma \ref{lm:far}.  The  path $\pi^{u,v}$ connects  $\cI_{o,\dvec}$ and   $\rim{\cF}^1_{\rim{o},\dvec}$   through the edge $\ed_0=((0,0),(1,0))$.    The  path  $\pi^{\,h^1\!, \,q^1}$  lies below $\pi^{u,v}$ and hence well below the $[h^1, q^1]$ line segment (dashed line).}
		\label{fig:bigdev}
	\end{figure}

	\begin{theorem}\label{thm:ub}  There exist constants $C(\e), N_0(\e)$ such that for $\delta\in(0,1)$ and  $N\ge \delta^{-1}\vee N_0(\e)$, 
	\[  
	\P\Big(\;	\sup_{u\,\in\,\partial^N,\,v\,\in\,\rim{\partial}^N} p^{u,v}_0>\delta \Big)\leq C(\e)\delta^{-1}N^{-\frac{1}{24}}.
	\]  
\end{theorem}
\begin{proof}
	As before,  $\dvec=(1,N^\frac{1}{8})$. 	
We first claim that for any  $o\in \partial^N$, 	
	\be\label{ub43} 
			\P\Big(\sup_{u\tspb\in\tspb\cI_{o,\dvec}\tspa, \, v\,\in\,\rim{\partial}^N
			} p^{u,v}_0>\delta\Big)\leq C(\e)\delta^{-1} N^{-\frac{3}{8}}. 
		\ee 
 		This comes from a combination of  Lemmas \ref{lm:close} and   \ref{lm:far}:  since $\rim{\partial}^N=\rim{\cI}_{\rim{o},\dvec}\cup \rim{\cF}_{\rim{o},\dvec}$,    
		\begin{align*} 
		\P\Big(\sup_{u\tspb\in\tspb\cI_{o,\dvec},v\,\in\,\rim{\partial}^N
		} p^{u,v}_0>\delta\Big)
		&\leq \P\Big(\sup_{u\tspb\in\tspb\cI_{o,\dvec},v\,\in\,\rim{\cI}_{\rim{o},\dvec}
	} p^{u,v}_0>\delta\Big)+\P\Big(\sup_{u\tspb\in\tspb\cI_{o,\dvec},v\,\in\,\rim{\cF}_{\rim{o},\dvec}
} p^{u,v}_0>\delta\Big) \\
&\le   C(\e)(\log N)^6N^{-\frac25}    
 + C(\e)\delta^{-1} N^{-\frac{3}{8}}
\leq C(\e)\delta^{-1} N^{-\frac{3}{8}}.     	 
		\end{align*} 
 Next we  coarse grain the southwest boundary  $\partial^N$.  Let 
	\begin{align*}
	\mathcal{O}^N=  \partial^N \cap \Big(\big\{(-N +id_1\fl{N^{\frac{2}{3}}}\,,-N)\big\}_{i\in\Z_{\ge0}} \; \bigcup  \,\big\{(-N,-N+jd_1\fl{N^{\frac{2}{3}}})\big\}_{j\in\Z_{\ge0}}\Big)   
	\end{align*}
	so that 
	\begin{align*}
	\Big\{\sup_{u\,\in\,\partial^N,\,v\,\in\,\rim{\partial}^N} p^{u,v}_0>\delta\Big\} \subseteq \Big\{\sup_{o\,\in\, \mathcal{O}^N} 	\sup_{u\,\in\,\cI_{o,\dvec},\,v\,\in\, \rim{\partial}^N} p^{u,v}_0>\delta\Big \}.
	\end{align*}
	As $|\mathcal{O}^N|\leq C(\e) d_1^{-1}N^{1-\frac{2}{3}}=C(\e)N^\frac{1}{3}$, a  union bound and   \eqref{ub43} give the conclusion: 
\begin{align*}
	\P\Big(\sup_{u\,\in\,\partial^N,\,v\,\in\,\rim{\partial}^N} p^{u,v}_0>\delta\Big) 
	&\leq \sum_{o\,\in\,\mathcal{O}^N} \P\Big(\;	\sup_{u\,\in\,\cI_{o,\dvec}, \,v\,\in\,\rim{\partial}^N} p^{u,v}_0>\delta\Big)\\
	&\leq C(\e)N^{\frac{1}{3}}\delta^{-1}N^{-\frac{3}{8}}=C(\e)\delta^{-1}N^{-\frac{1}{24}}.
	\qedhere 
	\end{align*} 
\end{proof}

\medskip 	
	
\section{Proof of the main theorem} 
\label{sec:pf-main}

	\begin{proof}[Proof of Theorem \ref{thm:noex}] 
By Theorem \ref{thm:e_i-mu}(b), for almost every $\w$ every bi-infinite Gibbs measure $\mu$  satisfies 
\be\label{3067} \begin{aligned}
&\bigl\{ \tsp\varliminf_{\abs n\to\infty} \abs{n^{-1} X_n\cdot\evec_1} =0\bigr\} 	
\cup
\bigl\{ \tsp\varliminf_{\abs n\to\infty} \abs{n^{-1} X_n\cdot\evec_2} =0\bigr\}  \\
	&\qquad\qquad =
	\{\text{$X_\bbullet$ is a bi-infinite straight line}\} 
	\qquad\text{$\mu$-almost surely}  
\end{aligned}\ee 
where $X_\bbullet=X_{-\infty:\infty}$  is   the bi-infinite polymer path under the measure $\mu$. 
This equality  follows because Theorem \ref{thm:e_i-mu}(b) has these consequences for \eqref{3067}: the union on the left is disjoint,  the event on the right is a subset of the  union on the left,  and   their $\mu$-probabilities are equal.  The complement of the union on the left is the following event:   the limit points of $\abs{n}^{-1}X_n$  lie in $\,]-\evec_2, -\evec_1[$ when  $n\to-\infty$ and  in $\,]\evec_2, \evec_1[$ when  $n\to\infty$.
Thus to complete the proof we show the existence of an event $\Omega'$ such that $\P(\Omega')=1$ and for each $\w\in\Omega'$, no  $\mu\in \overleftrightarrow{\text{\rm DLR}}^\omega$  assigns positive probability to this last property of the limit points  of $\abs{n}^{-1}X_n$. 
		
		

\medskip 		
		
We put $\e$ back into the notation. For $\e>0$ let
		\begin{align*}
			\cD^\e=\{\xi\in\,]\evec_2,\evec_1[\,:\e^{1/2} \leq \xi_2/\xi_1\leq {\e}^{-1/2}\}. 
		\end{align*}
Say that a bi-infinite path $x_\bbullet$ is 	$(-\cD^{\e})\times \cD^{\e}$-directed if the limit points of $\abs{n}^{-1}x_n$  lie in $-\cD^{\e}$ when  $n\to-\infty$ and  in $\cD^{\e}$ when  $n\to\infty$.  
				Recall the definition of the edges $\ed_i=(i\evec_2,i\evec_2+\evec_1)$   and define these sets of bi-infinite paths:  
		\begin{align*}
			\pathsp^{\e, i}=\big\{x_\bbullet\in\pathsp:  \text{ $x_\bbullet$ is $(-\cD^{\e})\times \cD^{\e}$-directed and $x_\bbullet$ goes through $\ed_i$}
			\big\}. 
		\end{align*}
	We show the existence of an event $\Omega'$ 	  of full $\P$-probability such that, for $\w\in\Omega'$,  $\mu\in \overleftrightarrow{\text{\rm DLR}}^\omega$, $\e>0$,  and $i\in\Z$,  
		\begin{align}\label{c00}
			\mu(\pathsp^{\e, i})=0.
		\end{align}
Assume this proved. 
 Let   $\e_k=2^{-k}$.  Then for $\w\in\Omega'$ and $\mu\in \overleftrightarrow{\text{\rm DLR}}^\omega$, 
		\begin{align*}
			\mu\{\text{$X_\bbullet$ 
			 is $\,]-\evec_2,-\evec_1[\,\times\,]\evec_2,\evec_1[\,$-directed}\}  
			 &\le \sum_{k\ge1}  \mu \{ \text{$X_\bbullet$ is $(-\cD^{\e_k})\times\cD^{\e_k}$-directed}\} \\
			&\le  \sum_{k\ge1}\sum_{i\in\Z}\mu\big(\pathsp^{\e_k, i}\big)
			 =0,
		\end{align*}
		which is the required result. 
		
	
	\medskip 
	
		It remains to define the event $\Omega'$ and verify \eqref{c00}.    Recall the  definition \eqref{pi} of $p^{u,v}_i$.   Define translations $T_x$ on weight configurations $\w=(Y_x)$ by $(T_x\w)_y=Y_{x+y}$.  Define 
\[  \xi^{\e}_N= \sup_{u\tspb\in\tspb\partial^{N\!,\tsp\e}\!,\,v\tspb\in\tspb\rim{\partial}^{N\!,\tsp\e}}p^{u,v}_0\,,
\qquad
\Omega''_{\e}=\bigl\{  \varliminf_{N\to\infty} \xi^{\e}_{N+\ce{N^{2/3}}}=0\bigr\}
\qquad\text{and}\qquad 
\Omega'=\bigcap_{k\ge 1} \bigcap_{i\in\Z}  T_{i\evec_2}\Omega''_{\e_k} .  \]  
By Theorem \ref{thm:ub}, 	$\xi^{\e}_{N}\to0$ in probability as $N\to\infty$, and hence $\P(\Omega')=\P(\Omega''_{\e})=1$.

		 A $(-\cD^\e)\times\cD^{\e}$-directed bi-infinite path intersects both $ \partial^{N\!,\tspa\e}$  and  $\rim\partial^{N\!,\tspa\e}$ for all large enough $N$. (This is because $\cD^\e$  bounds the slopes by $\e^{1/2}$ which is larger than $\e$.)  Thus if we let 
				\begin{equation*}
			\pathsp^{N,\e,  i}=\{x_\bbullet\in\pathsp^{\e, i}:x_\bbullet\cap\partial^{N\!,\tspa\e}\neq\emptyset, \,x_\bbullet\cap\rim\partial^{N\!,\tspa\e}\neq\emptyset\} 
		\end{equation*}
		then 
		\be\label{c45}   \pathsp^{\e, i} = \bigcup_{m\ge 1} \bigcap_{N\ge m} \pathsp^{N,\e,  i}. \ee

   		\begin{figure}
		\includegraphics[width=10cm]{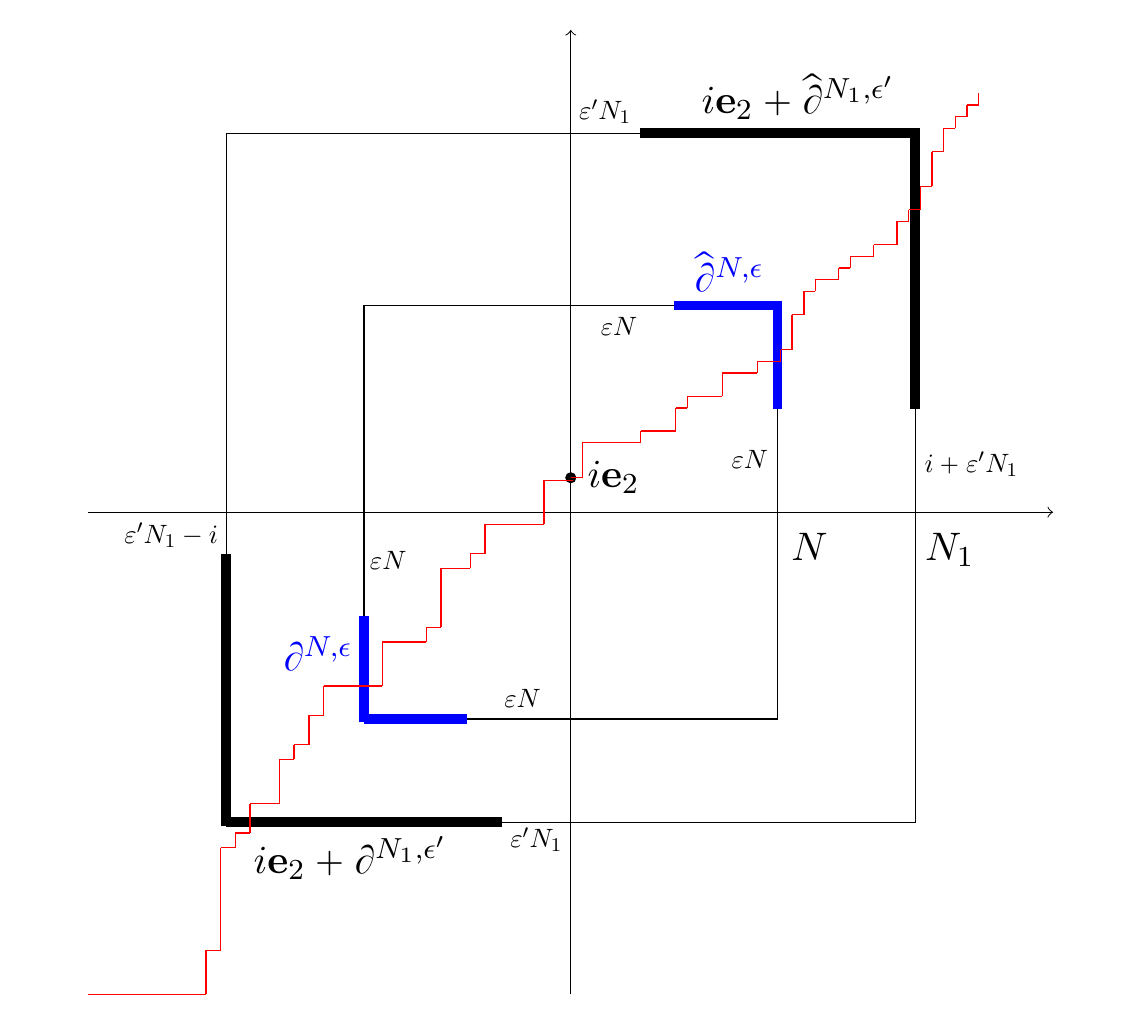} 
		\caption{\small The inner $N\times N$ square is centered at $\zevec$ while the outer $N_1\times N_1$ square is centered at $i\evec_2$. The (thick, dark) boundary segments of the outer square cover the (thick, light) boundary segments of the inner square. Thus the path through $i\evec_2$ that crosses $\partial^{N,\e}$ and $\rim\partial^{N,\e}$ is forced to also cross $i\evec_2+\partial^{N_1,\tspa\e'}$  and $i\evec_2+\rim\partial^{N_1,\tspa\e'}$.}
		\label{fig:2sq} 
	\end{figure}

Let $\e=2^{-k}$ for some $k\ge 1$, $\e'=\e/2$, and abbreviate $N_1=N+\ce{N^{2/3}}$.
   In the scale $N_1$ consider  the translated square $i\evec_2+\lzb -N_1, N_1\rzb^2$ centered at $i\evec_2$, with its boundary portions   $i\evec_2+\partial^{N_1,\tspa\e'}$ in the southwest  and $i\evec_2+\rim\partial^{N_1,\tspa\e'}$ in the northeast.  This translated $N_1$-square contains $\lzb-N,N\rzb^2$ for all  $i\in\lzb-N^{2/3}, N^{2/3}\rzb$.  

  There exists a finite constant $N_0(\e)$ such that  $\abs i+\e'N_1\le \e N$ for all  $i\in\lzb-N^{2/3}, N^{2/3}\rzb$ and $N\ge N_0(\e)$.  
   Then  every path $x_\bbullet\in\pathsp^{\e, \tsp N, \tsp i}$ necessarily goes through both  $i\evec_2+\partial^{N_1,\tspa\e'}$  and $i\evec_2+\rim\partial^{N_1,\tspa\e'}$.  In other words,  $x_\bbullet$ is a member  of the translate  $i\evec_2+\pathsp^{\e', \tsp N_1, \tsp 0}$ of the class of paths that go through the edge $\ed_0$. This is illustrated in Figure \ref{fig:2sq}.

		On the event  $\pathsp^{\e, \tsp N, \tsp i}$ let, in the coordinatewise ordering, $X_\partial=\inf\{X_\bbullet\cap (i\evec_2+\partial^{N_1,\tspa\e'})\}$ be the first vertex of the path $X_\bbullet$  in $i\evec_2+\partial^{N_1,\tspa\e'}$ and 
			$X_{\rim{\partial}}=\sup\{X_\bbullet\cap (i\evec_2+\rim\partial^{N_1,\tspa\e'})\}$  the last vertex of the path in $i\evec_2+\rim\partial^{N_1,\tspa\e'}$.   Note that for $u\in(i\evec_2+\partial^{N_1,\tspa\e'})$  and 
			$v\in(i\evec_2+\rim\partial^{N_1,\tspa\e'})$, the event $\{X_\partial=u, X_{\rim{\partial}}=v\}$ depends on the entire path $X_\bbullet$ only through its edges outside $i\evec_2+\lzb-N_1,N_1\rzb^2$. 
Suppose $\mu(\pathsp^{N,\e,  i})>0$ for some $\mu\in \overleftrightarrow{\text{\rm DLR}}^\omega$. Below we apply the Gibbs property,  recall the definition \eqref{tz}  of $\pathsp_{u,v}^0$ as the set of paths from $u$ to $v$ that take the edge $\ed_0=(\zevec, \evec_1)$, and write $Q^\w$ so that we can include explicitly translation of the weights $\w$. 
 		\begin{align*}
			\mu(\pathsp^{N,\e,  i})&\le \mu(i\evec_2+\pathsp^{\e', \tsp N_1, \tsp 0}) 
			\\
			&\le \sum_{u\tspb\in\tspb\partial^{N_1\!,\tspa\e'}\!,\,v\tspb\in\tspb\rim\partial^{N_1\!,\tspa\e'}} 
			\mu(i\evec_2+\pathsp_{u,v}^0\,\vert\,X_\partial=i\evec_2+u,X_{\rim{\partial}}=i\evec_2+v)\tspb\mu(X_\partial=i\evec_2+u,X_{\rim{\partial}}=i\evec_2+v)\\
			&=\sum_{u\tspb\in\tspb\partial^{N_1\!,\tspa\e'}\!,\,v\tspb\in\tspb\rim\partial^{N_1\!,\tspa\e'}} Q^\w_{i\evec_2+u,\tsp i\evec_2+v}(i\evec_2+\pathsp_{u,v}^0)
			\tspb\mu(X_\partial=i\evec_2+u,X_{\rim{\partial}}=i\evec_2+v)\\
			&\le \max_{u\tspb\in\tspb\partial^{N_1\!,\tspa\e'}\!,\,v\tspb\in\tspb\rim\partial^{N_1\!,\tspa\e'}}  Q^\w_{i\evec_2+u,\tsp i\evec_2+v}(i\evec_2+\pathsp_{u,v}^0)  
			=\max_{u\tspb\in\tspb\partial^{N_1\!,\tspa\e'}\!,\,v\tspb\in\tspb\rim\partial^{N_1\!,\tspa\e'}}  Q^{T_{i\evec_2}\w}_{u,\tsp v}(\pathsp_{u,v}^0) \\
			&=
			\max_{u\tspb\in\tspb\partial^{N_1\!,\tspa\e'}\!,\,v\tspb\in\tspb\rim\partial^{N_1\!,\tspa\e'}}p^{u,v}_0(T_{i\evec_2}\w)  = \xi^{\e'}_{N_1}(T_{i\evec_2}\w).
		\end{align*}
Then \eqref{c45} gives, on the event $\Omega'$, 	 
\[   	
\mu(\pathsp^{\e, i})  \le \varliminf_{N\to\infty} \mu(\pathsp^{N,\e,  i})   \le \varliminf_{N\to\infty}   \xi^{\e'}_{N_1}\circ T_{i\evec_2} =0.   
\]  
\eqref{c00} has been verified.  This completes the proof of the main result Theorem \ref{thm:noex}. 
	\end{proof}

	\smallskip

	\appendix

\section{General properties of planar directed polymers} 
\label{app:genpol}

This appendix covers some consequences of the general polymer formalism. 
Begin again with the partition function with given weights $Y_x>0$: 
 \be\label{m:G807} 
	Z_{u,v} =\sum_{x_\brbullet\tspa\in\tspa\pathsp_{u,v}} \prod_{i=0}^{\abs{v-u}_1} \Yw_{x_i} 
	\quad\text{for $u\le v$ on $\Z^2$,} 
	\ee
with $Z_{u,v}=0$ if $u\le v$ fails.  

\subsection{Ratio weights and nested polymers} 
  Keeping the base point $u$ fixed, define ratio weights for varying $x$: 
\[  I_x=I_{u,x}=\frac{Z_{u,x}}{Z_{u,x-\evec_1}}  \quad\text{ and }\quad  J_x=J_{u,x}=\frac{Z_{u,x}}{Z_{u,x-\evec_2}}.\] 
The ratio weights can be  calculated inductively from boundary values $I_{u+k\evec_1}=\Yw_{u+k\evec_1}$ and 
  $J_{u+\ell\evec_2}=\Yw_{u+\ell\evec_2}$  for $k,\ell\ge 1$,  by iterating 
\be\label{m:NE}   I_x= \Yw_x\bigl( 1+ I_{x-\evec_2} J_{x-\evec_1}^{-1}\bigr) 
 \quad\text{ and }\quad  
  J_x=  \Yw_x\bigl(  J_{x-\evec_1}I_{x-\evec_2}^{-1}+1 \bigr) . 
 \ee  
 
%
%


Let $u\le v$ on $\Z^2$.  On the boundary of the quadrant $v+\Z_{\ge0}^2$, put ratio  weights  of the partition functions with base point $u$: 
\[  Y^{(u)}_{v+i\evec_r}  =  \frac{Z_{u, v+i\evec_r}}{Z_{u, v+(i-1)\evec_r}}
\qquad \text{for } \ r\in\{1,2\} \ \text{ and }\    i\ge 1.   
\] 
The ratio weights dominate the original weights: $Y^{(u)}_{v+i\evec_r}\ge Y_{v+i\evec_r} $, 
and equality  holds iff $v=u+m\evec_r$ for some $m\ge 0$.

 Define a partition function  $Z^{(u)}_{v,w}$ that uses these boundary weights and ignores the first weight of the path: for $k,\ell\ge 1$ and $w\in v+\Z_{>0}^2$, 
 \begin{align*}  
  Z^{(u)}_{v,v}&=1, \quad   Z^{(u)}_{v,v+k\evec_1} =  \prod_{i=1}^k Y^{(u)}_{v+i\evec_1}, 
  \quad   Z^{(u)}_{v,v+\ell\evec_2} =  \prod_{j=1}^\ell Y^{(u)}_{v+j\evec_2}\\
   Z^{(u)}_{v,w}  &=\sum_{k=1}^{w_1-v_1} \biggl(\; \prod_{i=1}^k Y^{(u)}_{v+i\evec_1} \biggr) Z_{v+k\evec_1+\evec_2, w} 
 +  \sum_{\ell=1}^{w_2-v_2}  \biggl(\; \prod_{j=1}^\ell Y^{(u)}_{v+j\evec_2} \biggr) Z_{v+\evec_1+\ell\evec_2, w} . 
 \end{align*} 
For $w\in v+\Z_{>0}^2$  the definition from above can be rewritten as follows:
\begin{align*}
 Z^{(u)}_{v,w} 
& =\frac1{Z_{u,v}} \sum_{k=1}^{w_1-v_1}  Z_{u, v+k\evec_1} Z_{v+k\evec_1+\evec_2, w} 
 +  \frac1{Z_{u,v}} \sum_{\ell=1}^{w_2-v_2} Z_{u,v+\ell\evec_2} Z_{v+\evec_1+\ell\evec_2, w} . 
\end{align*} 
and thus for all   $u\le v\le w$ we have the identity  
\be\label{m:820}  Z^{(u)}_{v,w} = \frac{Z_{u,w}}{Z_{u,v}}. 
 \ee
	
Ratio variables  satisfy 	
\be\label{m:824} 
 I_{u,x}=\frac{Z_{u, x}}{Z_{u, x-\evec_1}} =\frac{Z_{u,v} \, Z^{(u)}_{v,x}}{Z_{u,v} \, Z^{(u)}_{v,x-\evec_1}}
 =\frac{Z^{(u)}_{v,x}}{Z^{(u)}_{v,x-\evec_1}}
 = I^{(u)}_{v,x}
 \ee
 with the analogous identity $J_{u,x}=J^{(u)}_{v,x}$.  
 
Recall the definition \eqref{exit77} of $\ex_{u,v,w}$.  Let $Q^{(u)}_{v,w}$ be the   quenched path probability on $\pathsp_{v,w}$  that corresponds to the partition function $Z^{(u)}_{v,w}$.  Then  we have the identity 
 \be\label{m:830}
 Q_{u,w}(\ex_{u,v,w}=\ell) =  Q^{(u)}_{v,w}(\ex_{v,w}=\ell)
 \qquad\text{for } \ 0\ne\ell\in\Z. 
 \ee
 Here is the derivation for the case where the path from $u$ to $w$ goes above $v$. Let $k\ge 1$. Apply \eqref{m:820} and \eqref{m:824}.  
 \begin{align*}
  Q_{u,w}(\ex_{u,v,w}=-k) &= \frac{Z_{u, v+k\evec_2} Z_{v+\evec_1+k\evec_2, w}}{Z_{u,w}}
  = \frac{Z_{u,v} \bigl( \,\prod_{j=1}^k J_{u, v+j\evec_2}\bigr)  Z_{v+\evec_1+k\evec_2, w}}{Z_{u,w}}\\
 &= \frac{ \bigl( \,\prod_{j=1}^k J^{(u)}_{v, v+j\evec_2}\bigr)  Z_{v+\evec_1+k\evec_2, w}}{Z^{(u)}_{v,w}}   
 =  Q^{(u)}_{v,w}(\ex_{v,w}=-k).  
 \end{align*}

\subsection{Inequalities for point-to-point partition functions} 

We state several inequalities that follow from the next basic lemma.  The inequalities in  \eqref{m:770}   below are proved together by induction on $x$  and $y$, 
beginning with $x=u+k\evec_1$ and $y=u+\ell\evec_2$.    The induction step is carried out by formulas \eqref{m:NE}.

 \begin{lemma}  Fix a base point $u$.  Let $\{\Yw^{(1)}_x\}$ and $\{\Yw^{(2)}_x\}$ be strictly positive weights from which partition functions $Z^{(1)}_{u,v}$ and $Z^{(2)}_{u,v}$ are defined. Assume that  $Y^{(1)}_u=\Yw^{(2)}_u$, $Y^{(1)}_{u+k\evec_1}\le \Yw^{(2)}_{u+k\evec_1}$,  $\Yw^{(2)}_{u+\ell\evec_2}\le  Y^{(1)}_{u+\ell\evec_2}$ and $Y^{(1)}_x= \Yw^{(2)}_x$ for all $k,\ell\ge 1$  and $x\in u+\Z_{>0}^2$.  Then we have the following inequalities
 for  $x\ge u+\evec_1$ and $y\ge u+\evec_2$:   
 \be\label{m:770}    \frac{Z^{(1)}_{u,x}}{Z^{(1)}_{u,x-\evec_1}} \le \frac{Z^{(2)}_{u,x}}{Z^{(2)}_{u,x-\evec_1}} 
 \qquad\text{and}\qquad 
 \frac{Z^{(2)}_{u,y}}{Z^{(2)}_{u,y-\evec_2}}\le  \frac{Z^{(1)}_{u,y}}{Z^{(1)}_{u,y-\evec_2}} . 
  \ee
 \end{lemma}
 
 From the lemma we obtain the following pair of inequalities:
 		\be\label{ine2}\begin{aligned}
			\frac{Z_{u,z}}{Z_{u,z-\evec_1}}\leq \frac{Z_{u+\evec_1,z}}{Z_{u+\evec_1,z-\evec_1}}   \qquad\text{and}\qquad  
			\frac{Z_{u,z}}{Z_{u,z-\evec_1}}\leq \frac{Z_{u-\evec_2,z}}{Z_{u-\evec_2,z-\evec_1}} . 
		\end{aligned}\ee   
The first inequality above follows from the first inequality of \eqref{m:770} by letting the weights $\{Y_{u+j\evec_2}\}_{j\ge 1}$ tend to zero, and the second one by letting the weights $\{Y_{u-\evec_2+i\evec_1}\}_{i\ge 1}$ tend to zero.


\begin{lemma}\label{lem:cl}
	Let $x,y,z\in\Z^2$ be such that $x\og y$ and $x,y\leq z-\evec_1-\evec_2$. We then have
	\begin{align}
		&\frac{Z_{x,z}}{Z_{x,z-\evec_1}}\leq \frac{Z_{y,z}}{Z_{y,z-\evec_1}}\label{in1}\\
		& \frac{Z_{y,z}}{Z_{y,z-\evec_2}}\leq\frac{Z_{x,z}}{Z_{x,z-\evec_2}}\label{in2}.
	\end{align}
	\end{lemma}
	\begin{proof}
		   \eqref{in1} follows from repeated application of  \eqref{ine2} along the steps $\evec_1$ and $-\evec_2$ from $x$ to $y$.  Transposing \eqref{in1}  gives \eqref{in2}. \end{proof}

Since $u+k\evec_1 \succcurlyeq u$ and $u+\ell\evec_2\preccurlyeq u$ for $k,\ell\ge0$,  inequalities \eqref{in1}--\eqref{in2} imply also these  for $1\le k< (x-u)\cdot\evec_1$ and $1\le\ell< (y-u)\cdot\evec_2$: 
\be\label{m:772} 
\frac{Z_{u,x}}{Z_{u,x-\evec_1}} 
 \le  \frac{Z_{u,x}(\ex_{u,x}\ge k)}{Z_{u,x-\evec_1}(\ex_{u,x-\evec_1}\ge k)}
 \quad\text{ and }\quad 
\frac{Z_{u,y}}{Z_{u,y-\evec_2}}  \le  \frac{Z_{u,y}(\ex_{u,y}\le-\ell)}{Z_{u,y-\evec_2}(\ex_{u,y-\evec_2}\le-\ell)}. 
  \ee
To illustrate the explicit proof of the first one: 
 \[    \frac{Z_{u,x}(\ex_{u,x}\ge k)}{Z_{u,x-\evec_1}(\ex_{u,x-\evec_1}\ge k)} 
 = \frac{ (\,\prod_{i=0}^{k-1} \Yw_{u+i\evec_1}) \, Z_{u+k\evec_1,x}}{(\,\prod_{i=0}^{k-1} \Yw_{u+i\evec_1}) \,Z_{u+k\evec_1,x-\evec_1}}
= \frac{Z_{u+k\evec_1,x}}{Z_{u+k\evec_1,x-\evec_1}}
\ge   \frac{Z_{u,x}}{Z_{u,x-\evec_1}}. 
 \] 

\smallskip

\subsection{Ordering of path measures}\label{sec:p-ord} 

The down-right partial order $\og$ on $\R^2$ and $\Z^2$ was defined by 
 $(x_1,x_2)\preccurlyeq(y_1,y_2)$  if $x_1\le y_1$ and $x_2\ge y_2$.  
Extend this relation  to pairs of vertices 
 $(x^1,y^1),(x^2,y^2)\in\Z^2\times\Z^2$ as follows (illustrated in Figure \ref{fig:opoints}):
 \be\label{def:ogg} 
 (x^1,y^1)\og(x^2,y^2) \quad\text{ if } x^1\leq y^1, \ x^2\leq y^2 , \ x^1\og x^2 \text{ and } y^1\og y^2. 
 \ee
 Extend this relation further to finite paths:  $\pi^1\in\pathsp_{x^1,y^1}$ and $\pi^2\in\pathsp_{x^2,y^2}$ satisfy $\pi^1\og\pi^2$ if the pairs of endpoints satisfy $(x^1,y^1)\og(x^2,y^2)$ and  whenever $z^1\in\pi^1$, $z^2\in\pi^2$, and $z^1\cdot(\evec_1+\evec_2)=z^2\cdot(\evec_1+\evec_2)$, we have $z^1\og z^2$.  Pictorially, in a very clear sense, $\pi^1$ lies (weakly)  above and to the left of $\pi^2$.  See  again Figure \ref{fig:opoints}.

\begin{figure}[t]
	\begin{center}
		\begin{subfigure}{.3\textwidth}
		\begin{tikzpicture}[scale=0.65, every node/.style={transform shape}]
		node/.style={transform shape}]
		\def\s{1.33}
		\node [scale=\s,below] at (0,2) {$x^1$};
		\fill (0,2) circle (3pt);
		\node [scale=\s,right] at (3,5) {$y^1$};
		\fill (3,5) circle (3pt);
		\node [scale=\s,right] at (2,0) {$x^2$};
		\fill (2,0) circle (3pt);
		\node [scale=\s,right] at (5,3) {$y^2$};
		\fill (5,3) circle (3pt);	
		\draw [line width=0.4mm] plot [smooth, tension=0.8] coordinates { (0,2) (1.7,2.8) (3,5) };
		\node [scale=\s,left] at (2.1,3.5) {$\pi^1$};
		\draw [line width=0.4mm] plot [smooth, tension=0.8] coordinates { (2,0) (3.5,2) (5,3) };
		\node [scale=\s,right] at (3.5,1.8) {$\pi^2$};
		\end{tikzpicture}
		\end{subfigure}%
		\begin{subfigure}{.3\textwidth}
			\begin{tikzpicture}[scale=0.65, every node/.style={transform shape}]
			node/.style={transform shape}]
			\def\s{1.33}
			\node [scale=\s,below] at (0,1) {$x^1$};
			\fill (0,1) circle (3pt);
			\node [scale=\s,right] at (4,3) {$y^1$};
			\fill (4,3) circle (3pt);
			\node [scale=\s,below] at (0,0) {$x^2$};
			\fill (0,0) circle (3pt);
			\node [scale=\s,right] at (3,4) {$y^2$};
			\fill (3,4) circle (3pt);	
			\draw [line width=0.4mm] plot [smooth, tension=0.8] coordinates { (0,1) (1.7,2.8) (4,3) };
			\node [scale=\s,left] at (1.8,3.1) {$\pi^1$};
			\draw [line width=0.4mm] plot [smooth, tension=0.8] coordinates { (0,0) (2,1) (3,4) };
			\node [scale=\s,right] at (2.5,1.7) {$\pi^2$};
			\end{tikzpicture}
		\end{subfigure}%
	\end{center}
	\caption{\small On the left the pairs $(x^1,y^1)$ and $(x^2,y^2)$ satisfy $(x^1,y^1)\og (x^2,y^2)$, while on the right this relation fails. Consistently with this, on the left the paths $\pi^1\in\pathsp_{x^1,y^1}$ and $\pi^2\in\pathsp_{x^2,y^2}$ satisfy $\pi^1\og\pi^2$ but on the right this fails.}
\label{fig:opoints}
\end{figure}
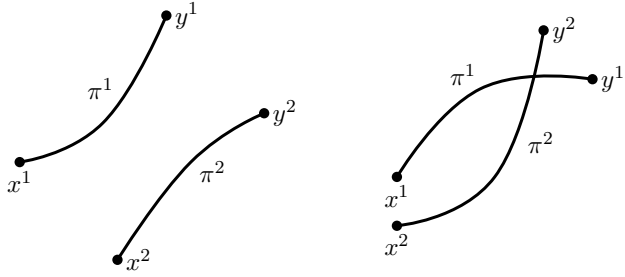

Let $\mu$ and $\nu$ be probability measures on the finite path spaces  $\pathsp_{x^1,y^1}$ and $\pathsp_{x^2,y^2}$,  respectively. We write $\mu\og \nu$ if there exist random paths $X^1\in\pathsp_{x^1,y^1}$ and $X^2\in\pathsp_{x^2,y^2}$ on a common probability space such that $X^1\sim \mu$, $X^2\sim \nu$, and $X^1\og X^2$. In other words, $\mu\og \nu$ if $\nu$ stochastically dominates $\mu$ under the partial order  $\og$ on paths. The following shows that for fixed weights  there exists a coupling of all  the  quenched polymer distributions $\{Q_{x,y}\}_{x\le y}$  on the lattice $\Z^2$ so that   $Q_{x,y}\og Q_{u,v}$ whenever $(x,y)\og (u,v)$.

\begin{lemma}\label{lem:opm}
	Let $(Y_x)_{x\tspa\in\tspa\Z^2}$ be an assignment of strictly positive weights on the lattice $\Z^2$.  Then 
	there exists a coupling of up-right random paths  $\{\pi^{x,y}\}_{x\le y}$  such that  $\pi^{x,y}\in\pathsp_{x,y}$,  $\pi^{x,y}$ has the quenched polymer distribution  $Q_{x,y}$,  and  $\pi^{x,y}\og \pi^{u,v}$ whenever    $(x,y)\og(u,v)$. 
\end{lemma}
\begin{proof}
  Let $\{U_z\}_{z\in\Z^2}$ be an assignment of i.i.d.\ uniform random variables $U_z\sim\text{\rm Unif}(0,1)$ to the vertices of $\Z^2$, defined under some probability measure $\mP$. 
    For each pair $x\le z$ such that $x\ne z$, define the down-left pointing random unit vector 
\be\label{V6}   V^x(z)= \begin{cases}  
-\evec_1, &\text{if }  \ \   \dfrac{Y_zZ_{x,z-\evec_1}}{Z_{x,z}} \ge U_z \\[5pt]  
-\evec_2, &\text{if }  \ \   \dfrac{Y_zZ_{x,z-\evec_2}}{Z_{x,z}} > 1- U_z. 
\end{cases}  
\ee  
If $z=x+k\evec_i$  this gives $V^x(z)=-\evec_i$ due to the convention $Z_{u,v}=0$ when $u\le v$ fails.  
Hence any path that  starts at some vertex $y\ge x$ distinct from $x$ and follows  the steps from each $z$ to  $z+V^x(z)$ terminates at $x$. 

	
Since the paths from distinct points that follow increments $V^x(z)$ for a given $x$  eventually coalesce, a realization of $\{V^x(z)\}_{z\ge x: \,z\ne x}$ defines a spanning tree $\mathcal{T}^x$ rooted at $x$  on the nearest-neighbor graph on the quadrant   $x+\Z_{\ge 0}^2$.
	    For  $x\le y$ let $\pi^{x,y}\in\pathsp_{x,y}$ be the path that connects $x$ and $y$ in the tree $\cT^x$.  Then for any path $x_\bbullet\in\pathsp_{x,y}$,  \eqref{V6} implies that 
$ 
	 	Q_{x,y}(x_\bbullet)=\mP(\pi^{x,y}=x_\bbullet) .
	$  
	 In other words, through the random paths $\{\pi^{x,y}\}_{x\le y}$ we have a coupling of the quenched polymer  distributions $\{Q_{x,y}\}_{x\le y}$.

	 Let $x\og u$. By Lemma \ref{lem:cl}
	 \begin{align*}
\frac{Y_zZ_{x,z-\evec_1}}{Z_{x,z}}\geq\frac{Y_zZ_{u,z-\evec_1}}{Z_{u,z}} 
\quad\text{ and }\quad
	 \frac{Y_zZ_{u,z-\evec_1}}{Z_{u,z}}\geq\frac{Y_zZ_{x,z-\evec_2}}{Z_{x,z}}.
	 \end{align*}
Hence 
	\be\label{ine5} \begin{aligned}
	 	\{V^x(z)=-\evec_2\}&\subseteq \{V^u(z)=-\evec_2\}\\
	 \text{and}\qquad 	\{V^u(z)=-\evec_1\}&\subseteq \{V^x(z)=-\evec_1\}.
	 \end{aligned}\ee
	 It follows from \eqref{ine5} that two paths satisfy $\pi^{x,y}\og \pi^{u,v}$  whenever $(x,y)\og(u,v)$. 	 This is because if these paths share a vertex $z$, then their subsequent down-left steps satisfy  $z+V^x(z)\og z+V^u(z)$.  
\end{proof}

Let $o\le x$.  In the tree $\cT^o$  constructed above, the path from $x$ down  to $o$ stays weakly   to the left of the path from $x+\evec_1$ down  to $o$. This gives the inequality  below: 
\be\label{pmon}  
\text{for  vertices }    o\leq x \text{ and } k\ge 1,   \quad Q_{o,x}(t_{\evec_1}\ge k)\leq Q_{o,x+\evec_1}(t_{\evec_1}\ge k).
\ee


\smallskip 

\subsection{Polymers on the upper half-plane} 
\label{app:pol-H}
The stationary inverse-gamma polymer process that is our   tool for calculations will be constructed   on  a half-plane. This section defines  the notational apparatus for this purpose, borrowed from  the forthcoming work \cite{fan-sepp-20+}.

Define  mappings of  bi-infinite sequences:  $I=(I_k)_{k\tsp\in\tsp\Z}$ and  $\Yw=(\Yw_j)_{j\in\Z}$ in $\R_{>0}^\Z$ that  are assumed to  satisfy 
	\begin{equation}\label{Iw}
	C(I, \Yw)= \lim_{m\to\,-\infty}  \sum_{j=m}^0  {\Yw_{j}} \prod_{i=j+1}^{0}\frac{\Yw_i}{I_i}<\infty.
	\end{equation}
	From these inputs, three outputs  $\wt I=(\wt I_k)_{k\tsp\in\tsp\Z}$, $J=(J_k)_{k\tsp\in\tsp\Z}$ and $\wt\Yw=(\wt\Yw_k)_{k\tsp\in\tsp\Z}$,   also elements of $\R_{>0}^\Z$,  are constructed  as follows. 
	
	Let $Z=(Z_k)_{k\tsp\in\tsp\Z}$ be any function on $\Z$ that satisfies $I_k=Z_{k}/Z_{k-1}$.   This defines $Z$ up to a  positive multiplicative constant.   Define the sequence  $\wt Z=(\wt Z_\ell)_{\ell\in\Z}$ by 
	\be\label{m:800}
	\wt Z_{\ell}=\sum_{k:\,k\le \ell}    Z_k\,\prod_{i=k}^{\ell} \Yw_i \,,  
	\quad \ell\in\Z. 
	\ee
	Under assumption \eqref{Iw} the sum on the right-hand side of \eqref{m:800} is finite.   To check this choose a particular  $Z$ by setting $Z_0=1$.  (Any other admissible $Z$ is a constant multiple of this one.)   Then 
	$Z_k=\prod_{i=k+1}^0 I_i^{-1}$ for $k\le -1$. 
\begin{align*}	
	\wt Z_{\ell}&=\sum_{k:\,k\le \ell\wedge 0}    Z_k\,\prod_{i=k}^{\ell} \Yw_i 
	+  \sum_{k:\, 1\le k\le\ell}    Z_k\,\prod_{i=k}^{\ell} \Yw_i  \\
	&=\sum_{k:\,k\le \ell\wedge 0}   \biggl( \;\prod_{i=k+1}^0 I_i^{-1}\biggr) \biggl( \; \prod_{i=k}^{0} \Yw_i \biggr)  C_\ell(Y)  
	+  \sum_{k:\, 1\le k\le\ell}    Z_k\,\prod_{i=k}^{\ell} \Yw_i  \\
	&\le C(I, \Yw) C_\ell(Y)  
	+  \sum_{k:\, 1\le k\le\ell}    Z_k\,\prod_{i=k}^{\ell} \Yw_i <\infty \,.  
\end{align*}

  For $k\in\Z$ define 
	\be\label{m:801}  \wt I_k =    \wt Z_k/\wt Z_{k-1},  \ee 
	\be\label{m:J} J_k=\wt Z_k/Z_k , 
	 \ee
	\be\label{m:R}   \wt \Yw_k=(I_k^{-1}+ J_{k-1}^{-1})^{-1}. 
	 \ee  
	The sequences $\wt I$, $J$ and $\wt\Yw$ are well-defined positive real sequences,  and they do not  depend  on the choice of the function $Z$ as long as $Z$  has  ratios  $I_k=Z_{k}/Z_{k-1}$.  
	The  three mappings are denoted by  
	\be\label{m:DSR}   \wt I=\Dop(I,\Yw),  \quad J=\Sop(I,\Yw), \quad\text{and} \quad \wt \Yw=\Rop(I,\Yw). \ee
	
	Beginning from  $\wt Z_k = Y_k( Z_k + \wt Z_{k-1})$  we derive these  equations: 
%
%
	\begin{align}\label{m:801a}
	  \wt I_k &= 
  \Yw_k\biggl( \frac{I_k}{J_{k-1}}  + 1 \biggr) 
	= \frac{Y_k}{\wt Y_k} I_k \\
\label{m:Ja}   \text{and }\qquad 	 J_k&=
	 \Yw_k\biggl(1+  \frac{J_{k-1}}{I_k}   \biggr)
	 = \frac{Y_k}{\wt Y_k} J_{k-1}.  
	 \end{align} 

The last formula iterates as follows: for $\ell<m$,
\be\label{J34} 
J_m= J_\ell \prod_{i=\ell+1}^m \frac{\Yw_i}{I_i} + \sum_{j=\ell+1}^m \Yw_j \prod_{i=j+1}^m \frac{\Yw_i}{I_i}. 
\ee	

We record two  inequalities.  From \eqref{m:801a}, 
\be\label{as} 
\wt I_j   \ge   Y_j.  
\ee
If we start with two coordinatewise ordered boundary weights  $I_j\le I'_j$ (for all $j$) and use the same bulk weights  $Y$ to compute vertical ratio  weights  $J=\Sop(I,Y)$ and $J'=\Sop(I',Y)$, the inequality is reversed: 
\be\label{tt}\begin{aligned} 
J'_k= \frac{\wt Z'_k}{Z'_k} =  \sum_{j:\,j\le k}    Y_j\,\prod_{i=j+1}^{k} \frac{Y_i}{I'_i}   \le  \sum_{j:\,j\le k}    Y_j\,\prod_{i=j+1}^{k} \frac{Y_i}{I_i}   =J_k.  
\end{aligned}\ee
 Further manipulation gives the next lemma. We omit the proof.

\begin{lemma}
To calculate $\{ \wt I_k, J_k, \wt\Yw_k: k\le m  \}$, we need  only the input $\{ I_k,\Yw_k: k\le m  \}$. 
\end{lemma} 	

%

 
 The next lemma is nontrivial and we include a complete proof.  
  
\begin{lemma}\label{lm:DR}   The identity 
$\Dop\bigl(\Dop(A, I), Y\bigr) = \Dop\bigl( \Dop(A, \Rop(I, Y)), \Dop(I, Y)\bigr)$ holds whenever the sequences $I, A, Y$ are such that the operations are well-defined.  
\end{lemma} 

\begin{proof}
Choose $(Z_j)$ and $(B_j)$ so that $Z_j/Z_{j-1}=I_j$ and $B_j/B_{j-1}=A_j$.  Then the output of $\Dop(A, I)$  is  the ratio sequence  $(\wt B_\ell/\wt B_{\ell-1})_{\ell}$  of  
\[  \wt B_\ell=\sum_{k:\,k\le \ell}   B_k \prod_{i=k}^\ell I_i . \]  
Next,  the output of 
 $\Dop(\Dop(A, I), Y)$ is  the ratio sequence $(H_m/H_{m-1})_m$  of 
\[   H_m= \sum_{\ell:\,\ell\le m}   \wt B_\ell    \prod_{j=\ell}^m Y_j  
= \sum_{k:\,k\le m}  B_k  \sum_{\ell=k}^{m}  \Bigl( \; \prod_{i=k}^\ell I_i \Bigr) \Bigl(\;  \prod_{j=\ell}^m Y_j\Bigr)  . \]

Similarly,  define first 
\[  \wt Z_j=\sum_{k:\,k\le j}   Z_k\prod_{i=k}^j Y_i 
\quad\text{and}\quad 
\wt I_j  =\frac{\wt Z_j}{\wt Z_{j-1}} 
\]
so that $\wt I=\Dop(I, Y)$.    Let $\wt Y=\Rop(I, Y)$ and then 
\[  \wc B_\ell=\sum_{k:\,k\le \ell}    B_k   \prod_{i=k}^\ell \wt Y_i. \]  

Then  the output of  $\Dop\bigl( \Dop(A, \Rop(I, Y)), \Dop(I, Y)\bigr)=\Dop( \Dop(A, \wt Y), \wt I\,)$ is  the ratio sequence of 
\[   \wt H_m= \sum_{\ell:\,\ell\le m}   \wc B_\ell   \prod_{j=\ell}^m \wt I_j  
= \sum_{k:\,k\le m}   B_k   \sum_{\ell=k}^{m}  \Bigl(  \; \prod_{i=k}^\ell \wt Y_i\Bigr) \Bigl( \;  \prod_{j=\ell}^m \wt I_j\Bigr). \]

The lemma follows from $H= \wt H$,  which we  verify by checking  that for all $k\le m$,  
 \be\label{DR45} 
  \sum_{\ell=k}^{m}  \Bigl( \; \prod_{i=k}^\ell I_i \Bigr) \Bigl(\;  \prod_{j=\ell}^m Y_j\Bigr) 
  =  \sum_{\ell=k}^{m}  \Bigl(  \; \prod_{i=k}^\ell \wt Y_i\Bigr) \Bigl( \;  \prod_{j=\ell}^m \wt I_j\Bigr). 
  \ee
We fix $k$ and prove this by    induction on $m$.   The case $m=k$ follows from   \eqref{m:R} and \eqref{m:801a}:  
  \[ \wt Y_k  \wt I_k  =  \frac{ Y_k\bigl( \frac{I_k}{J_{k-1}}  + 1 \bigr)}{\frac1{I_k}+ \frac1{J_{k-1}}}  = I_k Y_k.   \] 
  
  To prove the induction step, we introduce two auxiliary quantities by adding   terms separately on the left and right of \eqref{DR45}: let 
 \[  T_m=   J_{k-1}   \prod_{j=k}^m Y_j  
 +  \sum_{\ell=k}^{m}  \Bigl( \; \prod_{i=k}^\ell I_i \Bigr) \Bigl(\;  \prod_{j=\ell}^m Y_j\Bigr) 
 \] 
 and 
 \[  \wt T_m=  \sum_{\ell=k}^{m}  \Bigl(  \; \prod_{i=k}^\ell \wt Y_i\Bigr) \Bigl( \;  \prod_{j=\ell}^m \wt I_j\Bigr)   +    \Bigl(\; \prod_{i=k}^m \wt Y_i\Bigr) J_m.
 \]  
Repeated application of \eqref{m:Ja} implies that  $ J_{k-1}   \prod_{j=k}^m Y_j = (\; \prod_{i=k}^m \wt Y_i) J_m$.  Thus   \eqref{DR45} is equivalent to  $\wt T_m=T_m$. 
  
  First observe that $T_{m+1}=T_{m} \wt I_{m+1}$ for $m\ge k$.  This follows from checking inductively the pair of identities  
  \[   \frac{T_m}{ \prod_{i=k}^m I_i} = J_m 
  \quad\text{and}\quad   \frac{T_{m+1}}{T_{m}}= \wt I_{m+1} 
  \quad\text{for } \ m\ge k.  
  \] 
 This relies on the first equalities of the iterative formulas \eqref{m:801a} and \eqref{m:Ja}. 
  
 Now assume that  $\wt T_m=T_m$. We show that $\wt T_{m+1}=\wt T_{m} \wt I_{m+1}$ which then implies $\wt T_{m+1}=T_{m+1}$. 
 \begin{align*}
 \wt T_{m+1}&=  \sum_{\ell=k}^{m}  \Bigl(  \; \prod_{i=k}^\ell \wt Y_i\Bigr) \Bigl( \;  \prod_{j=\ell}^m \wt I_j\Bigr)  \wt I_{m+1}  +  \Bigl(  \; \prod_{i=k}^{m}  \wt Y_i\Bigr)  \wt Y_{m+1} \wt I_{m+1}  +    \Bigl(\; \prod_{i=k}^m \wt Y_i\Bigr) \wt Y_{m+1} J_{m+1} \\
 &=\wt T_m   \wt I_{m+1}  
 +  \Bigl(\; \prod_{i=k}^m \wt Y_i\Bigr) \bigl( - J_m   \wt I_{m+1}  +   \wt Y_{m+1} \wt I_{m+1} +  \wt Y_{m+1} J_{m+1} \bigr)  . 
 \end{align*} 
 The last expression in parentheses vanishes because 
 $J_m   \wt I_{m+1}=Y_{m+1}(I_{m+1}+J_m) $,   $\wt Y_{m+1} \wt I_{m+1}= Y_{m+1} I_{m+1}$ and  $\wt Y_{m+1} J_{m+1}= Y_{m+1} J_{m}$. 
  \end{proof}

\medskip

\section{The inverse-gamma polymer}  \label{app:inv-gam} 

This section reviews the ratio-stationary inverse-gamma polymer introduced in \cite{sepp-12-aop-corr} and then constructs   the two-variable jointly ratio-stationary process, which is a special case of the multivariate construction  from the forthcoming work \cite{fan-sepp-20+}. 

\subsection{Inverse-gamma weights} 
\label{app:i-g-w} 
Recall the   inverse gamma distribution from \eqref{invga} and the mean from \eqref{invga7}. 
%
%

 
	\begin{lemma}\label{v:lm} 
Define the mapping 	$(I,J,Y)\mapsto(I',J',Y')$ on $\R_{>0}^3$ by 
\be\label{IJw19} 
			I'=Y\biggl( 1+\frac{I}{J}\biggr)\,,\quad	J'=Y\biggl( 1+\frac{J}{I}\biggr)  \,,\quad 
			Y'=\frac1{I^{-1}+ J^{-1}}    .  
		\ee 
	\begin{enumerate}[{\rm(a)}] \itemsep=3pt 
	\item  $(I,J,Y)\mapsto(I',J',Y')$ is an involution. 
			\item   Let $\alpha, \beta>0$.    Suppose that $I,J,Y$ are   independent random variables with distributions  $I\sim{\rm Ga}^{-1}(\alpha)$, 
		$J\sim{\rm Ga}^{-1}(\beta)$ and $Y\sim{\rm Ga}^{-1}(\alpha+\beta)$.   Then the triple  $( I', J', Y')$  has the same distribution as  $(I,J,Y)$. 
		\end{enumerate}
	\end{lemma} 
	
\begin{proof}
Part (b) follows by applying the beta-gamma algebra to the reciprocals that satisfy 
\[  \frac1{I'}=Y^{-1}\frac{I^{-1}}{I^{-1}+ J^{-1}} \,,\quad
\frac1{J'}=Y^{-1}\frac{J^{-1}}{I^{-1}+ J^{-1}}\quad\text{and}\quad 
\frac1{Y'}= I^{-1}+ J^{-1}     . 
\qedhere\]
\end{proof}	

\medskip 

%
%
 	
	\begin{lemma}\label{m:Lem-I}  Let $0<\rho<\sigma$.  Let $I=(I_k)_{k\tsp\in\tsp\Z}$ and $\Yw=(\Yw_j)_{j\in\Z}$  be mutually independent random variables such that $I_k\sim{\rm Ga}^{-1}(\rho)$ and $\Yw_j\sim{\rm Ga}^{-1}(\sigma)$.  Use mappings \eqref{m:DSR} to define 
	\[  \wt I=\Dop(I,\Yw)\,\quad \wt\Yw=R(I,\Yw)\quad\text{and}\quad J=S(I,\Yw).\]      Let $V_k= (\{\wt I_j\}_{j\leq k}, \,J_k, \,\{\wt \Yw_j\}_{j\leq k})$.  
		\begin{enumerate}[{\rm(a)}] \itemsep=3pt
			\item   $\{V_k\}_{k\tsp\in\tsp\Z}$ is a stationary, ergodic process.  
			For each $k\in\Z$, the random variables $\{\wt I_j\}_{j\leq k}, \,J_k, \,\{\wt \Yw_j\}_{j\leq k} $ are mutually independent with marginal distributions  
			\[  \text{ $\wt I_j\sim{\rm Ga}^{-1}(\rho)$, \ \ $\wt\Yw_j\sim{\rm Ga}^{-1}(\sigma)$ \ and \   $J_k\sim{\rm Ga}^{-1}(\sigma-\rho)$.   } \] 
			
			\item  $\wt I$ and $\wt\Yw$  are    independent sequences of i.i.d.\ variables.  
		\end{enumerate}
	\end{lemma}
	
	\begin{proof} 
We start by verifying \eqref{Iw} 	to guarantee that the processes $\wt I$, $\wt\Yw$ and $J$ are almost surely  well-defined and finite.  To this end we   show that
\be\label{98.04} 
  \sum_{j=-\infty}^0  {\Yw_{j}} \prod_{i=j+1}^{0}\frac{\Yw_i}{I_i} 
  <\infty \quad\text{ with probability one.}    
 \ee 
  Rewrite the above as 
  \be\label{98.08} 
 \sum_{j=-\infty}^0  {\Yw_{j}} \prod_{i=j+1}^{0}\frac{\Yw_i}{I_i} =   \sum_{j=-\infty}^0  {\Yw_{j}}  \, e^{ \sum_{i=j+1}^0(\log\Yw_i-\log I_i)  }   
  =  \sum_{j=-\infty}^0  e^{\,j\delta} {\Yw_{j}} \,  e^{-j\delta+ \sum_{i=j+1}^0(\log\Yw_i-\log I_i)  }
\ee
where we can choose  $\delta>0$ to satisfy 
  \be\label{98.09}  \E[ \log\Yw_i-\log I_i]= -\psi_0(\sigma)+\psi_0(\rho)<-3\delta<0 \ee
because  $\psi_0$ is  strictly increasing.  Hence  almost surely for large enough $j<0$, 
 \be\label{98.10}  \sum_{i=j+1}^0(\log\Yw_i-\log I_i) \le 2j\delta   . 
\ee  
The estimate below shows that, for any $\delta>0$,  $\sup_{j\le0}  e^{j\delta} {\Yw_{j}}$ is almost surely finite: 
\begin{align*}
\sum_{j\le -1} \P(Y_j\ge e^{-j\delta}) =  \sum_{j\le -1} \P(\log Y_j\ge {-j\delta}) 
\le \sum_{j\le -1} \frac{\E[(\log Y_j)^2\tsp]}{j^2\delta^2} <\infty . 
\end{align*}
The almost sure convergence of the series \eqref{98.04} has been verified.  We turn to the proof of the lemma. 
 
\medskip 
	
	  Part (b) follows from part (a) by dropping the $J_k$ coordinate and letting $k\to\infty$.    Stationarity and ergodicity of $\{V_k\}$ follow from its construction as a mapping applied to the independent i.i.d.\ sequences $I$ and $\Yw$.

	\medskip
		
		The distributional claims in part (a) are proved   by coupling $(\wt I_k, J_{k-1}, \wt\Yw_k)_{k\tsp\in\tsp\Z}$ with another sequence of processes (indexed by $N$ below) whose distribution we know.  
		Let $Z$ be a fixed ${\rm Ga}^{-1}(\sigma-\rho)$ variable that is independent of $(I,\Yw)$.
		
		For each $N\geq 0$,   construct a process  $(\wh I_k^{N}, \wh J_{k-1}^{N}, \wh\Yw_k^{N})_{k\ge -N+1}$ as follows.    First let $\wh J_{-N}^{N}=Z$.    Then   iterate the steps 
		\be\label{98.3} \begin{aligned}
			(\wh I_k^{N},\wh J_k^{N},\wh\Yw_k^{N})= \Theta\big(I_k,\wh J_{k-1}^{N},\Yw_k\big) \quad \text{for }k\geq -N+1,
		\end{aligned}\ee
		where $\Theta(I,J,Y)=(I',J',Y')$ is the involution \eqref{IJw19}  in Lemma \ref{v:lm}.
		We claim that for each $k\in \Z$,
		\be\label{98.8} 
		\lim_{N\to \infty} \wh J_{k}^{N}=J_{k}\,, \quad \lim_{N\to\infty} \wh I_{k}^{N}=\wt I_{k} \quad \text{and}\quad  \lim_{N\to\infty} \wh\Yw_{k}^{N}=\wt\Yw_{k}  \quad\text{in probability}. 
		\ee  
		
Applying \eqref{J34} gives 		
		\begin{equation}\label{98.91}
			{J_k}- {\wh J_k^{N}}= (  {J_{-N}} - {\wh J_{-N}^{N}}) \prod_{i=-N+1}^k\frac{\Yw_i}{I_i} = (  {J_{-N}} - Z) \prod_{i=-N+1}^k\frac{\Yw_i}{I_i} 
		\end{equation}
from which 
\be\label{98.35}
 \abs{\, {J_k}- {\wh J_k^{N}}\,} \le e^{-N\delta} ( {J_{-N}} +  Z) \,  e^{N\delta + \sum_{i=-N+1}^k(\log\Yw_i-\log I_i)  } 
\ee
where we chose $\delta>0$ as in \eqref{98.09}.  Hence the last exponential factor above vanishes almost surely as $N\to\infty$.  
 The equation  
\be\label{J36.4} 
J_k=\frac{\wt Z_k}{Z_k} =  \sum_{j:\,j\le k}    \Yw_j\,\prod_{i=j+1}^{k} \frac{\Yw_i}{I_i} 
\ee 
shows that $\{J_k\}$ is a finite stationary process, and consequently $e^{-N\delta} J_{-N}\to0$ in probability.  \eqref{98.35} implies the first limit in probability  in \eqref{98.8}.


 

To get the  second limit  in \eqref{98.8}, apply  \eqref{98.3} and the first limit as $N\to\infty$:   
\be\label{98.73} 
\wh I_{k}^{N}  = \Yw_k\biggl( \frac{I_k}{\wh J^N_{k-1}}  + 1 \biggr) 
\; \overset{P}\longrightarrow \;   \Yw_k\biggl( \frac{I_k}{J_{k-1}}  + 1 \biggr)=  \wt I_{k}.  
\ee
For the  last limit  in \eqref{98.8},
\be\label{98.76} 
\wh\Yw_{k}^{N}  =  (I_k^{-1}+ \bigl(\wh J^N_{k-1})^{-1}\bigr)^{-1} 
\; \overset{P}\longrightarrow \;   \bigl(I_k^{-1}+ J_{k-1}^{-1}\bigr)^{-1} =  \wt \Yw_k.  
\ee
		
		\medskip
		
		Next, we prove the following claim  for each $N\geq 1$: 
		\be\label{98.7} \begin{aligned} 
			&\text{for each $m\ge -N+1$, the random   variables $\wh I^N_{-N+1},\dotsc, \wh I^N_m, \wh J^N_m,  \wh\Yw^N_m ,\dotsc, \wh\Yw^N_{-N+1} $}\\
			&\text{are mutually independent with marginal distributions} \\  
			&\wh I^N_k\sim{\rm Ga}^{-1}(\rho)\,, \quad \wh J^N_m\sim{\rm Ga}^{-1}(\sigma-\rho), \quad\text{and}\quad  \wh\Yw^N_j\sim{\rm Ga}^{-1}(\sigma). 
		\end{aligned} \ee
		
		By construction    $(I_{-N}, \wh J^N_{-N}, \Yw_{-N})\sim {\rm Ga}^{-1}(\rho) \otimes  {\rm Ga}^{-1}(\sigma-\rho) \otimes {\rm Ga}^{-1}(\sigma)$.      
		The base case $m=-N+1$ of \eqref{98.7}  comes by applying Lemma \ref{v:lm} to  the mapping \eqref{98.3} with $k=-N+1$.   
		
		Assume  \eqref{98.7}   holds for $m$.   By the induction assumption and by the  independence of the ingredients that go into the construction,  
		\[  \wh I^N_{-N+1},\dotsc, \wh I^N_m, (I_{m+1}, \wh J^N_{m}, \Yw_{m+1}), \wh\Yw^N_m ,\dotsc, \wh\Yw^N_{-N+1}  \] 
		are independent.  Furthermore,    $(I_{m+1}, \wh J^N_{m}, \Yw_{m+1})\sim {\rm Ga}^{-1}(\rho) \otimes  {\rm Ga}^{-1}(\sigma-\rho) \otimes {\rm Ga}^{-1}(\sigma)$.     By Lemma \ref{v:lm}, the mapping \eqref{98.3} turns the triple $(I_{m+1}, \wh J^N_{m}, \Yw_{m+1})$ into  $(\wh I^N_{m+1}, \wh J^N_{m+1}, \wh\Yw^N_{m+1})\sim {\rm Ga}^{-1}(\rho) \otimes  {\rm Ga}^{-1}(\sigma-\rho) \otimes {\rm Ga}^{-1}(\sigma)$.    Statement \eqref{98.7} has been extended  to $m+1$.  The proof of \eqref{98.7} is complete.
		
		Part (a) follows from \eqref{98.8} and \eqref{98.7}. 
	\end{proof}

Next we describe  a distributional  fixed point of  the mapping 
$(I^1, I^2)\mapsto \bigl( \Dop(I^1,Y), \Dop(I^2,Y) \bigr) $ when $Y$ is an i.i.d.\ inverse-gamma sequence. 
   Let  
  $\sigma>\alpha_1>\alpha_2>0$.  Let $A^1=(A^1_j)_{j\in\Z}$,  $A^2=(A^2_j)_{j\in\Z}$, $Y=(Y_j)_{j\in\Z}$ be   mutually independent i.i.d.\ sequences with marginals   $A^k_j\sim{\rm Ga}^{-1}(\alpha_k)$ for $k\in\{1,2\}$   and $Y_j\sim{\rm Ga}^{-1}(\sigma)$.    Define a jointly distributed pair of  boundary sequences by 
  $(I^1, I^2)= \bigl( A^1, \Dop( A^2, A^1)\bigr)$.  From these and bulk weights $Y$, define jointly distributed output variables: 
  \[  
   \wt I^k=\Dop(I^k,Y),  \quad J^k=\Sop(I^k,Y), \quad\text{and} \quad \wc Y^k=\Rop(I^k,Y)
   \quad \text{  for } k\in\{1,2\}.  
   \]  
  
\begin{lemma}\label{lm:DR5}   We have the following properties. 
\begin{enumerate}[{\rm(i)}]   \itemsep=3pt 
\item Marginally $I^2$ is a sequence of i.i.d.\ ${\rm Ga}^{-1}(\alpha_2)$ variables. 
\item For   fixed   $k\in\{1,2\}$ and $m\in\Z$, the random variables   $\{\wt I^k_j\}_{j\le m}$, $J^k_m$, and $\{\wc Y^k_j\}_{j\le m}$ are mutually independent with marginal distributions $\wt I^k_j\sim\text{\rm Ga}^{-1}(\alpha_k)$, $J^k_m\sim\text{\rm Ga}^{-1}(\sigma-\alpha_k)$, and $\wc Y^k_j\sim\text{\rm Ga}^{-1}(\sigma)$. 
\item For  fixed   $k\in\{1,2\}$,   $\wt I^k$ and $\wc Y^k$ are mutually independent sequences of i.i.d.\ random variables with marginal distributions $\wt I^k_j\sim\text{\rm Ga}^{-1}(\alpha_k)$ and $\wc Y^k_j\sim\text{\rm Ga}^{-1}(\sigma)$. 

\item  $(\wt I^1, \wt I^2)\deq(I^1, I^2)$, in other words, we have  a distributional fixed point for this joint polymer operator. 

\item   For any $m\in\Z$, the random variables $\{ I^2_i\}_{i\le m}$ and $\{I^1_j\}_{j\ge m+1}$ are mutually independent.  


\end{enumerate} 
\end{lemma} 

\begin{proof}
  Parts (i)--(iii) come from Lemma \ref{m:Lem-I}. 

For part  (iv),   the marginal distributions of $\wt I^1$ and $\wt I^2$ are the correct ones  by Lemma \ref{lm:DR5}(iii). To establish the correct joint distribution,   the definition of $(I^1, I^2)$ points us to find an i.i.d.\ ${\rm Ga}^{-1}(\alpha_2)$  random sequence $Z$ that is independent of $\wt I^1$ and satisfies $\wt I^2=   \Dop( Z, \wt I^1)$.  From the definitions and  Lemma \ref{lm:DR}, 
\begin{align*}
\wt I^2=\Dop(I^2,Y) =  \Dop\bigl( \Dop( A^2, I^1),Y\bigr)
= \Dop\bigl( \Dop(A^2, \Rop(I^1, Y)), \Dop(I^1, Y)\bigr)
= \Dop\bigl( \Dop(A^2, \wc Y^1), \wt I^1\bigr).
\end{align*} 
By assumption $A^2, I^1, Y$ are independent. Hence by Lemma \ref{lm:DR5}(iii)   $A^2, \wc Y^1, \wt I^1$ are independent.  So we take $Z=\Dop(A^2, \wc Y^1)$  which is an  i.i.d.\ ${\rm Ga}^{-1}(\alpha_2)$    sequence by Lemma \ref{lm:DR5}(iii).  This proves part  (iv).

 We know that marginally $I^1$ and  $I^2$ are i.i.d.\ sequences.   \eqref{m:800} and \eqref{m:801} show  that variables $\{ I^2_i\}_{i\le m}$ are functions of $(\{A^2_i\}_{i\le m}\,, \{I^1_i\}_{i\le m})$ which are independent of $\{I^1_j\}_{j\ge m+1}$. 
\end{proof}

\medskip

\subsection{Two jointly ratio-stationary polymer processes}
\label{sec:stat-pol} 
Pick $0<\lambda<\rho<\sigma$ and a base vertex $u=(u_1,u_2)\in\Z^2$.  We construct   two coupled polymer  processes  $Z^\lambda_{u,\bbullet}$ and $Z^\rho_{u,\bbullet}$   on the nonnegative  quadrant $u+\Z_{\ge0}^2$  such that the joint  process
$\{ (Z^\lambda_{u,y}/Z^\lambda_{u,x},  Z^\rho_{u,y}/Z^\rho_{u,x}): x,y\in u+\Z_{\ge0}^2\}$ 
of  ratios is  stationary under translations $(x,y)\mapsto(x+v, y+v)$.  
Both processes use the same   i.i.d.\ ${\rm Ga}^{-1}(\sigma)$ weights $\{\Yw_x\}_{x\,\in\,u+\Z_{>0}^2}$ in the bulk.  They have  boundary conditions on the positive  $x$- and $y$-axes emanating from the origin at $u$, coupled in a way described in the next theorem. 

For $\alpha\in\{\lambda, \rho\}$, we repeat here  the definition of the process $Z^\alpha_{u,\bbullet}$    given earlier in \eqref{g:Z1}.  
 On the boundaries of the quadrant we have strictly positive   boundary weights  
	$\{I^\alpha_{u+i\evec_1}, J^\alpha_{u+j\evec_2}:i, j\in\Z_{>0}\}$.   Put  $Z^\alpha_{u,u}=1$ and on the boundaries 
	\be\label{Zr11a} Z^\alpha_{u,\,u+\,k\evec_1}=\prod_{i=1}^k I^\alpha_{u+i\evec_1} 
	\quad\text{and}\quad
	Z^\alpha_{u,\,u+\,l\evec_2}= \prod_{j=1}^l  J^\alpha_{u+j\evec_2}  \quad\text{ for } k,l\ge 1.   \ee 
	In the bulk 
	for $x=(x_1,x_2)\in u+ \Z_{>0}^2$, 
	\be\label{Zr12a}\begin{aligned} 
	Z^\alpha_{u,\,x}&= \sum_{k=1}^{x_1-u_1} \biggl( \;\prod_{i=1}^k I^\alpha_{u+i\evec_1}\biggr) Z_{u+k\evec_1+\evec_2, \,x}  + \sum_{\ell=1}^{x_2-u_2}\biggl(  \;\prod_{j=1}^\ell  J^\alpha_{u+j\evec_2} \biggr)  Z_{u+\evec_1+\ell \evec_2, \,x}   \\
	&=\bigl( Z^\alpha_{u,\,x-\evec_1}+ Z^\alpha_{u,\,x-\evec_2} \bigr)  \Yw_x.
	\end{aligned} \ee
	 $Z^\alpha_{u,\bbullet}$ does not use a weight at the base point $u$. $Z_{x,y}$  above is the partition function   \eqref{m:G807} that uses the bulk weights $\Yw$. 
Define ratio variables  for vertices $x\in u+\Z_{>0}^2$  by 
\be\label{IJ85}    I^\alpha_x=Z^\alpha_x/Z^\alpha_{x-\evec_1} 
\quad\text{and}\quad 
J^\alpha_x=Z^\alpha_x/Z^\alpha_{x-\evec_2}. 
\ee

The next theorem describes the jointly stationary process that is used in the proofs of Section \ref{sec:estim}.  Since those arguments work with the $J$-ratio variables on the $y$-axis, in order to tailor this theorem to its application we construct the joint process on the right half-plane and then restrict that process to the first quadrant. Consequently the upper half-plane of Sections \ref{app:pol-H} and \ref{app:i-g-w}   has been turned into the right half-plane, and thereby horizontal has become vertical. 
An important part of the theorem   is the independence of various collections of ratio variables.  These are illustrated in Figure \ref{fig:indIJ}. 
 
 \begin{figure}
	\begin{tikzpicture}[scale=0.6, every node/.style={transform shape}]
	\def\s{1.6}
	\draw[->,line width=0.08mm,>=latex]   (0,0) --(0,5); 
	\draw[->,line width=0.08mm,>=latex]   (0,0) --(5,0); 
	\node [scale=\s][below] at (0,0) {$u$};
	\draw [fill] (0,0) circle [radius=0.07];
	\node [scale=\s][below] at (3.2,3.5) {$x$};
	\draw [dashed][line width=0.01cm] (3.6,0) -- (3.6,5);
	\node [scale=\s][below] at (4.2,2) {$J^\lambda$};
	\node [scale=\s][below] at (4.2,5.2) {$J^\rho$};
	\draw [fill] (3.6,3.3) circle [radius=0.07];
	\end{tikzpicture}
	\qquad\quad 
	\begin{tikzpicture}[scale=0.6, every node/.style={transform shape}]
	\def\s{1.6}
	\draw[->,line width=0.08mm,>=latex]   (0,0) --(0,5); 
	\draw[->,line width=0.08mm,>=latex]   (0,0) --(5,0); 
	\node [scale=\s][below] at (0,0) {$u$};
	\draw [fill] (0,0) circle [radius=0.07];
	\node [scale=\s][below] at (2,2) {$v$};
	\draw [fill] (2,2) circle [radius=0.07];
	\draw [dashed][line width=0.01cm] (2,2) -- (2,5);
	\draw [dashed][line width=0.01cm] (2,2) -- (5,2);
	\node [scale=\s][below] at (4,2.8) {$I^\alpha$};
	\node [scale=\s][below] at (2.6,4) {$J^\alpha$};
	\end{tikzpicture}
	\qquad\quad  
	\begin{tikzpicture}[scale=0.6, every node/.style={transform shape}]
	\def\s{1.6}
\draw[->,line width=0.08mm,>=latex]   (0,0) --(0,5); 
\draw[->,line width=0.08mm,>=latex]   (0,0) --(5,0); 
\node [scale=\s][below] at (0,0) {$u$};
\draw [fill] (0,0) circle [radius=0.07];
\node [scale=\s][below] at (3.3,3.6) {$v$};
\draw [fill] (3,3) circle [radius=0.07];
\draw [dashed][line width=0.01cm] (3,3) -- (0,3);
\draw [dashed][line width=0.01cm] (3,3) -- (3,0);
\node [scale=\s][below] at (1.6,3.8) {$I^\alpha$};
\node [scale=\s][below] at (3.5,2) {$J^\alpha$};
	\end{tikzpicture}
	\caption{\small The independent ratio variables from Theorem \ref{thm:st-lpp}. Left: $J^\lambda$ below $x$ and $J^\rho$ above $x$ from part (i). Middle and right: $I^\alpha$ and $J^\alpha$ ratios on down-right lattice paths from part (ii).}
	\label{fig:indIJ} 
\end{figure}
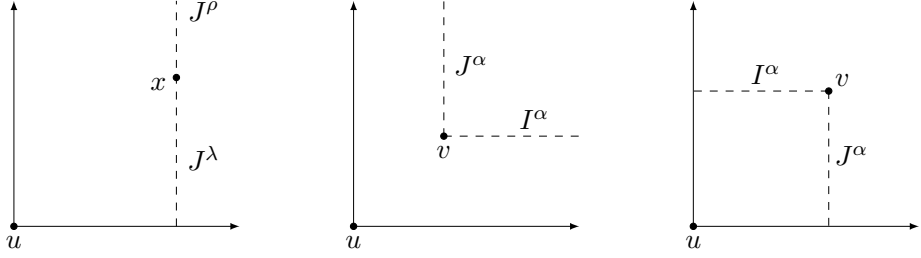
 
\begin{theorem}\label{thm:st-lpp}   Let  $0<\lambda<\rho<\sigma$ and $u\in\Z^2$.  There exists a coupling of the boundary weights $\{I^\lambda_{u+i\evec_1},I^\rho_{u+i\evec_1}, 	J^\lambda_{u+j\evec_2}$, $J^\rho_{u+j\evec_2}:i, j\in\Z_{>0}\}$  such that the joint process $(Z^\lambda_{u,\bbullet}\,, Z^\rho_{u,\bbullet})$ has the following properties.  
\begin{enumerate}[{\rm(i)}] \itemsep=5pt 
\item {\rm(Joint)}     The joint process of ratios is stationary: for each $v\in u+\Z_{\ge0}^2$,  
\be\label{Zr17}    \biggl\{\biggl(\frac{Z^\lambda_{u,v+x}}{Z^\lambda_{u,v}}, \frac{Z^\rho_{u,v+x}}{Z^\rho_{u,v}}\biggr): x\in\Z_{\ge0}^2\biggr\}
\deq
 \bigl\{(Z^\lambda_{u,u+x}, Z^\rho_{u, u+x}): x\in\Z_{\ge0}^2\bigr\}.
 \ee
{\rm(}On the right above the implicit denominators $Z^\lambda_{u,u}=Z^\rho_{u,u}=1$ were omitted.{\rm)}   
The following independence property holds along vertical lines: for each $x\in u+\Z_{>0}^2$, the variables $\{ J^\lambda_{x+j\evec_2}: u_2-x_2+1\le j\le 0\}$ and $\{ J^\rho_{x+j\evec_2}: j\ge 1\}$ are mutually independent.  

\item {\rm(Marginal)}  For both $\alpha\in\{\lambda, \rho\}$ and for each $v=(v_1,v_2)\in u+\Z_{\ge0}^2$,  the ratio variables  $\{I^\alpha_{v+i\evec_1}, J^\alpha_{v+j\evec_2}:i, j\in\Z_{>0}\}$	 are mutually independent with marginal distributions %
 \[
I^\alpha_{v+i\evec_1}\sim\text{\rm Ga}^{-1}(\sigma-\alpha) \quad \text{ and }\quad J^\alpha_{v+j\evec_2}\sim\text{\rm Ga}^{-1}(\alpha). 
 \]
The same is true of the variables $\{I^\alpha_{v-i\evec_1}, J^\alpha_{v-j\evec_2}: 0\le i <  v_1-u_1,\, 0\le j< v_2-u_2\}$. 	

\item {\rm(Monotonicity)}  The boundary weights can be coupled with  i.i.d.\ ${\rm Ga}^{-1}(\sigma)$  weights $\{\eta_{u+i\evec_1}, \eta_{u+j\evec_2}: i,j\ge 1\}$ independent of the bulk weights $\Yw$ so that these inequalities hold almost surely for all $i,j\ge 1$: 
\be\label{L25}
	\eta_{u+i\evec_1}\le I^{\lambda}_{u+i\evec_1} \le I^{\rho}_{u+i\evec_1}
	\quad\text{and}\quad 
	\eta_{u+j\evec_2}\le J^{\rho}_{u+j\evec_2} \le  J^{\lambda}_{u+j\evec_2}.  
	\ee
\end{enumerate} 
\end{theorem} 

\begin{proof} 
We  construct  a  joint partition function  process $(L^\lambda_{x}, L^\rho_{x})_{x\,\in\,u+\Z_{\ge0}\times\Z}$ on the discrete  right half-plane  $u+\Z_{\ge0}\times\Z$ with origin fixed at $u$.  The restriction of this process to the quadrant $u+\Z_{\ge0}^2$ then furnishes the process $(Z^\lambda_{u,\bbullet}\,, Z^\rho_{u,\bbullet})$ whose properties are claimed in the theorem. 
 
In the interior  put i.i.d.\ ${\rm Ga}^{-1}(\sigma)$ weights $\Yvec=\{\Yw_x: x_1>u_1\}$ as before.  (We write some weight configurations with bold symbols to distinguish the notation of this proof from earlier notation.) 
  For $\alpha\in\{\lambda, \rho\}$ let $\Yvec^\lambda=\{Y^\lambda_j\}_{j\in\Z}$ and $\Yvec^\rho=\{Y^\rho_j\}_{j\in\Z}$   be independent  sequences of  i.i.d.\ variables with marginal distributions  $Y^\alpha_j\sim$ ${\rm Ga}^{-1}(\alpha)$, independent of $\Yvec$.   
 From these we define the  boundary weights  $\Jvec^\lambda=\{J^\lambda_{u+j\evec_2}\}_{j\in\Z}$ and  $\Jvec^\rho=\{J^\rho_{u+j\evec_2}\}_{j\in\Z}$ 
on the $y$-axis through $u$ by the equation  
     $(\Jvec^\rho, \Jvec^\lambda)=(\Yvec^\rho, \Dop(\Yvec^\lambda, \Yvec^\rho))$. 
      $\Dop$  is the partition function  operator  from \eqref{m:DSR}.   This gives a pair of coupled sequences $(\Jvec^\rho, \Jvec^\lambda)$.  Marginally  $\{J^\alpha_{u+j\evec_2}\}_{j\in\Z}$ are i.i.d.\ ${\rm Ga}^{-1}(\alpha)$. 

  For $\alpha\in\{\lambda, \rho\}$ define the partition function values on the $y$-axis centered at $u$   by 
\[
L^\alpha_{u}=1 \quad\text{and}\quad 
  \frac{L^\alpha_{u+j\evec_2}}{L^\alpha_{u+(j-1)\evec_2}} = J^\alpha_{u+j\evec_2}\quad \text{for } j\in\Z. 
\]
   Complete the definitions by putting, again for $\alpha\in\{\lambda, \rho\}$ and now for $x\in u+\Z_{>0}\times\Z$, 
\be\label{Ly4}  L^\alpha_{x}=\sum_{j: \tspa j\tspa\le\tspa x_2-u_2}  L^\alpha_{u+j\evec_2}   \tsp Z_{u+\evec_1+j\evec_2,\tsp x} , \quad 
I^\alpha_x=\frac{L^\alpha_{x}}{L^\alpha_{x-\evec_1}} 
\quad\text{and}\quad
J^\alpha_x=\frac{L^\alpha_{x}}{L^\alpha_{x-\evec_2}}. 
\ee
As in \eqref{m:800}, the series converges    because the boundary variables $J^\alpha$ are stochastically larger than the bulk weights.  This follows from the distributional properties established below. 
The evolution in \eqref{Ly4} satisfies a semigroup property from vertical line to line:  
for each $k\ge0$ the values $L^\alpha_x$   for $x_1\ge u_1+k+1$ satisfy 
\be\label{Ly7}    L^\alpha_{x}=\sum_{j: \,j\,\le\, x_2-u_2}    L^\alpha_{u+k\evec_1+j\evec_2} \tspa Z_{u+(k+1)\evec_1+j\evec_2,\tsp x}. 
\ee

For $k\ge 0$, denote the sequences of $J$-ratios  on the vertical line shifted by $k\evec_1$   by 
$\Jvec^{\alpha,k}=\{J^{\alpha,k}_j\}_{j\in\Z}=\{J^\alpha_{u+k\evec_1+j\evec_2}\}_{j\in\Z}$ and the sequences of weights by $Y^k=\{Y^k_j\}_{j\in\Z}=\{\Yw_{u+k\evec_1+j\evec_2}\}_{j\in\Z}$.    $\Jvec^{\alpha,0}$ is the original boundary sequence $\Jvec^\alpha$ we began with.  One verifies inductively that  $\Jvec^{\alpha,k}=\Dop( \Jvec^{\alpha,k-1}, Y^k)$ for each $k\ge 1$ and $\alpha\in\{\lambda, \rho\}$.

 Apply Lemma \ref{lm:DR5} with parameters  $(\sigma, \alpha_1, \alpha_2)=(\sigma,\rho, \lambda)$. Directly from the definition  $(\Jvec^\rho, \Jvec^\lambda)=(\Yvec^\rho, \Dop(\Yvec^\lambda, \Yvec^\rho))$ follows that   $( \Jvec^\rho, \Jvec^\lambda)$ has the distribution of $(I^1, I^2)$ in Lemma \ref{lm:DR5}.   Repeated application of Lemma \ref{lm:DR5}(iv)  implies the distributional equality  $( \Jvec^{\rho,k},  \Jvec^{\lambda,k})\deq( \Jvec^\rho, \Jvec^\lambda)$  for all $k\ge 0$. 
 Since the $J$-ratios have the same joint distribution on each vertical line, 
  the semigroup property \eqref{Ly7} implies  that the entire process of ratios is invariant under translations that keep it in the half-space:  for $z\in \Z_{\ge0}\times\Z$, 
\be\label{L55} \begin{aligned}  
&\{ I^\lambda_{z+x+\evec_1}, I^\rho_{z+x+\evec_1}, J^\lambda_{z+x}, J^\rho_{z+x}: x\in u+\Z_{\ge0}\times\Z\}   \\
&\qquad 
\deq
\{ I^\lambda_{x+\evec_1}, I^\rho_{x+\evec_1}, J^\lambda_x, J^\rho_x: x\in u+\Z_{\ge0}\times\Z\}. 
\end{aligned}\ee
(The index is $x+\evec_1$ rather than $x$ in the $I$-ratios simply because these are not defined on the boundary where $x_1=u_1$.)    Lemma \ref{lm:DR5}(v) gives the property that, for any $x\in u+\Z_{\ge0}\times\Z$, the ratio variables 
\be\label{L47} 
\{ J^\lambda_{x+j\evec_2}:  j\le 0\} \quad\text{and}\quad \{ J^\rho_{x+j\evec_2}: j\ge 1\}\quad
 \text{are mutually independent.}   
\ee

We claim that for $\alpha\in\{\lambda, \rho\}$  and for any new base point  $v\in u+\Z_{\ge0}\times\Z$,
\be\label{L57} \begin{aligned} 
  &\text{$\{I^\alpha_{v+i\evec_1}, J^\alpha_{v+j\evec_2}:i, j\in\Z_{>0}\}$	are mutually independent with marginal distributions}\\[3pt] 
   &\qquad 
   I^\alpha_{v+i\evec_1}\sim\text{\rm Ga}^{-1}(\sigma-\alpha)\quad \text{ and } \quad  J^\alpha_{v+j\evec_2}\sim\text{\rm Ga}^{-1}(\alpha).
\end{aligned} \ee
Since the joint distribution  is shift-invariant, we can take  $v=u$.  As observed  above,  $\Jvec^\alpha$ is a sequence of i.i.d.\ ${\rm Ga}^{-1}(\alpha)$ random variables  by Lemma \ref{lm:DR5}(i). Thus it suffices to prove the marginal statement about $\{I^\alpha_{u+i\evec_1}:i\ge1\}$ because these variables are a function of $\{J^\alpha_{u+j\evec_2},\, \Yw_{u+(i,j)}:i\ge1,  j\le0\}$ which are independent of $\{ J^\alpha_{u+j\evec_2}: j\ge1\}$. 

The claim for $\{I^\alpha_{u+i\evec_1}:i\ge1\}$ follows from proving inductively the following statement for each $n\ge 1$: 
\be\label{L59} \begin{aligned} 
  &\text{$\{I^\alpha_{u+i\evec_1}, J^\alpha_{u+n\evec_1+j\evec_2}: 1\le i\le n, j\le0\}$	are mutually independent with}\\
  & \text{marginal distributions}\quad 
   I^\alpha_{u+i\evec_1}\sim\text{\rm Ga}^{-1}(\sigma-\alpha)\quad \text{ and } \quad  J^\alpha_{u+n\evec_1+j\evec_2}\sim\text{\rm Ga}^{-1}(\alpha).
\end{aligned} \ee
  Begin with the case $n=1$.   From  the inputs given by boundary weights 
$\{I_j=J^\alpha_{u+j\evec_2}:j\le0\}$ and bulk weights  $\{Y_j=\Yw_{u+\evec_1+j\evec_2}: j\le 0\}$,  equation \eqref{m:801}  computes the ratio weights  $\{\wt I_j=J^\alpha_{u+\evec_1+j\evec_2}:j\le0\}$ and equation \eqref{m:J} gives   $J_0=I^\alpha_{u+\evec_1}$. (Note here the switch between ``horizontal'' and ``vertical''.)   Part of  Lemma \ref{lm:DR5}(ii)  then gives exactly statement \eqref{L59} for $n=1$. (The dual bulk weights $\wc Y_j$ that also appear in  Lemma \ref{lm:DR5}(ii)  are not needed here.)   

Continue inductively.  Assume that \eqref{L59} holds for a given $n$. Then feed into the polymer operators boundary weights 
$\{I_j=J^\alpha_{u+n\evec_1+j\evec_2}:j\le0\}$ and bulk weights  $\{Y_j=\Yw_{u+(n+1)\evec_1+j\evec_2}: j\le 0\}$, all independent of $\{I^\alpha_{u+i\evec_1}: 1\le i\le n\}$.   Compute the ratio weights  $\{\wt I_j=J^\alpha_{u+(n+1)\evec_1+j\evec_2}:j\le0\}$ and  $J_0=I^\alpha_{u+(n+1)\evec_1}$.    Lemma \ref{lm:DR5}(ii) extends the validity of \eqref{L59} to  $n+1$.  Claim \eqref{L57} has been verified. 

To prove the full Theorem \ref{thm:st-lpp} on the quadrant $u+\Z_{\ge0}^2$, take the coupled boundary weights 
$\{ I^\alpha_{u+i\evec_1}, J^\alpha_{u+j\evec_2}:  i,j\ge 1, \alpha\in\{\lambda, \rho\}\}$ as constructed above.   The partition function process $\{Z^\alpha_{u,x}: x\in u+\Z_{\ge0}^2\}$ defined by \eqref{Zr11a}--\eqref{Zr12a} is then exactly the same as the restriction  $\{L^\alpha_x: x\in u+\Z_{\ge0}^2\}$   of $L^\alpha$.  To verify this rewrite  \eqref{Zr12a}  as follows for $x$ in the bulk $u+\Z_{>0}^2$:
\begin{align*}
Z^\alpha_{u,\,x}&= \sum_{k=1}^{x_1-u_1}    L^\alpha_{u+k\evec_1}   Z_{u+k\evec_1+\evec_2, \,x}  
	+ 
	\sum_{\ell=1}^{x_2-u_2}   L^\alpha_{u+\ell \evec_2}   Z_{u+\evec_1+\ell \evec_2, \,x} \\
	&= \sum_{k=1}^{x_1-u_1}   \biggl( \; \sum_{j: j\le 0}  L^\alpha_{u+j\evec_2}   \tsp Z_{u+\evec_1+j\evec_2,\tsp u+k\evec_1}  \biggr) \;  Z_{u+k\evec_1+\evec_2, \,x}  
	+  \sum_{\ell=1}^{x_2-u_2}   L^\alpha_{u+\ell \evec_2}   Z_{u+\evec_1+\ell \evec_2, \,x} \\
	&= \sum_{j\le 0}      L^\alpha_{u+j\evec_2}   \sum_{k=1}^{x_1-u_1}  Z_{u+\evec_1+j\evec_2,\tsp u+k\evec_1}   \tspa Z_{u+k\evec_1+\evec_2, \,x} 
	+  \sum_{\ell=1}^{x_2-u_2}   L^\alpha_{u+\ell \evec_2}   Z_{u+\evec_1+\ell \evec_2, \,x} \\
	&=   \sum_{\ell\le x_2-u_2}   L^\alpha_{u+\ell \evec_2}   Z_{u+\evec_1+\ell \evec_2, \,x}
	 \; =\;  L^\alpha_x. 
\end{align*}
Invariance \eqref{Zr17} comes from \eqref{L55}.   The statement in part (i) about independence comes from \eqref{L47}. 
 The first statement of part (ii) of the theorem comes from \eqref{L57} and  the second statement from   \eqref{L59}. 
 
 As the last step we prove part (iii).   The inequality $J^{\rho}_{u+j\evec_2} \le  J^{\lambda}_{u+j\evec_2}$ comes directly from \eqref{as}, due to the construction $(\Jvec^\rho, \Jvec^\lambda)=(\Yvec^\rho, \Dop(\Yvec^\lambda, \Yvec^\rho))$.   Then \eqref{tt} gives the inequality $I^{\lambda}_{u+i\evec_1} \le I^{\rho}_{u+i\evec_1}$ because, in terms of the notation used above, the sequence  $\Ivec^{\alpha,k}= \{I^\alpha_{u+k\evec_1+j\evec_2}\}_{j\in\Z}$  satisfies $\Ivec^{\alpha,k}=\Sop( \Jvec^{\alpha,k-1}, Y^k)$.    
 
 Let $F_\alpha(x)$ be the c.d.f.\ of the ${\rm Ga}^{-1}(\alpha)$ distribution.  It is continuous and strictly increasing in $x\in(0,\infty)$ and strictly increasing in $\alpha$.    Thus $F_{\sigma-\lambda}(I^{\lambda}_{u+i\evec_1})\sim{\rm Unif}(0,1)$, and we define 
 $\eta_{u+i\evec_1}=F_\sigma^{-1}(F_{\sigma-\lambda}(I^{\lambda}_{u+i\evec_1}))\sim{\rm Ga}^{-1}(\sigma)$. 
 $F_{\sigma-\lambda}(I^{\lambda}_{u+i\evec_1})<F_\sigma(I^{\lambda}_{u+i\evec_1})$ implies  $\eta_{u+i\evec_1}<I^{\lambda}_{u+i\evec_1}$ because $F_\sigma^{-1}$ is also strictly increasing.
 
Define  analogously  $\eta_{u+j\evec_2}=F_\sigma^{-1}(F_{\rho}(J^{\rho}_{u+j\evec_2}))$.  
 \end{proof}

\subsection{Wandering exponent}\label{sec:kpz5}

 We quote from \cite{sepp-12-aop-corr}  bounds on the fluctuations  of the inverse-gamma polymer path.  
 The results below are proved in \cite{sepp-12-aop-corr} with  couplings and calculations with the ratio-stationary polymer process, without recourse to the integrable probability features of the inverse-gamma polymer.  
 
 Let  the bulk weights $(Y_x)_{x\tspa\in\tspa\Z^2}$ be i.i.d.\ $\text{\rm Ga}^{-1}(1)$ distributed. 
 Recall the definition of the averaged path distribution $P_{\zevec, v}$ from \eqref{h:P}. 
  On large scales  the $P_{\zevec, v}$-distributed random path $X_\bbullet\in\pathsp_{\zevec, v}$ follows the straight line segment $[\zevec, v]$ between its endpoints.  Typical deviations from the line segment obey the Kardar-Parisi-Zhang (KPZ) exponent $2/3$.  The result below  gives a quantified upper bound.  It is used in the proof of Lemma \ref{lm:far}.

 Given the endpoints $\zevec=(0,0)$ and $v=(v_1,v_2)>\zevec$ on $\Z^2$ and $0<h<1$, let 
 \[
 I_{v,h,b} = [h v-bN^{2/3}\evec_2, \, h v+bN^{2/3}\evec_2] 
 \]
 be the vertical line segment of length $2bN^{2/3}$ centered at $h v$.

\begin{theorem}\label{thm:kpz5} {\rm\cite[Theorem 2.5]{sepp-12-aop-corr}}  Let $0<s,t,\kappa<\infty$ and $0<h<1$.  Then there exist finite  $(s,t, \kappa, h)$-dependent  constants $N_0$, $b_0$ and $C$  such that, whenever $N\ge N_0$,  $v\in\Z_{>0}^2$ satisfies 
\be \abs{v-(Ns,Nt) }_1\le  \kappa N^{2/3} 
 \label{mnNgamma}  \ee
and   $b\ge b_0$,   we have 
\be\label{kpz8}
P_{\zevec,v }\bigl\{  X_\bbullet \cap  I_{v,h,b} = \emptyset  \bigr\} 
 \le Cb^{-3}. 
 \ee
 The parameter vector  $(N_0, b_0,C)$ is   bounded if $(s,t, \kappa, h)$ is restricted to a compact subset of $\R_{>0}^3\times(0,1)$. 
\end{theorem}

We also state a KPZ bound on the exit point of the stationary polymer used in the proof of Lemma \ref{lem-lb1}.  Take a  parameter $\rho\in(0,1)$ with characteristic direction $\xi(\rho)$ of \eqref{XtR}.  Consider the ratio-stationary inverse-gamma polymer with quenched path measure $Q^\rho_{\zevec,v}$ and annealed measure $P^\rho_{\zevec,v}(\cdot)=\E[Q^\rho_{\zevec,v}(\cdot)]$, as developed in Section \ref{sec:invga}. 

\begin{theorem}\label{thm:kpz3}  
Let $\kappa\in(0,\infty)$. 
There exist finite  $(\rho, \kappa)$-dependent  constants $N_0$, $b_0$ and $C$  such that, whenever $N\ge N_0$,  $v\in\Z_{>0}^2$ satisfies 
\be \abs{v- N\xi(\rho)  }_1\le  \kappa N^{2/3} 
 \label{m:980}  \ee
and   $b\ge b_0$,  we have 
\be\label{kpz10}
P_{\zevec,v }^\rho\bigl\{  \ex_{\zevec, v} \ge b N^{2/3}   \bigr\} 
 \le Cb^{-3}. 
 \ee
 The parameter vector  $(N_0, b_0,C)$ is   bounded if $(\rho, \kappa)$ is restricted to a compact subset of $(0,1)\times\R_{>0}$. 
\end{theorem}

\section{Bound on the running maximum of a random walk}\label{sec:rw}	
In this appendix we quote a random walk estimate from \cite{busa-sepp-rw}, used in the proof of Lemma \ref{lm:close}.  
	For $\alpha, \beta>0$  let  $S^{\alpha,\beta}_m=\sum_{i=1}^m X^{\alpha,\beta}_i$  denote the random walk with   i.i.d.\ steps $\{X^{\alpha,\beta}_i\}_{i\geq 1}$ specified by 
	\[	X^{\alpha,\beta}_1 \; \deq \; \log G^{\alpha}-\log G^{\beta} \]
	with two  independent gamma variables $G^{\alpha}\sim{\rm Ga}(\alpha)$   and $G^{\beta}\sim{\rm Ga}(\beta)$ on the right.   Denote the mean step by  
$\mu_{\alpha, \beta}=\E(X^{\alpha,\beta}_1)=\psi_0(\alpha)-\psi_0(\beta)$.

Fix a compact interval $[\rho_{\rm min}, \rho_{\rm max}]\subset(0,\infty)$.  Fix a positive constant $a_0$ and let $\{s_N\}_{N\ge 1}$ be a sequence of nonnegative reals such that $0\le s_N\le a_0(\log N)^{-3}$. 
Define a set of admissible pairs 
\[  \cS_N=\{(\alpha, \beta):  \alpha, \beta\in [\rho_{\rm min}, \rho_{\rm max}], \; 
-s_N\le \alpha-\beta \le 0\}. 
\] 
 The point of the theorem below is that for $(\alpha, \beta)\in\cS_N$ the walk $\{S^{\alpha, \beta}_m\}_{1\le m\le N}$ has a small enough negative drift that we can establish a positive lower bound for its running maximum. 
 	
 	\begin{theorem}{\rm\cite[Corollary 2.8]{busa-sepp-rw}}\label{thm:lm2}  In the setting described above   the bound below holds for all  $N\ge N_0$,  $(\alpha, \beta)\in\cS_N$,  and $x\ge(\log N)^2$: 
		\begin{align*}
		\mathbb{P}\big\{\tspa\max_{1\leq m \leq N}S^{\alpha,\beta}_m\le  x\big\}\leq C\tsp x \tsp (\log N)(\mu_{\alpha, \beta} \vee N^{-1/2}\tspb).
		\end{align*}
The constants $C$ and $N_0$ depend on $a_0$, $\rho_{\rm min}$, and  $\rho_{\rm max}$.  		
	\end{theorem}

	\medskip 
	
	\small
	\bibliography{Timo_old_bib}
	\bibliographystyle{plain}
\end{document}